\documentclass[3p]{elsarticle} %³£ÓÃ

\usepackage{exscale}
\usepackage{enumitem}
\usepackage{relsize}
\usepackage{mathrsfs}
\usepackage{color}
\usepackage{subfig}
\usepackage{lineno,hyperref}
\modulolinenumbers[5]
\usepackage{subfig}
\usepackage{amsmath,amssymb,cases,amsthm}

\newtheorem{Definition}{Definition}[section]
\newtheorem{Theorem}{Theorem}[section]
\newtheorem{Lemma}{Lemma}[section]

\newtheorem{Remark}{Remark}[section]

\begin{document}

\begin{frontmatter}

\title{A Separation Principle for  Conditional Mean-Field Type Linear Quadratic Optimal Control Problem}
\tnotetext[mytitlenote]{This work is supported by the National Key R\&D Program of China under Grant No. 2022YFA1006100, the NSFC under Grant No. 61925306 and the NSF of Shandong Province under Grant No. ZR2020ZD24.}

\author[]{Zhongbin Guo}
\ead{Guozhongbin0628@163.com}

\author[]{Guangchen Wang\corref{cor1}}
\ead{wguangchen@sdu.edu.cn}

\address{School of Control Science and Engineering, Shandong University, Jinan 250061,   China}
\cortext[cor1]{Corresponding author.}

%\author{Elsevier\fnref{nomyfootnote}}

%\author[mymainaddress,mysecondaryaddress]{Pengyan Huang}
%\author[mysecondaryaddress]{Guangchen Wang\corref{mycorrespondingauthor}}

\begin{abstract}
This paper investigates a conditional mean-field type linear  quadratic (LQ) optimal control problem with partial observation and regime switching, where the conditional expectations of the  state and control given the history of Markov chain enter into the dynamics and cost. The exact regime of Markov chain is accessible, whereas the system state can only be partially observed.  A separation principle is established, showing that the estimate and control  procedures  can be separated and   implemented independently. It extends the classical separation principle to conditional mean-field system. Utilizing two sets of Riccati equations and a set of first-order ordinary differential equations, we derive the feedback representation of the  optimal control. To illustrate the effectiveness of the theoretical results,  two applications with numerical simulations are provided,  including a vector-valued  LQ example and a coupled electrical machines control problem.
\end{abstract}


\end{frontmatter}

\section{Introduction}
Let $(\Omega,\mathcal{F},\mathbb{F},\mathbb{P})$ be a complete filtered probability space supporting a   Gaussian random vector $\xi$, two standard Brownian motions $W=\{W_t\}_{t\in[0,T]}
$ and $\bar{W}=\{\bar{W}_t\}_{t\in[0,T]}
$, and a continuous-time Markov chain $\theta=\{\theta_t\}_{t\in[0,T]}
$ with initial regime $\theta_0=\ell_0$. Define $\mathcal{F}^{\theta}_t=\sigma\{\theta_s;s\in[0,t]\}$. Let $\mathbb{E}^{\theta}_t[\cdot]=\mathbb{E}[\cdot|\mathcal{F}^{\theta}_t]$, which denotes the conditional expectation given the history of Markov chain  $\theta$ up to  time $t$.  Consider the  linear stochastic differential equation (SDE)
\begin{equation}
	\label{equ:state}
	\left\{\begin{aligned}
		\mathrm{d}X^v_t=&\ \left[A_{(\theta_t,t)}X^v_t+\widehat{A}_{(\theta_t,t)}\mathbb{E}^{\theta}_t[X^v_t]+B_{(\theta_t,t)}v_t+\widehat{B}_{(\theta_t,t)}\mathbb{E}^{\theta}_t[v_t]+b_{(\theta_t,t)}\right]\mathrm{d}t +C_{(\theta_t,t)}\mathrm{d}W_t+\bar{C}_{(\theta_t,t)}\mathrm{d}\bar{W}_t,\\
		X^v_0=&\ \xi,
	\end{aligned}\right.
\end{equation}where $X^v$ is the $n$-dimensional state process and $v$ is the $m$-dimensional control process.

One feature of (\ref{equ:state})  is that    continuous dynamics $X^{v}$ and discrete events $\theta$ coexist. The two-component process $(X^{v},\theta)$ is commonly referred to as a switching diffusion or  regime-switching diffusion. In many practical situations, models based solely on continuous solutions of  SDEs with deterministic coefficients are insufficient to capture real phenomena. For instance, financial markets may randomly switch between   ``bullish'' and  ``bearish''  modes.  The market parameters such as return rate  and volatility rate of  stock price differ substantially across these modes \cite{Zhou 2003}. Similarly, in  solar electrical generating system consisting of movable mirrors (heliostats), the  flow rates and pressures  parameters of  steam in pipeline show huge differences  in  clear sky day  and dense cloud one \cite{Sworder 1983}.    The common  feature of these phenomena is that the internal coefficients of   dynamics  change in response  to   external   environments that evolve  in a discrete manner.  Thanks to the  ability to effectively describe such phenomena,    switching diffusion  has been widely applied to the modeling, analysis, and optimization of systems subject to abrupt changes in random environments, see monograph \cite{Yin hybrid 2009} for details.

Another distinctive feature of \eqref{equ:state} is the incorporation of the conditional expectations of  state and  control given the history of the Markov chain i.e.,  $\mathbb{E}^{\theta}_t[X^v_t]$ and $\mathbb{E}^{\theta}_t[v_t]$. Mean-field type SDEs, in which expectations appear  in the state equation, have been extensively studied in recent years, see \cite{Anderson 2011,Li 2012,Sun 2017,Wang 2022,Yong 2013,Zhang general 2018} and references therein. By contrast, conditional mean-field type SDEs such as (\ref{equ:state})  receive relatively little attention. The introduction of the conditional expectations is of substantial significance.  Theoretically,   it is closely related to mean-field model  with  Markovian switching. In the new era, the increasing structural complexity of systems calls for effective methods to address decision-making problems in large-scale systems,  where the major challenge is to reduce computational complexity and avoid the curse of dimensionality. The mean-field framework offers a classical approach: it replaces the influence of the entire system on an individual particle (agent or decision-maker) with the empirical probability distribution of the population, thereby enabling an effective description of interactions among a large number of particles. A key issue is to analyze the asymptotic behavior of mean-field models as the number of particles tends to infinity, since such results help simplify the resolution of large-scale decision-making problems, see \cite{Dawson 1983,Sznitman 1989}  and references therein. More recently, \cite{Nguyen George Yin Hoang 2020} studied the law of large numbers for a mean-field model with a common Markov chain. Their results show that, under the influence of a shared switching process, the particles are no longer independent, and the limiting empirical measure is not deterministic but a random measure conditional on the Markov chain's history. This observation motivates the setup adopted in this paper and justifies the use of conditional expectations $\mathbb{E}^{\theta}_t[\cdot]$ in place of unconditional expectations $\mathbb{E}[\cdot]$.

Motivated by \cite{Nguyen George Yin Hoang 2020}, the study of optimization problems involving conditional expectations given the Markov chain in both the dynamics and the cost functional has attracted growing attention in recent years. Interested readers may refer to \cite{Nguyen 2020,Nguyen general 2021} for maximum principles, \cite{Lv Xiong Zhang 2023} for Stackelberg games, and \cite{Barreiro-Gomez linear 2019,Lv linear 2023,Rolón Gutiérrez markovian 2024} for nonzero-sum differential games. It is worth noting, however, that these works are conducted under the assumption that the state is perfectly observable. In practice, the exact value of the state variable is often inaccessible. For instance, the appreciate return rate process of a stock is unobservable and can only be partially inferred from stock prices. Consequently,  investors should  make decisions on the basis of stock prices information, see, e.g.,  \cite{Xiong mean 2007}.

In this paper, we assume that the state process $X^{v}$ is not directly observable and can only be measured through a noise-corrupted observation process $Y^{v}$ ($r$-dimensional), which satisfies
\begin{equation}
	\label{equ:obser}
	\left\{\begin{aligned}
		\mathrm{d}Y^v_t=&\ \left[F_{(\theta_t,t)}X^v_t+\widehat{F}_{(\theta_t,t)}\mathbb{E}^{\theta}_t[X^v_t]+f_{(\theta_t,t)}\right]\mathrm{d}t+\mathrm{d}W_t,\\
		Y^v_0=&\ 0.
	\end{aligned}\right.
\end{equation}
While state $X^{v}$ is unobservable,  Markov chain $\theta$ is assumed to be perfectly observed. Examples of such settings, where the state variable is partially observed  but  the Markov chain is observable, can be found in \cite{Costa discrete 1995}. Naturally, the design of control input should rely on the information provided  by  the observation process $Y^{v}$ and the switching process $\theta$. In other words, the control should be adapted to the filtration $\mathbb{F}^{Y^v,\theta}$, generated jointly by $Y^{v}$ and $\theta$.  Although the switching process $\theta$  is not affected by $v$, the observation process $Y^v$ depends on $v$ via $X^v$ indirectly. Then there exists a circular (coupling) dependence between the control process $v$ and   $\mathbb{F}^{Y^v,\theta}$, which is similar to the ``chicken and egg matter''.  In this case, if admissible controls were simply defined as 
$\mathbb{F}^{Y^v,\theta}$-adapted processes, the resulting admissible control set would no longer be convex.  Actually,  it is not even easy to determine whether a control process $v$ is admissible or not. This immediately leads to a difficulty that the classical variation method is invalid, rendering the computation of optimal control hopeless, see \cite{Bensoussan 1992} for further details.

The cost functional considered in this paper is given by
\begin{equation}
	\label{equ:ori cost}
	\begin{aligned}
		 J^{P}\left(\xi,\ell_0; v\right)
		= &\  \frac{1}{2} \mathbb{E}\left\{\int _ { 0 } ^ { T } \left[\left\langle\left[\begin{array}{l}
			X^v_t \\
			v_t
		\end{array}\right],\left[\begin{array}{cc}
			H_{(\theta_t,t)} & G^{\top}_{(\theta_t,t)} \\
			G_{(\theta_t,t)} & R_{(\theta_t,t)}
		\end{array}\right]\left[\begin{array}{l}
			X^v_t \\
			v_t
		\end{array}\right]\right\rangle\right.\right.\\
		&\ \left.\left.+\left\langle\left[\begin{array}{l}
			\mathbb{E}^{\theta}_t[X^v_t] \\
			\mathbb{E}^{\theta}_t[v_t]
		\end{array}\right],\left[\begin{array}{cc}
			\widehat{H}_{(\theta_t,t)} & \widehat{G}^{\top}_{(\theta_t,t)} \\
			\widehat{G}_{(\theta_t,t)} & \widehat{R}_{(\theta_t,t)}
		\end{array}\right]\left[\begin{array}{c}
			\mathbb{E}^{\theta}_t[X^v_t] \\
			\mathbb{E}^{\theta}_t[v_t]
		\end{array}\right]\right\rangle \right]\mathrm{d} t\right.\\
		&\ \left. +\langle X^v_T, S_{\theta_T} X^v_T\rangle+\langle\mathbb{E}^{\theta}_T[X^v_T], \widehat{S}_{\theta_T} \mathbb{E}^{\theta}_T[X^v_T]\rangle\right\}.
	\end{aligned}
\end{equation}This functional explicitly incorporates quadratic terms of conditional expectations in both the running and terminal costs. Such a formulation is motivated by a concern of (weighted) variance reduction in cost  (for example,  the problem in Section \ref{sec:coupled ele machine} in this paper and that considered in \cite{Barreiro-Gomez linear 2019}) and conditional mean-variance problem \cite{Nguyen general  2021}. It is worth noting that both the system dynamics and the cost functional depend on the conditional distributions of the state and control. Therefore,    the problem under consideration is a  conditional mean-field type linear quadratic (LQ) optimal  control problem with partial observation and regime switching, which can be stated as follows.\\
\noindent
\textbf{\textbf{(PO-CMF)}:} Find an admissible control $u^{*}\in\mathcal{U}_{ad}$ such that  $J^{P}\left(\xi,\ell_0; u^{*}\right)=\inf_{v\in\mathcal{U}_{ad}}J^{P}\left(\xi,\ell_0; v\right)$ holds subject to (\ref{equ:state})-(\ref{equ:obser}),  where the precise definition of the admissible controls set $\mathcal{U}_{ad}$ is provided in Definition \ref{def:admis cont}.

Separation principle is an effective tool for addressing LQ optimal control problems under partial observation. It asserts that, in order to construct the optimal control, one may first estimate the unobservable state, which involves obtaining the optimal filter for the true state based on the available information, and then substitute this estimate for the true state into the optimal control law derived under complete observation case. In this way, the original problem is reduced to two sub-problems of filtering and control that can be solved independently. However, the category of  systems for which the separation principle works is rather limited (see, e.g., \cite[Section 2]{Bensoussan 1992}, \cite{Wonham 1968}). Without conditional mean-field terms,  separation principle for partially observed regime-switching diffusions has been studied in \cite{Dufour adaptive 1998,Ji 1992} and \cite{Mariton1990}. In particular, \cite{Ji 1992} established a separation principle for the LQ problem driven by switching diffusions by virtue of  Bellman's dynamic programming approach. Later, \cite{Dufour adaptive 1998} extended the results of \cite{Ji 1992} to systems with non-Gaussian initial conditions and stochastic integrals with respect to counting processes, and proved that the separation principle still holds. Due to the fact that  cost functional (\ref{equ:ori cost}) involves nonlinear functions of conditional expectations, Bellman's optimality principle no longer applies (see, \cite{Bjork  Tomas Agatha 2017}). Consequently, the method in \cite{Ji 1992} cannot be employed to address our problem. Moreover, the  filtering results in \cite{Dufour adaptive 1998} and \cite{Ji 1992} cannot be directly utilized  due to the presence of conditional expectations in the  dynamics we consider. In \cite{Mariton1990}, although the entire path information  $\{Y^v_s, s\in[0,t]\}$ and $\{\theta_s, s\in[0,t]\}$ is available to the decision-maker at time $t$, the filtering equation is derived only based on the instantaneous regime $\theta_t$ of the Markov chain together with the observation history $\{Y^v_s, s\in[0,t]\}$. As pointed out in \cite{Ji 1992}, the state estimate proposed in \cite{Mariton1990} can therefore be regarded as a suboptimal filter. In contrast, our paper derives a closed-form filtering system utilizing all the available information $\{ Y^v_s,\theta_s, s\in[0,t]\}$. Consequently, even without incorporating conditional mean-field terms, the filtering and separation principle results presented here are essentially different from those in \cite{Mariton1990}.

The main contribution of our paper is that we establish a  separation principle  for  \textbf{(PO-CMF)}, enlarging the categories of the  systems for which   this principle holds.     The   establishment of separation principle follows from four main steps.    First,  we give a modified definition of admissible control,  which removes the coupled dependence between control and observation.  Second, we  derive  a  closed system for optimal  filtering of $X^v$  given $\mathbb{F}^{Y^v,\theta}$. Third, utilizing the optimal filtering result,  \textbf{(PO-CMF)} is transformed into an equivalent problem  (see \textbf{(FO-CMF)} in Section~\ref{sub:equi prob}) with full observation. The  optimal control of  \textbf{(FO-CMF)} is obtained.  Finally,  we  derive a feedback representation for  optimal control of \textbf{(PO-CMF)} which shows that the separation principle  is valid for \textbf{(PO-CMF)}. In the above procedures,  several new features  together with challenges emerge and novel techniques are adopted,  which are listed  as follows. 

\begin{enumerate}
	\item  Employing a state decomposition technique, we split the state-observation process $(X^v,Y^v)$ into an  autonomous part  $(\eta,\lambda)$ and a controlled part $(X^{1,v},Y^{1,v})$ (see (\ref{equ:sum rela})). Based on this,  a modified definition of admissible control is introduced, under which the coupling dependence  between control and observation is eliminated. This enables the use of the variation method to derive the optimality conditions.

	\item By exploiting the conditional Gaussian property, we derive a closed system for optimal filter, which consists of explicit equations for the filtering process $\mathbb{E}^{Y^v,\theta}_t[X^v_t]$ and the conditional covariance $\Phi_t$. The presence of  switching process $\theta$ destroys the Gaussian structure, which is essential for obtaining a closed filtering system in the classical Kalman filtering framework. To address this, we prove that $\eta$ is conditionally Gaussian given $\lambda$ and $\theta$. This enables us to derive a closed system for the filtering process $\mathbb{E}^{\lambda,\theta}_t[\eta_t]$, and further for $\mathbb{E}^{Y^v,\theta}_t[X^v_t]$ based on the relationship between $\eta_t$ and $X^v_t$. Different with the classical Kalman filtering theory, where the conditional covariance is a solution of deterministic   Riccati equation,  the conditional covariance $\Phi_t$ in our paper  satisfies a  Riccati equation with Markov-modulated coefficients (see (\ref{equ:Riccati equ for error})).   As a result,  unlike other coefficients that  depend only   on the instant regime of Markov chain $\theta_t$ , $\Phi_t$ depends on the entire sample  path  $\{\theta_s,s\in[0,t]\}$.  Consequently, the filter gain $Q_t$ (see (\ref{equ:def of QX in con})) is inherently non-Markovian.

	\item Utilizing  the orthogonality between the estimation error and the optimal filter process, we transform \textbf{(PO-CMF)} into  \textbf{(FO-CMF)}, which is a problem with full observation  driven by the optimal filter process $\mathbb{E}^{Y^v,\theta}_t[X^v_t]$. \textbf{(FO-CMF)} possesses two distinct features. First,  \textbf{(FO-CMF)} is still a conditional mean-field type control problem and  time-inconsistent, where the Bellman optimality principle does not hold.   Second, although all the  coefficients in \textbf{(PO-CMF)} are Markovian, \textbf{(FO-CMF)} becomes an optimal control problem with a non-Markovian switching model due to the non-Markovian nature of the filter gain $Q_t$. For the time-inconsistency, we resort to maximum principle to seek for optimal control.  It is well known that the adjoint process plays a crucial role in characterizing the optimality conditions in maximum principle. Since both the filtering  and  control processes are $\mathbb{F}^{Y^v,\theta}$-adapted, the adjoint process, which is the  solution to a backward stochastic differential equation (BSDE),  should be adapted to $\mathbb{F}^{Y^v,\theta}$ as well. This immediately raises a barrier when constructing adjoint  process  for the reason that observation  $Y^v$ is not a Brownian motion.  To overcome this difficulty, we establish that any $\mathbb{F}^{Y^v,\theta}$-martingale admits a stochastic integral representation with respect to the innovation process and the fundamental martingale associated with the Markov chain. Building on this result, we formulate the adjoint equation  and present a necessary and sufficient condition for optimality of  \textbf{(FO-CMF)}. By employing the four-step scheme  approach to decouple the Hamiltonian system, we derive a feedback representation of the optimal control through two sets of coupled  Riccati equations and a set of coupled linear ordinary differential equations (ODEs). For the non-Markovian  nature, different from the Markovian case in \cite{Moon Basar 2024} and  \cite{Sun Xiong 2023}, where deterministic ODE is enough to compute optimal cost,  we introduce a BSDE driven by Markov chain to explicitly characterize the optimal cost.

	\item According to the equivalence between \textbf{(PO-CMF)} and \textbf{(FO-CMF)}, we obtain the optimal control for \textbf{(PO-CMF)}. The result demonstrates that the optimal control of \textbf{(PO-CMF)} can be constructed by substituting the state estimate for the true state into the feedback control law of the complete observation case. This extends the classical separation principle to linear conditional mean-field systems and enlarges  the category of systems for which the separation principle remains valid. 
\end{enumerate}

The rest of the  paper is organized as follows.  Section \ref{seq: notations and preliminaries} presents the notations and assumptions on the coefficients of the model. Section \ref{sec: separation principle} establishes the separation principle of  \textbf{(PO-CMF)}. Section \ref{sec: applications  and numerical simulations} provides two applications with numerical simulations, including a  vector-valued   LQ  example and a coupled electrical machines control problem. Section \ref{sec: conclusion} draws some conclusions.

\section{Notations and Assumptions}
\label{seq: notations and preliminaries}

For integers $n$ and $m$, we denote  by $\mathbb{R}^{n\times m}$ the  space of $n\times m$  real  matrices, by $\mathbb{S}^n$ the  space of symmetric $n\times n$  real  matrices. The symbol $\mathrm{tr}$ and the  superscript ${}^\top$ represent the trace and the transpose of a matrix, respectively. The notation $\langle\cdot,\cdot\rangle$ stands for the inner product of two vectors, and  $|\cdot|$ denotes the corresponding Euclidean norm.    If $M,N\in \mathbb{S}^{n}$, the notation $M\geq0$ means  that  $M$ is positive semi-definite, while  $M\geq N$ means  that $M-N\geq 0$. For a measurable set $\mathcal{A}$, $\mathbb{I}_{\mathcal{A}}$  denotes the indicator function of  $\mathcal{A}$. Moreover, we use $\mathcal{I}_{n\times n}$ and $0_{n\times m}$ to denote  $n\times n$  identity matrix and $n\times m$ zero matrix, respectively.

We work with a   filtered probability space $(\Omega,\mathcal{F},\mathbb{F},\mathbb{P})$ satisfying the usual conditions. Let $W=\{W_t\}_{t\in[0,T]}
$ and $\bar{W}=\{\bar{W}_t\}_{t\in[0,T]}
$ be standard Brownian motions taking values in  $\mathbb{R}^r$ and $\mathbb{R}^d$, respectively. Let   $\theta=\{\theta_t\}_{t\in[0,T]}
$ be a time homogeneous Markov chain with  finite state space  and  $\xi$ be an $\mathbb{R}^n$-valued  Gaussian random variable with expectation  $\mathbb{E}[\xi]=\mu$ and  covariance matrix $\mathbb{E}[\left(\xi-\mu\right)\left(\xi-\mu\right)^{\top}]=\sigma$.  We assume  that processes $W$, $\bar{W}$, $\theta$, and random variable $\xi$ are mutually independent. For convenience,  we set $\mathcal{F}_t=\sigma\{\xi,W_s,\bar{W}_s,\theta_s;s\in[0,t]\}$, $\mathcal{F}^{\theta}_t=\sigma\{\theta_s;s\in[0,t]\}$ for any $t\in[0,T]$ and let $\mathbb{F}=\{\mathcal{F}_t\}_{t\in[0,T]}
$ and $\mathbb{F}^{\theta}=\{\mathcal{F}^{\theta}_t\}_{t\in[0,T]}
$. The notation $\mathbb{E}^{\beta}_t[\cdot]=\mathbb{E}[\cdot|\mathcal{F}^{\beta}_t]$ denotes the conditional expectation  given the information generated by process $\beta$ up to  time $t$.    The Markov chain $\theta$ takes values in a finite state space  $\mathbb{D}=\left\{1,\dots,D\right\}$, where $D$ is the number of regimes.  Let $\Pi=\{\pi^{\ell j}\}^D_{ \ell,j=1}$ be its  transition rates matrix, where  $\pi^{\ell j}\geq 0$ is  the instantaneous intensity of a transition of  $\theta$ from state $\ell$ to $j$ $(\ell\neq j)$ and it holds that $\pi^{\ell \ell}=-\sum_{j=1,j\neq \ell}^{D}\pi^{\ell j}$ for any $\ell\in\mathbb{D}$. For $\ell,j\in\mathbb{D}$, $\ell\neq j$ and $t\in[0,T]$, define $[M^{\ell j}]_t=\sum_{s\in[0,t]}\mathbb{I}_{\{\theta_{s-}=\ell\}}\mathbb{I}_{\{\theta_s=j\}}, \langle M^{\ell j}\rangle_t=\pi^{\ell j}\int_{0}^{t}\mathbb{I}_{\{\theta_{s-}=\ell\}}\mathrm{d}s,M^{\ell j}_t=[M^{\ell j}]_t-\langle M^{\ell j}\rangle_t.$ It follows from \cite[Lemma IV.21.12]{Rogers 2000} and  \cite[Lemma B.3.16]{Donnelly 2008} that  $M^{\ell j}=\{M^{\ell j}_t\}_{t\in[0,T]}$  for $\ell,j\in\mathbb{D}$ is a purely discontinuous and square-integrable martingale with respect to $\mathbb{F}^{\theta}$, whose optional and predictable quadratic variations are $[M^{\ell j}]$ and  $\langle M^{\ell j}\rangle$, respectively. For convenience, we set $[M^{\ell \ell}]_t=\langle M^{\ell \ell}\rangle_t=M^{\ell \ell}_t=0$ for any  $\ell\in\mathbb{D}$ and $t\in[0,T]$. Moreover,  by the definition of optional quadratic covariations  (\cite[Section 1.8]{Lipster 1989}), the following orthogonality relations hold: $[\bar{W},M^{\ell j}]_t=0$, $[W,M^{\ell j}]_t=0$, and $[M^{\ell j},M^{k\iota}]_t=0$ for any $ (\ell,j)\neq(k,\iota)\in\mathbb{D}\times \mathbb{D}$ and $t\in[0,T].$

Given an arbitrary Euclidean space  $\mathbb{H}$ and a filtration $\mathbb{G}$, which may vary at different locations, we  introduce the following  spaces:
\begin{equation*}
	\begin{aligned}
		\mathcal{L}_{\mathbb{G}}^2([0,T];\mathbb{H})=&\ \{\mathcal{X}|\mathcal{X} \text{ is } \mathbb{H}\text{-valued and }\mathbb{G}\text{-progressively measurable process such that }   \mathbb{E}[\int_{0}^{T}|\mathcal{X}_t|^2\mathrm{d}t]<\infty\},\\
		\mathcal{S}_{\mathbb{G}}^2([0,T];\mathbb{H})=&\ \{\mathcal{X}|\mathcal{X} \text{ is } \mathbb{H}\text{-valued}, \mathbb{G}\text{-adapted}, \text{and c}\grave{a}\text{dl}\grave{a}\text{g process such that }  \mathbb{E}[\sup_{t\in[0,T]}|\mathcal{X}_t|^2]<\infty\}, \\
		\mathcal{L}_{\mathbb{G}}^0([0,T];\mathbb{H})=&\ \{\mathcal{X}|\mathcal{X} \text{ is } \mathbb{H}\text{-valued and } \mathbb{G}\text{-progressively measurable process}\},\\
		\mathcal{M}_{\mathbb{G}}^2([0,T];\mathbb{H})=&\ \{\mathcal{X}=\{\mathcal{X}^{\ell j}\}_{\ell,j=1}^{D}| \mathcal{X}^{\ell j}\in \mathcal{L}_{\mathbb{G}}^0([0,T];\mathbb{H}),  \sum_{\ell,j=1}^{D}\mathbb{E}[\int_{0}^{T}|\mathcal{X}^{\ell j}_t|^2\mathrm{d}[M^{\ell j}]_t]<\infty,  \text{ and }\mathcal{X}^{\ell \ell}=0 \}.
	\end{aligned}
\end{equation*}

   For  process $\mathcal{X}\in\mathcal{M}_{\mathbb{G}}^2([0,T];\mathbb{H})$ and martingale $M^{\ell j}$ with $\ell,j\in\mathbb{D}$, we define 
   \begin{equation*}
   		\begin{aligned}
   			\int_{0}^{t}\mathcal{X}_s\bullet\mathrm{d}M_s=\sum_{\ell,j=1}^{D}\int_{0}^{t}\mathcal{X}^{\ell j}_s\mathrm{d}M^{\ell j}_s \quad \text{and} \quad  \mathcal{X}_t\bullet\mathrm{d}M_t=\sum_{\ell,j=1}^{D}\mathcal{X}^{\ell j}_t\mathrm{d}M^{\ell j}_t.
   		\end{aligned}
   \end{equation*}

To simplify the  notations, for any $\ell\in\mathbb{D}$, we define  $\check{\phi}_{(\ell,t)}=\phi_{(\ell,t)}+\widehat{\phi}_{(\ell,t)}$ with $\phi=A,B,H,G,R$,  and  $\check{S}_{\ell}=S_{\ell}+\widehat{S}_{\ell}$.

Regarding the coefficients in (\ref{equ:state})-(\ref{equ:ori cost}), the following assumptions are in force throughout our paper.
\begin{enumerate}[label=$\mathbf{(A\arabic*)}$,leftmargin=3em]
	\item \label{ass:coef in dyna}For any $\ell\in\mathbb{D}$,  $A_{(\ell,t)},\widehat{A}_{(\ell,t)}\in\mathbb{R}^{n\times n}$, $B_{(\ell,t)},\widehat{B}_{(\ell,t)}\in\mathbb{R}^{n\times m}$, $C_{(\ell,t)}\in\mathbb{R}^{n\times r}$, $\bar{C}_{(\ell,t)}\in\mathbb{R}^{n\times d}$, $F_{(\ell,t)},\widehat{F}_{(\ell,t)}\in\mathbb{R}^{r\times n}$, $b_{(\ell,t)}\in\mathbb{R}^n$, and $f_{(\ell,t)}\in\mathbb{R}^{r}$ are deterministic and bounded functions  on $[0,T]$.
	\item\label{ass: coe in cost one} For any $\ell\in\mathbb{D}$,  $H_{(\ell,t)},\widehat{H}_{(\ell,t)}\in\mathbb{S}^{n}$, $G_{(\ell,t)},\widehat{G}_{(\ell,t)}\in\mathbb{R}^{m\times n}$, and  $R_{(\ell,t)},\widehat{R}_{(\ell,t)}\in\mathbb{S}^{m}$  are deterministic and bounded functions  on $[0,T]$, $S_{\ell},\widehat{S}_{\ell}\in\mathbb{S}^n$ are deterministic and bounded matrices. 
	\item\label{ass: coe in cost two} For any $\ell\in\mathbb{D}$ and $t\in[0,T]$, there exists a constant $\delta$ such that $R_{(\ell,t)},\check{R}_{(\ell,t)}\geq \delta \mathcal{I}_{m\times m}$, $H_{(\ell,t)}-G^{\top}_{(\ell,t)}R^{-1}_{(\ell,t)}G_{(\ell,t)},\check{H}_{(\ell,t)}-\check{G}^{\top}_{(\ell,t)}\check{R}^{-1}_{(\ell,t)}\check{G}_{(\ell,t)}\geq 0$, and $S_{\ell},\check{S}_{\ell}\geq 0$.
\end{enumerate}

%\begin{Remark}
%	 We set the intensity coefficient before $\mathrm{d}W_t$ in \eqref{equ:obser} to be the identity matrix just for the sake of notational simplicity. In fact, all arguments in this paper remain valid as long as the intensity coefficient satisfies the nondegeneracy assumption.
%\end{Remark}

Note that all the coefficients in dynamics and cost are time-varying and modulated by  Markov chain. However, for notational simplicity, parameter $(\theta_t,t)$ in the  coefficients and time variable $t$ in the processes will be frequently omitted hereafter whenever no confusion arises.

\section{Separation Principle for \textbf{(PO-CMF)}}
\label{sec: separation principle}
\subsection{A Modified Definition of Admissible Controls}

In this section,  we adopt a state decomposition technique to  provide a modified definition of admissible control, which eliminates the coupled dependence between control and observation.

Observing the  linear structure of state-observation dynamics  (\ref{equ:state})-(\ref{equ:obser}), we introduce   processes $\eta=\{\eta_t\}_{t\in[0,T]}
$ and $\lambda=\{\lambda_t\}_{t\in[0,T]}
$ by
\begin{equation}
	\label{equ:eta}
	\left\{\begin{aligned}
		\mathrm{d}\eta_t=&\ \left[A_{(\theta_t,t)}\eta_t+\widehat{A}_{(\theta_t,t)}\mathbb{E}^{\theta}_t[\eta_t]+b_{(\theta_t,t)}\right]\mathrm{d}t +C_{(\theta_t,t)}\mathrm{d}W_t+\bar{C}_{(\theta_t,t)}\mathrm{d}\bar{W}_t,\\
		\eta_0=&\ \xi,
	\end{aligned}\right.
\end{equation}and
\begin{equation}
	\label{equ:lambda}
	\left\{\begin{aligned}
		\mathrm{d}\lambda_t=&\ \left[F_{(\theta_t,t)}\eta_t+\widehat{F}_{(\theta_t,t)}\mathbb{E}^{\theta}_t[\eta_t]+f_{(\theta_t,t)}\right]\mathrm{d}t+\mathrm{d}W_t,\\
		\lambda_0=&\ 0.
	\end{aligned}\right.
\end{equation}Systems (\ref{equ:eta})-(\ref{equ:lambda}) can be regarded as   autonomous state-observation dynamics,  where the influence of  control is taken away.  In other words, they  separate the uncertainties from control process $v$ compared with (\ref{equ:state})-(\ref{equ:obser}).

Then for any  square integrable (with respect to  $\omega$ and $t$) process $v$ without imposing any adaptation requirement, we introduce $X^{1,v}=\{X^{1,v}_t\}_{t\in[0,T]}$ and $Y^{1,v}=\{Y^{1,v}_t\}_{t\in[0,T]}$   by
\begin{equation}
	\label{equ:x1}
	\left\{\begin{aligned}
		\mathrm{d}X^{1,v}_t=&\ \left[A_{(\theta_t,t)}X^{1,v}_t+\widehat{A}_{(\theta_t,t)}\mathbb{E}^{\theta}_t[X^{1,v}_t] +B_{(\theta_t,t)}v_t+\widehat{B}_{(\theta_t,t)}\mathbb{E}^{\theta}_t[v_t]\right]\mathrm{d}t,\\
		X^{1,v}_0=&\ 0,
	\end{aligned}\right.
\end{equation}and
\begin{equation}
	\label{equ:Y1}
	\left\{\begin{aligned}
		\mathrm{d}Y^{1,v}_t=&\ \left[F_{(\theta_t,t)}X^{1,v}_t+\widehat{F}_{(\theta_t,t)}\mathbb{E}^{\theta}_t[X^{1,v}_t]\right]\mathrm{d}t,\\
		Y^{1,v}_0=&\ 0.
	\end{aligned}\right.
\end{equation}

From Lemma 2.1 in \cite{Nguyen 2020}, we can derive that, under assumption \ref{ass:coef in dyna}, equations (\ref{equ:eta})-(\ref{equ:Y1}) admit unique solutions $(\eta,\lambda)$ and $(X^{1,v},Y^{1,v})$. Moreover, it is straightforward to verify that $\eta+X^{1,v}$  and $\lambda+Y^{1,v}$ solve (\ref{equ:state}) and (\ref{equ:obser}), respectively.  Therefore, we obtain the  following decomposition for solutions of (\ref{equ:state}) and (\ref{equ:obser}):
\begin{equation}
	\label{equ:sum rela}
	\begin{aligned}
		X^v=\eta+X^{1,v}, \hspace{0.5cm} \text{ and } \hspace{0.5cm}Y^v=\lambda+Y^{1,v}.
	\end{aligned}
\end{equation}

Based on the  process  $\lambda$, we  introduce a  filtration $\mathbb{F}^{\lambda,\theta}=\{\mathcal{F}^{\lambda,\theta}_t\}_{t\in[0,T]}$,  where $\mathcal{F}^{\lambda,\theta}_t=\sigma\{\lambda_s,\theta_s;s\in[0,t]\}$,  and a set $\mathcal{U}=\{v|v \text{ is } \mathbb{R}^{m}\text{-valued and }  \mathbb{F}^{\lambda,\theta}\text{-adapted}$ $\text{process such that } \mathbb{E}[\sup_{t\in[0,T]}|v_t|^2]<\infty\}$. Then we present the  modified definition of  admissible control.

\begin{Definition}
	\label{def:admis cont}
	A process $v$ is referred to as an admissible control if $v\in\mathcal{U}$ and $v$ is $\mathbb{F}^{Y^v,\theta}$-adapted. Moreover, the  set  containing all the admissible controls is denoted as  $\mathcal{U}_{ad}$.
\end{Definition}

We have the following two results, which play an important role in the  subsequent treatment.

\begin{Lemma}\label{lem:filtr ident}
	For any admissible control $v\in\mathcal{U}_{ad}$, the filtration $\mathbb{F}^{Y^v,\theta}$  coincides with  $\mathbb{F}^{\lambda,\theta}$, which is independent of control $v$.
\end{Lemma} 
\begin{proof}
	On the one hand, since  $v$ is adapted to $\mathbb{F}^{\lambda,\theta}$,  it follows from (\ref{equ:x1}) and  (\ref{equ:Y1}) that $Y^{1,v}$ is $\mathbb{F}^{\lambda,\theta}$-adapted. Using the identity $Y^v=\lambda+Y^{1,v}$, we obtain  $\mathbb{F}^{Y^v}\subseteq\mathbb{F}^{\lambda,\theta}$, which further yields $\mathbb{F}^{Y^v,\theta}\subseteq\mathbb{F}^{\lambda,\theta}$. On the other hand, since   $v$ is $\mathbb{F}^{Y^v,\theta}$-adapted, it again  yields from (\ref{equ:x1}) and  (\ref{equ:Y1})  that $Y^{1,v}$ is adapted to $\mathbb{F}^{Y^v,\theta}$. In view of the relation  $\lambda=Y^v-Y^{1,v}$, we deduce $\mathbb{F}^{\lambda}\subseteq\mathbb{F}^{Y^v,\theta}$ and hence  $\mathbb{F}^{\lambda,\theta}\subseteq\mathbb{F}^{Y^v,\theta}$.  Summarizing the above analysis, we conclude  that  $\mathbb{F}^{Y^v,\theta}=\mathbb{F}^{\lambda,\theta}$.
\end{proof}
\begin{Lemma}
	\label{lem:equi bet con on two sets}
	Under assumptions \ref{ass:coef in dyna} and \ref{ass: coe in cost one}, it holds that $	\inf_{v\in\mathcal{U}_{ad}}J^{P}\left(\xi,\ell_0; v\right)=\inf_{v\in\mathcal{U}}J^{P}\left(\xi,\ell_0; v\right).$
\end{Lemma}

\begin{proof}
	To keep a better flow of presentation, the proof of Lemma \ref{lem:equi bet con on two sets} is relegated to  Appendix \ref{sec: pro appe}.
\end{proof}
%\begin{Remark}
%	It should be emphasized that the proof of  Lemma \ref{lem:equi bet con on two sets} relies crucially on the complete observation of the Markov chain $\theta$.  In the absence of such information, the  argument would generally fail.
%\end{Remark}

%Lemma \ref{lem:filtr ident} shows that, under any admissible control $v\in\mathcal{U}_{ad}$, the available information $\mathbb{F}^{Y^v,\theta}$ is given a priori and does not depend on the choice of $v$. Hence, the circular dependence between the control and the observation is eliminated, thereby rendering the classical variation method applicable. Consequently, we shall omit the superscript $v$ in $Y$ and use $\mathbb{F}^{Y,\theta}$ to denote the available information  hereinafter.

\subsection{Optimal Filter}

In this part, we first present two auxiliary results (Lemmas \ref{lem:N BM} and \ref{lem:filt of eta}) and then provide  the optimal filter for $X^{v}$ given $\mathbb{F}^{Y^v,\theta}$.

\begin{Lemma}\label{lem:N BM}
	Process $N=\{N_t\}_{t\in[0,T]}
	$ defined by
	\begin{equation}
		\label{equ:def N}
		\begin{aligned}
			N_t=\lambda_t-\int_{0}^{t}\left[F_{(\theta_s,s)}\mathbb{E}^{\lambda,\theta}_s[\eta_s]+\widehat{F}_{(\theta_s,s)}\mathbb{E}^{\theta}_s[\eta_s]+f_{(\theta_s,s)}\right]\mathrm{d}s
		\end{aligned}
	\end{equation}satisfies the following properties: 
	\begin{enumerate}[label=(\roman*)]
		\item \label{con:brow mot} $N$ is an $\mathbb{F}^{\lambda,\theta}$-adapted standard Brownian motion  taking values in $\mathbb{R}^{r}$;
		\item  \label{con:condi expe zero}
		Suppose that $\mathcal{X}$ is an $\mathbb{F}^{\theta}$-adapted and square-integrable process with appropriate dimension, then $	\mathbb{E}^{\theta}_t[\int_{0}^{t}\mathcal{X}_s\mathrm{d}N_s]=0$ for any $t\in[0,T]$.
	\end{enumerate}
\end{Lemma}

\begin{proof}
	\noindent \ref{con:brow mot}: In view of (\ref{equ:lambda}),   $N_t$ can be rewritten  as 
	\begin{equation}
		\label{equ:N SDE}
		N_t=\int_{0}^{t}F\left(\eta_s-\mathbb{E}^{\lambda,\theta}_s[\eta_s]\right)\mathrm{d}s+W_t.
	\end{equation}For any $\tau\in[0,t)$, by  the tower  property of conditional expectation (\cite[Section 9.7]{Williams 1991}), it follows that
	\begin{equation}
		\label{equ:integ condi aero}
		\begin{aligned}
			\mathbb{E}^{\lambda,\theta}_{\tau}\left[\int_{\tau}^{t}F\left(\eta_s-\mathbb{E}^{\lambda,\theta}_s[\eta_s]\right)\mathrm{d}s\right]
			= \int_{\tau}^{t}\mathbb{E}^{\lambda,\theta}_{\tau}\left[F\mathbb{E}^{\lambda,\theta}_{s}\left[\eta_s-\mathbb{E}^{\lambda,\theta}_s[\eta_s]\right]\right]\mathrm{d}s=0,
		\end{aligned}
	\end{equation}and, in conjunction with the martingale property of $W$,
	\begin{equation}
		\label{equ:Wt - Wtau equ zero}
		\begin{aligned}
			\mathbb{E}^{\lambda,\theta}_{\tau}\left[W_t-W_{\tau}\right]=\mathbb{E}^{\lambda,\theta}_{\tau}\left[\mathbb{E}^{\xi,W,\bar{W},\theta}_{\tau}\left[W_t-W_{\tau}\right]\right]=0.
		\end{aligned}
	\end{equation}Combining  (\ref{equ:N SDE}), (\ref{equ:integ condi aero}),  and (\ref{equ:Wt - Wtau equ zero}), we have $\mathbb{E}^{\lambda,\theta}_{\tau}[N_t-N_{\tau}]=0$ for any $\tau\in[0,t)$. This implies  that $N$ is a continuous martingale with respect to $\mathbb{F}^{\lambda,\theta}$. Further, applying It$\widehat{\text{o}}$'s formula to $(N_t-N_{\tau})(N_t-N_{\tau})^{\top}$ for any $\tau\in[0,t)$ and taking conditional expectation given $\mathcal{F}^{\lambda,\theta}_{\tau}$, it follows that
	\begin{equation*}
		\begin{aligned}
			\mathbb{E}^{\lambda,\theta}_{\tau}\left[(N_t-N_{\tau})(N_t-N_{\tau})^{\top}\right]=(t-\tau)\mathcal{I}_{r\times r}.
		\end{aligned}
	\end{equation*}Then, we can derive statement \ref{con:brow mot} from the L\'evy's characterization of Brownian motion (\cite[Theorem 4.2]{Liptser 1978 1}).
	
	\noindent
	\ref{con:condi expe zero}: In view of \eqref{equ:N SDE}, we obtain 
	\begin{equation}\label{equ: x dNs}
		\begin{aligned}
			\int_{0}^{t}\mathcal{X}\mathrm{d}N_s=\int_{0}^{t}\mathcal{X}F\left(\eta_s-\mathbb{E}^{\lambda,\theta}_s[\eta_s]\right)\mathrm{d}s+\int_{0}^{t}\mathcal{X}\mathrm{d}W_s.
		\end{aligned}
	\end{equation}Taking conditional expectation $\mathbb{E}^{\theta}_t[\cdot]$ on both sides of \eqref{equ: x dNs} and  employing   \cite[Lemma 2.4]{Lv Xiong Zhang 2023}, we have $\mathbb{E}^{\theta}_t[\int_{0}^{t}\mathcal{X}\mathrm{d}W_s]=0$ and  $\mathbb{E}^{\theta}_t[\int_{0}^{t}\mathcal{X}F(\eta_s-\mathbb{E}^{\lambda,\theta}_s[\eta_s])\mathrm{d}s]=\int_{0}^{t}\mathbb{E}^{\theta}_{s}[\mathcal{X}F(\eta_s-\mathbb{E}^{\lambda,\theta}_s[\eta_s])]\mathrm{d}s=\int_{0}^{t}\mathcal{X}F\mathbb{E}^{\theta}_{s}[\eta_s-\mathbb{E}^{\lambda,\theta}_s[\eta_s]]\mathrm{d}s=0$ since $\mathcal{X}$ is $\mathbb{F}^{\theta}$-adapted.  This gives the desired result.
\end{proof}

\begin{Lemma}\label{lem:filt of eta}
	Regarding process $\eta$ and its optimal filter given $\mathbb{F}^{\lambda,\theta}$, we have the following conclusions:
	
	\begin{enumerate}[label=(\roman*)]
		\item \label{lem:cond gaussian} For every $t\in[0,T]$, the conditional distribution  of  $\eta_t$ given $\mathcal{F}^{\lambda,\theta}_t$ is Gaussian;
		\item \label{lem:filter of eta} The optimal filter process  $\mathbb{E}^{\lambda,\theta}_t[\eta_t]$  satisfies 
		\begin{equation}
			\label{equ:sde for filtering of eta}
			\left\{\begin{aligned}
				\mathrm{d}\mathbb{E}^{\lambda,\theta}_t[\eta_t]=&\ \left[A_{(\theta_t,t)}\mathbb{E}^{\lambda,\theta}_t[\eta_t]+\widehat{A}_{(\theta_t,t)}\mathbb{E}^{\theta}_t[\eta_t]+b_{(\theta_t,t)}\right]\mathrm{d}t+Q^{\eta}_t\mathrm{d}N_t,\\
				\mathbb{E}^{\lambda,\theta}_{0}[\eta_0]=&\ \mu,
			\end{aligned}\right.
		\end{equation}where  $N_t$ is defined as (\ref{equ:def N}) and the filter gain $Q^{\eta}_t$  is given by  
		\begin{equation}
			\label{equ:def of Q in con}
			\begin{aligned}
				Q^{\eta}_t= C_{(\theta_t,t)}+\Phi^{\eta}_tF^{\top}_{(\theta_t,t)}.
			\end{aligned}
		\end{equation}The conditional covariance   $\Phi^{\eta}_t=\mathbb{E}^{\lambda,\theta}_t[(\eta_t-\mathbb{E}^{\lambda,\theta}_t[\eta_t])(\eta_t-\mathbb{E}^{\lambda,\theta}_t[\eta_t])^{\top}]$ in (\ref{equ:def of Q in con}) satisfies matrix-valued Riccati equation
		\begin{equation}
			\label{equ:Riccati equ for Phi}
			\left\{\begin{aligned}
				\dot{\Phi}^{\eta}_t=&\ \widetilde{A}_{(\theta_t,t)}\Phi^{\eta}_t+\Phi^{\eta}_t\widetilde{A}_{(\theta_t,t)}^{\top}+\bar{C}_{(\theta_t,t)}\bar{C}^{\top}_{(\theta_t,t)} -\Phi^{\eta}_tF_{(\theta_t,t)}^{\top}F_{(\theta_t,t)}\Phi^{\eta}_t,\\
				\Phi^{\eta}_0= &\ \sigma,
			\end{aligned}\right.
		\end{equation}where $\widetilde{A}_{(\theta_t,t)}=A_{(\theta_t,t)}-C_{(\theta_t,t)}F_{(\theta_t,t)}$. 
	\end{enumerate}
\end{Lemma}

\begin{proof}
	\ref{lem:cond gaussian}: Taking conditional expectation $\mathbb{E}^{\theta}_t[\cdot]$ on both sides of (\ref{equ:eta}), we obtain 
	\begin{equation}
		\label{equ:condi eta}
		\left\{\begin{aligned}
			\mathrm{d}\mathbb{E}^{\theta}_t[\eta_t]=&\ \left[\check{A}_{(\theta_t,t)}\mathbb{E}^{\theta}_t[\eta_t]+b_{(\theta_t,t)}\right]\mathrm{d}t,\\
			\mathbb{E}^{\theta}_{0}[\eta_0]=&\ \mu.
		\end{aligned}\right.
	\end{equation}Introduce $\check{\Upsilon}_t$ by $\mathrm{d}\check{\Upsilon}_t=\check{A}_{(\theta_t,t)}\check{\Upsilon}_t\mathrm{d}t$ for any $t\in[0,T]$ and $\check{\Upsilon}_0= \mathcal{I}_{n\times n}$. Then the solution of (\ref{equ:condi eta}) is $\mathbb{E}^{\theta}_t[\eta_t]=\check{\Upsilon}_t[\mu+\int_{0}^{t}\check{\Upsilon}^{-1}b\mathrm{d}s]$.  Thus, there exists a functional $\gamma_{(x,t)}$ such that  $\mathbb{E}^{\theta}_t[\eta_t]=\gamma_{(\theta,t)}$, where $\gamma_{(x,t)}$ is  measurable with respect  to $\mathcal{B}_t$ (the $\sigma$-algebra in the space of c$\grave{a}$dl$\grave{a}$g  functions $x=\{x_s,s\in[0,T]\}$, generated by $\{x_s,s\in[0,t]\}$).

	Define 
	\begin{equation*}
		\begin{aligned}
			\mathcal{Z}_{t}=&\ \left[\begin{array}{c}
				\eta_t\\
				\lambda_t
			\end{array}\right],\quad z= \left[\begin{array}{c}
			\xi \\
			0_{r\times1}
			\end{array}\right],\quad \mathcal{A}_{(\theta_t,t)}=\left[\begin{array}{cc}
			A_{(\theta_t,t)} & 0_{n\times r} \\
			F_{(\theta_t,t)} & 0_{r\times r}
			\end{array}\right],\\
			\mathcal{D}_{(\theta,t)}=&\ \left[\begin{array}{c}
				\widehat{A}_{(\theta_t,t)}\gamma_{(\theta,t)}+b_{(\theta_t,t)} \\
				\widehat{F}_{(\theta_t,t)}\gamma_{(\theta,t)}+f_{(\theta_t,t)}
			\end{array}\right],\quad\mathcal{C}_{(\theta_t,t)}=\left[\begin{array}{c}
			C_{(\theta_t,t)} \\
			\mathcal{I}_{r\times r}
			\end{array}\right],\quad \bar{\mathcal{C}}_{(\theta_t,t)}=\left[\begin{array}{c}
			\bar{C}_{(\theta_t,t)} \\
			0_{r\times d}
			\end{array}\right].
		\end{aligned}
	\end{equation*}Let $\Upsilon_t$ satisfy $\mathrm{d}\Upsilon_t=\mathcal{A}_{(\theta_t,t)}\Upsilon_t\mathrm{d}t$ with initial value  $\Upsilon_0=\mathcal{I}_{(n+r)\times (n+r)}$. Then we can derive that 
	\begin{equation}\label{equ: mathcal Z}
		\begin{aligned}
			\mathcal{Z}_t=\Upsilon_t\left[z+\int_{0}^{t}\Upsilon_s^{-1}\mathcal{D}_{(\theta,s)}\mathrm{d}s+\int_{0}^{t}\Upsilon_s^{-1}\mathcal{C}_{(\theta_s,s)}\mathrm{d}W_s+\int_{0}^{t}\Upsilon_s^{-1}\bar{\mathcal{C}}_{(\theta_s,s)}\mathrm{d}\bar{W}_s\right].
		\end{aligned}
	\end{equation}Conditional on $\mathcal F_t^\theta$, the coefficient
	processes appearing in \eqref{equ: mathcal Z} are determined
	by the realized trajectory of the Markov chain. Since
	$\xi$, $W$, and $\bar W$ are mutually independent
	Gaussian objects and are independent of $\theta$, applying
	\eqref{equ: mathcal Z} at $t,t_1,\ldots,t_k$ shows that,
	for every $k\geq1$ and
	$0\leq t_1\leq\cdots\leq t_k\leq t$, the random vector
	$
	(\eta_t,\lambda_{t_1},\ldots,\lambda_{t_k})
	$
	is jointly Gaussian conditional on $\mathcal F_t^\theta$.
	Since $\lambda$ has continuous paths,
	$\mathcal F_t^\lambda$ is generated by its values at
	rational times in $[0,t]$. The conditional Gaussian
	projection property therefore implies that the conditional distribution  of  $\eta_t$ given $\mathcal{F}^{\lambda,\theta}_t$ is Gaussian.

	\noindent
	\ref{lem:filter of eta}: In this proof, for a matrix $A$,  we denote its $ij$-th  entry by  $A^{ij}$ and  its  $i$-th row by   $A^{i\cdot}$. For convenience, we introduce the notation $\hbar_t= F\mathbb{E}^{\lambda,\theta}_{t}[\eta_t]+\widehat{F}\mathbb{E}^{\theta}_{t}[\eta_t]+f$.
	
	Applying the Kushner-Stratonovich equation (see Theorem 8.1 in \cite{Liptser 1978 1}), we have ($i\in\{1,\cdots,n\}$)
	\begin{equation}\label{equ:fil of eta i}
		\begin{aligned}
			 \mathrm{d}\mathbb{E}^{\lambda,\theta}_{t}[\eta^{i}_t]
			= \left[\sum_{k=1}^{n}\left(A^{ik}\mathbb{E}^{\lambda,\theta}_{t}[\eta^k_{t}]+\widehat{A}^{ik}\mathbb{E}^{\theta}_t[\eta^k_{t}]\right)+b^i\right]\mathrm{d}t +\left(C^{i\cdot}+\mathbb{E}^{\lambda,\theta}_{t}[\eta^i_{t}\hbar_t^{\top}]-\mathbb{E}^{\lambda,\theta}_{t}[\eta^i_{t}]\mathbb{E}^{\lambda,\theta}_{t}[\hbar_{t}^{\top}]\right)\mathrm{d}N_t.
		\end{aligned}
	\end{equation}

	By the definition of the conditional covariance  $\Phi^{\eta}_t$, we have ($i,j\in\{1,\cdots,n\}$)
	\begin{equation}\label{equ: def of Phi ij}
		\begin{aligned}
			(\Phi^{\eta}_t)^{ij}
			=&\ \mathbb{E}^{\lambda,\theta}_{t}[\eta^i_{t}\eta^j_{t}]-\mathbb{E}^{\lambda,\theta}_{t}[\eta^i_{t}]\mathbb{E}^{\lambda,\theta}_{t}[\eta^j_{t}].
		\end{aligned}
	\end{equation}Using \eqref{equ: def of Phi ij} and the fact that all coefficients and $\mathbb{E}^{\theta}_t[\eta_t]$ are $\mathcal{F}^{\lambda,\theta}_t$-adapted, we obtain
	\begin{equation}\label{equ: stoc int}
		\begin{aligned}
			&\ \left(C^{i\cdot}+\mathbb{E}^{\lambda,\theta}_{t}[\eta^i_{t}\hbar_t^{\top}]-\mathbb{E}^{\lambda,\theta}_{t}[\eta^i_{t}]\mathbb{E}^{\lambda,\theta}_{t}[\hbar_{t}^{\top}]\right)\mathrm{d}N_t\\
			= &\  \left[C^{i\cdot}+\left(\mathbb{E}^{\lambda,\theta}_{t}[\eta^i_{t}\eta_t^{\top}]-\mathbb{E}^{\lambda,\theta}_{t}[\eta^i_{t}]\mathbb{E}^{\lambda,\theta}_{t}[\eta_{t}^{\top}]\right)F^{\top}\right]\mathrm{d}N_t\\
			= &\ \sum_{\ell=1}^{r} \left[C^{i\ell} +\sum_{k=1}^{n}\left(\mathbb{E}^{\lambda,\theta}_{t}[\eta^i_{t}\eta^{k}_{t}]-\mathbb{E}^{\lambda,\theta}_{t}[\eta^i_{t}]\mathbb{E}^{\lambda,\theta}_{t}[\eta^k_{t}]\right)(F^{\top})^{k\ell}\right]\mathrm{d}N^{\ell}_t\\
			= &\ \sum_{\ell=1}^{r} \left[C^{i\ell}+\sum_{k=1}^{n}(\Phi^{\eta}_t)^{ik}(F^{\top})^{k\ell}\right]\mathrm{d}N^{\ell}_t.\\
		\end{aligned}
	\end{equation}Substituting \eqref{equ: stoc int} into \eqref{equ:fil of eta i}, we obtain
	\begin{equation}\label{equ: SDE for cond i}
		\begin{aligned}
			\mathrm{d}\mathbb{E}^{\lambda,\theta}_{t}[\eta^{i}_t]
			=  \left[\sum_{k=1}^{n}\left(A^{ik}\mathbb{E}^{\lambda,\theta}_{t}[\eta^k_{t}]+\widehat{A}^{ik}\mathbb{E}^{\theta}_t[\eta^k_{t}]\right)+b^i\right]\mathrm{d}t +\sum_{\ell=1}^{r} \left[C^{i\ell}+\sum_{k=1}^{n}(\Phi^{\eta}_t)^{ik}(F^{\top})^{k\ell}\right]\mathrm{d}N^{\ell}_t.
		\end{aligned}
	\end{equation}Moreover, we can derive that $\mathbb{E}^{\lambda,\theta}_{0}[\eta^i_0]=\mu^i$ since  $\lambda_0=0$ and the random variable $\xi$ is independent of the Markov chain $\theta$. This  yields  filtering equation \eqref{equ:sde for filtering of eta}. 
	
	Next, we derive  differential equation \eqref{equ:Riccati equ for Phi} for the evolution of the conditional covariance  $\Phi^{\eta}_t$. Applying  It$\widehat{\text{o}}$'s formula to $\eta^{i}_t\eta^{j}_t$, we get
	\begin{equation}\label{equ: SDE eta i j}
		\begin{aligned}
			&\ \mathrm{d}\left(\eta^{i}_t\eta^{j}_t\right)\\
			=&\  \left[\sum_{k=1}^{n}\left(A^{ik}\eta^{k}_t\eta^j_t+\widehat{A}^{ik}\mathbb{E}^{\theta}_t[\eta^{k}_t]\eta^j_t+A^{jk}\eta^{k}_t\eta^i_t +\widehat{A}^{jk}\mathbb{E}^{\theta}_t[\eta^{k}_t]\eta^i_t\right)+b^i\eta^j_t+b^j\eta^i_t+\left(CC^{\top}+\bar{C}\bar{C}^{\top}\right)^{ij}\right]\mathrm{d}t\\
			&\ +\sum_{\ell=1}^{r}\left(C^{i\ell}\eta^{j}_t+C^{j\ell}\eta^{i}\right)\mathrm{d}W^{\ell}_t +\sum_{\ell=1}^{d}\left(\bar{C}^{i\ell}\eta^{j}_t+\bar{C}^{j\ell}\eta^{i}\right)\mathrm{d}\bar{W}^{\ell}_t.\\
		\end{aligned}
	\end{equation}Applying Kushner-Stratonovich equation to \eqref{equ: SDE eta i j} and arguing as in \eqref{equ: stoc int}, we obtain 
	\begin{equation}\label{equ: cond eta i j}
		\begin{aligned}
			&\ \mathrm{d}\mathbb{E}^{\lambda,\theta}_{t}[\eta^{i}_t\eta^{j}_t]\\
			=&\  \left[\sum_{k=1}^{n}\left(A^{ik}\mathbb{E}^{\lambda,\theta}_{t}[\eta^{k}_t\eta^j_t]+\widehat{A}^{ik}\mathbb{E}^{\theta}_t[\eta^{k}_t]\mathbb{E}^{\lambda,\theta}_{t}[\eta^j_t] +A^{jk}\mathbb{E}^{\lambda,\theta}_{t}[\eta^{k}_t\eta^i_t]+\widehat{A}^{jk}\mathbb{E}^{\theta}_t[\eta^{k}_t]\mathbb{E}^{\lambda,\theta}_{t}[\eta^i_t]\right)+b^i\mathbb{E}^{\lambda,\theta}_{t}[\eta^j_t]\right.\\
			&\ \left. +b^j\mathbb{E}^{\lambda,\theta}_{t}[\eta^i_t]+\left(CC^{\top}+\bar{C}\bar{C}^{\top}\right)^{ij}\right]\mathrm{d}t\\
			&\ +\sum_{\ell=1}^{r}\left[C^{i\ell}\mathbb{E}^{\lambda,\theta}_{t}[\eta^j_t]+C^{j\ell}\mathbb{E}^{\lambda,\theta}_{t}[\eta^i_t] +\sum_{k=1}^{n}\left(\mathbb{E}^{\lambda,\theta}_{t}[\eta^i_t\eta^j_t\eta^k_t]-\mathbb{E}^{\lambda,\theta}_{t}[\eta^i_t\eta^j_t]\mathbb{E}^{\lambda,\theta}_{t}[\eta^k_t]\right)(F^{\top})^{k\ell}\right]\mathrm{d}N^{\ell}_t.
		\end{aligned}
	\end{equation}Since  $\eta$ is conditionally Gaussian with respect to $\mathbb{F}^{\lambda,\theta}$, its conditional third-order moment satisfies
	\begin{equation}\label{equ: third comment}
		\begin{aligned}
			\mathbb{E}^{\lambda,\theta}_{t}[\eta^i_t\eta^j_t\eta^k_t]= \mathbb{E}^{\lambda,\theta}_{t}[\eta^i_t]\mathbb{E}^{\lambda,\theta}_{t}[\eta^j_t]\mathbb{E}^{\lambda,\theta}_{t}[\eta^k_t]+\mathbb{E}^{\lambda,\theta}_{t}[\eta^i_t](\Phi^{\eta}_t)^{jk} +\mathbb{E}^{\lambda,\theta}_{t}[\eta^j_t](\Phi^{\eta}_t)^{ik}+\mathbb{E}^{\lambda,\theta}_{t}[\eta^k_t](\Phi^{\eta}_t)^{ij}.
		\end{aligned}
	\end{equation}On the other hand,   applying It$\widehat{\text{o}}$'s formula to $\mathbb{E}^{\lambda,\theta}_{t}[\eta^{i}_t]\mathbb{E}^{\lambda,\theta}_{t}[\eta^{j}_t]$  and using \eqref{equ: SDE for cond i} yields 
	\begin{equation}\label{equ: con i con j}
		\begin{aligned}
			&\ \mathrm{d}\mathbb{E}^{\lambda,\theta}_{t}[\eta^{i}_t]\mathbb{E}^{\lambda,\theta}_{t}[\eta^{j}_t]\\
			=&\  \left[\sum_{k=1}^{n}\left(A^{ik}\mathbb{E}^{\lambda,\theta}_{t}[\eta^{k}_t]\mathbb{E}^{\lambda,\theta}_{t}[\eta^j_t]+\widehat{A}^{ik}\mathbb{E}^{\theta}_t[\eta^{k}_t]\mathbb{E}^{\lambda,\theta}_{t}[\eta^j_t]\right.\right.\\
			&\ \left.\left.+A^{jk}\mathbb{E}^{\lambda,\theta}_{t}[\eta^{k}_t]\mathbb{E}^{\lambda,\theta}_{t}[\eta^i_t] +\widehat{A}^{jk}\mathbb{E}^{\theta}_t[\eta^{k}_t]\mathbb{E}^{\lambda,\theta}_{t}[\eta^i_t]\right) +b^i\mathbb{E}^{\lambda,\theta}_{t}[\eta^j_t]+b^j\mathbb{E}^{\lambda,\theta}_{t}[\eta^i_t]\right]\mathrm{d}t\\
			&\ +\sum_{\ell=1}^{r} \left[C^{i\ell}\mathbb{E}^{\lambda,\theta}_{t}[\eta^j_t]+\sum_{k=1}^{n}(\Phi^{\eta}_t)^{ik}(F^{\top})^{k\ell}\mathbb{E}^{\lambda,\theta}_{t}[\eta^j_t] +  C^{j\ell}\mathbb{E}^{\lambda,\theta}_{t}[\eta^i_t]+\sum_{k=1}^{n}(\Phi^{\eta}_t)^{jk}(F^{\top})^{k\ell}\mathbb{E}^{\lambda,\theta}_{t}[\eta^i_t]\right]\mathrm{d}N^{\ell}_t\\	
			&\ +\left[\left(\left(C+\Phi^{\eta}F^{\top}\right)\mathrm{d}N\right)^{i},\left(\left(C+\Phi^{\eta}F^{\top}\right)\mathrm{d}N\right)^{j}\right]_t.
		\end{aligned}
	\end{equation}
	
	Since the innovation process $N$ is a standard Brownian motion (see Lemma \ref{lem:N BM}) and $(\Phi^{\eta}_t)^{\top}=\Phi^{\eta}_t$ is symmetric,  the quadratic covariation term satisfies 
	\begin{equation}\label{equ: qua cova}
		\begin{aligned}
			\left[\left(\left(C+\Phi^{\eta}F^{\top}\right)\mathrm{d}N\right)^{i},\left(\left(C+\Phi^{\eta}F^{\top}\right)\mathrm{d}N\right)^{j}\right]_t
			=&\ 	\left[\left(C+\Phi^{\eta}F^{\top}\right)\left(C+\Phi^{\eta}F^{\top}\right)^{\top}\right]^{ij}\mathrm{d}t\\
			=&\ \left(CC^{\top}+CF\Phi^{\eta}+\Phi^{\eta}F^{\top}C^{\top}+\Phi^{\eta}F^{\top}F\Phi^{\eta}\right)^{ij}\mathrm{d}t.
		\end{aligned}
	\end{equation}
	
	Recalling identity \eqref{equ: def of Phi ij}, we insert \eqref{equ: third comment} into \eqref{equ: cond eta i j} and \eqref{equ: qua cova} into \eqref{equ: con i con j}, and then subtract \eqref{equ: con i con j} from \eqref{equ: cond eta i j}. This yields the following differential equation for the $ij$-th component of $\Phi^{\eta}_t$, 
	\begin{equation*}\label{equ: SDE for Phi ij}
		\begin{aligned}
			&\ \mathrm{d}(\Phi^{\eta}_t)^{ij}\\
			=&\ \left\{\sum_{k=1}^{n}\left[A^{ik}\left(\mathbb{E}^{\lambda,\theta}_{t}[\eta^{k}_t\eta^j_t]-\mathbb{E}^{\lambda,\theta}_{t}[\eta^{k}_t]\mathbb{E}^{\lambda,\theta}_{t}[\eta^j_t]\right) +A^{jk}\left(\mathbb{E}^{\lambda,\theta}_{t}[\eta^{k}_t\eta^i_t]-\mathbb{E}^{\lambda,\theta}_{t}[\eta^{k}_t]\mathbb{E}^{\lambda,\theta}_{t}[\eta^i_t]\right)\right] +\left(CC^{\top}+\bar{C}\bar{C}^{\top}\right)^{ij}\right\}\mathrm{d}t\\
			=&\ \left[\sum_{k=1}^{n}\left(A^{ik}(\Phi^{\eta}_t)^{kj}+A^{jk}(\Phi^{\eta}_t)^{ki}\right)+\left(CC^{\top}+\bar{C}\bar{C}^{\top}\right)^{ij} -\left(CC^{\top}+CF\Phi^{\eta}+\Phi^{\eta}F^{\top}C^{\top}+\Phi^{\eta}F^{\top}F\Phi^{\eta}\right)^{ij}\right]\mathrm{d}t\\
			=&\ \left[(A\Phi^{\eta}_t)^{ij}+(\Phi^{\eta}_tA^{\top})^{ij}+\left(CC^{\top}+\bar{C}\bar{C}^{\top}\right)^{ij} -\left(CC^{\top}+CF\Phi^{\eta}+\Phi^{\eta}F^{\top}C^{\top}+\Phi^{\eta}F^{\top}F\Phi^{\eta}\right)^{ij}\right]\mathrm{d}t\\
			=&\ \left[\widetilde{A}\Phi^{\eta}_t+\Phi^{\eta}_t\widetilde{A}^{\top}+\bar{C}\bar{C}^{\top}-\Phi^{\eta}_tF^{\top}F\Phi^{\eta}_t\right]^{ij}\mathrm{d}t.\\
		\end{aligned}
	\end{equation*}We note that all the integrands  before $\mathrm{d}N_t$  in the above expression  vanish.  Moreover, it follows directly from \eqref{equ: def of Phi ij} that $(\Phi^{\eta}_0)^{ij}=\sigma^{ij}$. This leads to  Riccati equation \eqref{equ:Riccati equ for Phi}.
\end{proof}

Based on the above auxiliary lemmas, we now  present the optimal filter regarding state-observation dynamics  (\ref{equ:state})-(\ref{equ:obser}).
\begin{Theorem}\label{the:filter of X}
	
	\begin{enumerate}[label=(\roman*)]
		\item  \label{asser two in main filter}  Process $V=\{V_t\}_{t\in[0,T]}
		$  defined by 
		\begin{equation}
			\label{equ:def of V}
			\begin{aligned}
				V_t=&\ Y^v_t -\int_{0}^{t}\left[F_{(\theta_s,s)}\mathbb{E}^{Y^v,\theta}_s[X^v_s]+\widehat{F}_{(\theta_s,s)}\mathbb{E}^{\theta}_s[X^v_s]+f_{(\theta_s,s)}\right]\mathrm{d}s
			\end{aligned}
		\end{equation}is an $\mathbb{F}^{Y^v,\theta}$-adapted and $\mathbb{R}^{r}$-valued standard Brownian motion. Moreover, if $\mathcal{X}$ is an $\mathbb{F}^{\theta}$-adapted and square-integrable process with appropriate dimension, then $	\mathbb{E}^{\theta}_t[\int_{0}^{t}\mathcal{X}_s\mathrm{d}V_s]=0$ for any $t\in[0,T]$.
		\item \label{asser one in main filter} For any admissible control $v\in\mathcal{U}_{ad}$, the optimal filter $\mathbb{E}^{Y^v,\theta}_t[X^v_t]$ satisfies
		\begin{equation}
			\label{equ: filter of x}
			\left\{\begin{aligned}
				\mathrm{d}\mathbb{E}^{Y^v,\theta}_t[X^v_t]=&\ \left[A_{(\theta_t,t)}\mathbb{E}^{Y^v,\theta}_t[X^v_t]+\widehat{A}_{(\theta_t,t)}\mathbb{E}^{\theta}_t[X^v_t] +B_{(\theta_t,t)}v_t+\widehat{B}_{(\theta_t,t)}\mathbb{E}^{\theta}_t[v_t] +b_{(\theta_t,t)}\right]\mathrm{d}t+Q_t\mathrm{d}V_t,\\
				\mathbb{E}^{Y^v,\theta}_0[X^v_0]=&\ \mu.
			\end{aligned}\right.
		\end{equation}In (\ref{equ: filter of x}),  the filter gain  $Q_t$  is given by 
		\begin{equation}
			\label{equ:def of QX in con}
			\begin{aligned}
				Q_t= C_{(\theta_t,t)}+\Phi_tF^{\top}_{(\theta_t,t)},
			\end{aligned}
		\end{equation}where the  conditional covariance 
		\begin{equation}
			\label{equ:def of Phi}
			\begin{aligned}
				\Phi_t= \mathbb{E}^{Y^v,\theta}_t[(X^v_t-\mathbb{E}^{Y^v,\theta}_t[X^v_t])(X^v_t-\mathbb{E}^{Y^v,\theta}_t[X^v_t])^{\top}]
			\end{aligned}
		\end{equation}satisfies
		\begin{equation}
			\label{equ:Riccati equ for error}
			\left\{\begin{aligned}
				\dot{\Phi}_t=&\ \widetilde{A}_{(\theta_t,t)}\Phi_t+\Phi_t\widetilde{A}_{(\theta_t,t)}^{\top}+\bar{C}_{(\theta_t,t)}\bar{C}^{\top}_{(\theta_t,t)} -\Phi_tF_{(\theta_t,t)}^{\top}F_{(\theta_t,t)}\Phi_t,\\
				\Phi_0= &\ \sigma.
			\end{aligned}\right.
		\end{equation}
	\end{enumerate}

\end{Theorem}

\begin{proof}
	\ref{asser two in main filter}: Utilizing  identity (\ref{equ:sum rela}), Lemma \ref{lem:filtr ident}, and noting that $X^{1,v}$ is $\mathbb{F}^{Y^v,\theta}$-adapted, we have, for any $t\in[0,T]$,
	\begin{equation}
		\label{equ:decom for filtering of x}
		\begin{aligned}
			\mathbb{E}^{Y^v,\theta}_t[X^v_t]=&\ \mathbb{E}^{Y^v,\theta}_t[\eta_t+X^{1,v}_t]=\mathbb{E}^{Y^v,\theta}_t[\eta_t]+X^{1,v}_t
			= \mathbb{E}^{\lambda,\theta}_t[\eta_t]+X^{1,v}_t.
		\end{aligned}
	\end{equation}Combining  (\ref{equ:Y1})-(\ref{equ:def N}), (\ref{equ:def of V}), and (\ref{equ:decom for filtering of x}), we derive 
	\begin{equation}
		\label{equ: rela betw N and V}
		\begin{aligned}
			N_t =&\ \lambda_t-\int_{0}^{t}\left[F\mathbb{E}^{\lambda,\theta}_s[\eta_s]+\widehat{F}\mathbb{E}^{\theta}_s[\eta_s]+f\right]\mathrm{d}s\\
			=&\ Y^v_t -\int_{0}^{t}\left[FX^{1,v}_s+\widehat{F}\mathbb{E}^{\theta}_s[X^{1,v}_s]\right]\mathrm{d}s -\int_{0}^{t}\left[F\mathbb{E}^{\lambda,\theta}_s[\eta_s]+\widehat{F}\mathbb{E}^{\theta}_s[\eta_s]+f\right]\mathrm{d}s\\
			=&\ Y^v_t-\int_{0}^{t}\left[F\mathbb{E}^{Y^v,\theta}_s[X^v_s]+\widehat{F}\mathbb{E}^{\theta}_s[X^v_s]+f\right]\mathrm{d}s\\
			=&\ V_t.
		\end{aligned}
	\end{equation}According to Lemmas \ref{lem:filtr ident}, \ref{lem:N BM}, and relation (\ref{equ: rela betw N and V}), we prove \ref{asser two in main filter}.

	\ref{asser one in main filter}: In view of  (\ref{equ:x1}), (\ref{equ:sde for filtering of eta}), and  (\ref{equ:decom for filtering of x}), it follows that $\mathbb{E}^{Y^v,\theta}_t[X^v_t]$ satisfies
	\begin{equation}
		\label{equ: filter of x in proof}
		\left\{\begin{aligned}
			\mathrm{d}\mathbb{E}^{Y^v,\theta}_t[X^v_t]=&\ \left[A_{(\theta_t,t)}\mathbb{E}^{Y^v,\theta}_t[X^v_t]+\widehat{A}_{(\theta_t,t)}\mathbb{E}^{\theta}_t[X^v_t] +B_{(\theta_t,t)}v_t+\widehat{B}_{(\theta_t,t)}\mathbb{E}^{\theta}_t[v_t]+b_{(\theta_t,t)}\right]\mathrm{d}t +Q^{\eta}_t\mathrm{d}N_t,\\
			\mathbb{E}^{Y^v,\theta}_0[X^v_0]=&\ \mu.
		\end{aligned}\right.
	\end{equation}Noting (\ref{equ:sum rela}), (\ref{equ:def of Phi}), Lemma \ref{lem:filtr ident} and   the fact that $X^{1,v}$ is $\mathbb{F}^{\lambda,\theta}$-adapted, we obtain $\Phi^{\eta}_t=\mathbb{E}^{\lambda,\theta}_t[(\eta_t-\mathbb{E}^{\lambda,\theta}_t[\eta_t])(\eta_t-\mathbb{E}^{\lambda,\theta}_t[\eta_t])^{\top}]	=\Phi_t$, which, together with (\ref{equ:def of Q in con}), gives 
	\begin{equation}
		\label{equ:Q eta=Q X}
		\begin{aligned}
			Q^{\eta}_t= C_{(\theta_t,t)}+\Phi^{\eta}_tF^{\top}_{(\theta_t,t)}
			= C_{(\theta_t,t)}+\Phi_tF^{\top}_{(\theta_t,t)}
			=Q_t.
		\end{aligned}
	\end{equation}Combining Lemma \ref{lem:filt of eta} and  (\ref{equ: rela betw N and V})-(\ref{equ:Q eta=Q X}) yields  assertion \ref{asser one in main filter}.
\end{proof}

\begin{Remark}
	Since the coefficients in (\ref{equ:state}) and (\ref{equ:obser}) are modulated by the observable Markov chain $\theta$, the conditional covariance matrix $\Phi$ becomes a stochastic process. From (\ref{equ:Riccati equ for error}), one can see that $\Phi_t$ is determined by the realized switching trajectory $\{\theta_s, s\in[0,t]\}$. Consequently, $\Phi_t$ cannot, in general, be represented as a deterministic function of the current regime $\theta_t$ alone, and the filter gain $Q_t$ also depends on the past switching history of the Markov chain.
	
	It should be emphasized that this path dependence does not require storing or repeatedly processing the entire regime history. The covariance matrix $\Phi_t$ evolves recursively through the Riccati equation and can be updated sequentially using its current value and the current regime. However, unlike the Riccati equations associated with the control gain, whose solutions can be computed offline for all regimes, $\Phi_t$ and the corresponding filter gain must be propagated online along the realized trajectory of the Markov chain. Therefore, the non-Markovian feature manifests itself primarily through the need for online updating rather than increased memory requirements.
\end{Remark}

\subsection{Equivalent Problem with Full Observation}\label{sub:equi prob}
In this section, first, we formulate problem \textbf{(FO-CMF)} with a  full observation framework and show that  it is  equivalent  to  \textbf{(PO-CMF)}. Then, we solve  \textbf{(FO-CMF)} using maximum principle and obtain a feedback representation of optimal control.

\subsubsection{Formulation of \textbf{(FO-CMF)}} 
For  convenience of notation, we use $\mathbb{X}^v_t$ to denote $\mathbb{E}^{Y^v,\theta}_t[X^v_t]$. By the tower property of conditional expectation, it follows that $\mathbb{E}^{\theta}_t[X^v_t]=\mathbb{E}^{\theta}_t[\mathbb{E}^{Y^v,\theta}_t[X^v_t]]=\mathbb{E}^{\theta}_t[\mathbb{X}^v_t]$. Bearing this in mind, dynamics (\ref{equ: filter of x}) can be rewritten as 
\begin{equation}
	\label{equ: rewr dyna}
	\left\{\begin{aligned}
		\mathrm{d}\mathbb{X}^v_t=&\ \left[A_{(\theta_t,t)}\mathbb{X}^v_t+\widehat{A}_{(\theta_t,t)}\mathbb{E}^{\theta}_t[\mathbb{X}^v_t]+B_{(\theta_t,t)}v_t +\widehat{B}_{(\theta_t,t)}\mathbb{E}^{\theta}_t[v_t]+b_{(\theta_t,t)}\right]\mathrm{d}t +Q_t\mathrm{d}V_t,\\
		\mathbb{X}^v_0=&\ \mu.
	\end{aligned}\right.
\end{equation}Define the estimation error $\mathcal{E}=X^v-\mathbb{X}^v$ and notation $\mathcal{A}_t=A_{(\theta_t,t)}-Q_tF_{(\theta_t,t)}$. Then, by (\ref{equ:state}) and Theorem \ref{the:filter of X}, $\mathcal{E}$ satisfies 
\begin{equation}
	\label{equ:dyna for error}
	\left\{\begin{aligned}
		\mathrm{d}\mathcal{E}_t=&\
		\mathcal{A}_t\mathcal{E}_t\mathrm{d}t -\Phi_tF^{\top}_{(\theta_t,t)}\mathrm{d}W_t+\bar{C}_{(\theta_t,t)}\mathrm{d}\bar{W}_t,\\
		\mathcal{E}_0=&\ \xi-\mu,
	\end{aligned}\right.
\end{equation}which implies that estimation error $\mathcal{E}$ is independent of control.

To proceed, we  introduce a $\mathbb{X}^{v}$-related cost functional: 
\begin{equation}
	\label{equ: rewr cost}
	\begin{aligned}
		&\ J^{F}\left(\xi,\ell_0;v\right)\\
		=&\   \frac{1}{2} \mathbb{E}\left\{\int _ { 0 } ^ { T } \left[\left\langle\left[\begin{array}{l}
			\mathbb{X}^v \\
			v
		\end{array}\right],\left[\begin{array}{cc}
			H & G^{\top} \\
			G & R
		\end{array}\right]\left[\begin{array}{l}
			\mathbb{X}^v \\
			v
		\end{array}\right]\right\rangle +\left\langle\left[\begin{array}{l}
			\mathbb{E}^{\theta}_t[\mathbb{X}^v] \\
			\mathbb{E}^{\theta}_t[v]
		\end{array}\right],\left[\begin{array}{cc}
			\widehat{H} & \widehat{G}^{\top} \\
			\widehat{G} & \widehat{R}
		\end{array}\right]\left[\begin{array}{c}
			\mathbb{E}^{\theta}_t[\mathbb{X}^v] \\
			\mathbb{E}^{\theta}_t[v]
		\end{array}\right]\right\rangle \right]\mathrm{d} t \right.\\
		&\ \left. +\langle \mathbb{X}^v_T, S_{\theta_T} 	\mathbb{X}^v_T\rangle +\langle\mathbb{E}^{\theta}_T[\mathbb{X}^v_T], \widehat{S}_{\theta_T} 	\mathbb{E}^{\theta}_T[\mathbb{X}^v_T]\rangle \right\},\\
	\end{aligned}
\end{equation}and  an  $\mathcal{E}$-related one:
\begin{equation}
	\label{equ: cost inde of cont}
	\begin{aligned}
		\mathbb{J}\left(\xi,\ell_0\right)=\frac{1}{2}\mathbb{E} \left\{\int_{0}^{T}\left \langle \mathcal{E}, H \mathcal{E}\right\rangle \mathrm{d}t+\left\langle \mathcal{E}_T, S_{\theta_T} \mathcal{E}_T\right\rangle \right\}.
	\end{aligned}
\end{equation}Based on (\ref{equ: rewr dyna}) and (\ref{equ: rewr cost}), we formulate the following optimal control problem.\\
\textbf{\textbf{(FO-CMF)}:} Find an admissible control $u^{*}\in\mathcal{U}_{ad}$ such that $J^{F}\left(\xi,\ell_0; u^{*}\right)=\inf_{v\in\mathcal{U}_{ad}}J^{F}\left(\xi,\ell_0; v\right)$ subject to (\ref{equ: rewr dyna}).

For \textbf{(FO-CMF)}, it is still a conditional mean-field type control problem. However, since system (\ref{equ: rewr dyna}) is driven by innovation process  $V$, which is an $\mathbb{F}^{Y^v,\theta}$-adapted Brownian motion, \textbf{(FO-CMF)} is  an optimal control problem with fully observed state. Furthermore, we have the following relation between \textbf{(PO-CMF)} and \textbf{(FO-CMF)}.

\begin{Theorem}
	\label{the: equi betw two pro}
	\textbf{(PO-CMF)} is equivalent to 	\textbf{(FO-CMF)}.
\end{Theorem}
\begin{proof}
	From (\ref{equ:dyna for error}), we observe that estimation error $\mathcal{E}$ is independent of control $v$. Furthermore, by virtue of the connection between conditional expectation and orthogonal projection (\cite[Section 6.11]{Williams 1991}), the following orthogonality property holds between $\mathcal{E}$ and all $\mathbb{F}^{Y^v,\theta}$-adapted processes:
	\begin{equation}	
		\label{equ:orth rela}
		\begin{aligned}
			\mathbb{E}\left\langle\mathcal{E}_t,\mathcal{X}_t\right\rangle=0, \hspace{0.2cm}  t\in[0,T], \mathcal{X}\in\mathcal{L}^2_{\mathbb{F}^{Y^v,\theta}}([0,T];\mathbb{R}^{n}).
		\end{aligned}
	\end{equation}Using (\ref{equ:orth rela}), we can decompose cost functional (\ref{equ:ori cost}) as
	\begin{equation}
		\label{equ: deco of cost}
		\begin{aligned}
			J^{P}\left(\xi,\ell_0; v\right)=J^{F}\left(\xi,\ell_0;v\right)+ \mathbb{J}\left(\xi,\ell_0\right).
		\end{aligned}
	\end{equation}Given that $\mathcal{E}$ is independent of control $v$,  cost $\mathbb{J}\left(\xi,\ell_0\right)$ is unaffected by $v$. Therefore, minimizing $J^{P}\left(\xi,\ell_0; v\right)$ subject to (\ref{equ:state}) (i.e., \textbf{(PO-CMF)}) is identical to  minimizing  cost $J^{F}\left(\xi,\ell_0;v\right)$ subject to (\ref{equ: rewr dyna}) (i.e., \textbf{(FO-CMF)}).
\end{proof}

\subsubsection{Solution of \textbf{(FO-CMF)}}

Due to the fact that the cost functions in  (\ref{equ: rewr cost}) depend on    the conditional expectations terms in a nonlinear way, the resulting control problem is time-inconsistent, which implies that  the Bellman optimality principle  fails to work. Consequently, we resort to the stochastic maximum principle to solve \textbf{(FO-CMF)}. To characterize the optimality condition,  adjoint equation plays a  central role within maximum principle. Adjoint equation takes the form of a BSDE, which is fundamentally different from a deterministic backward ODE. In particular, to ensure the adaptedness of the solution, the BSDE introduces a stochastic integral term of another process (see \cite[Chapter 7]{Yong Zhou 1999}).  This additional component is elegantly constructed via the martingale representation theorem.

Since  control process $v$ is required to be adapted to $\mathbb{F}^{Y^v,\theta}$, the solution of the adjoint equation should also be $\mathbb{F}^{Y^v,\theta}$-adapted. However, the observation process $Y^v$ is not a Brownian motion, which makes it hard to construct the desired stochastic integral representation directly through the martingale representation theorem in terms of $Y^v$. Recall from Theorem \ref{the:filter of X}  that  innovation process $V$ is an $\mathbb{F}^{Y^v,\theta}$-adapted Brownian motion. This naturally raises the question of whether an $\mathbb{F}^{Y^v,\theta}$-martingale admits a stochastic integral representation with respect to $V$. To explore this, define the filtration $\mathbb{F}^{V}=\{\mathcal{F}^{V}_t\}_{t\in[0,T]}
$, where $\mathcal{F}^{V}_t=\sigma\left\{V_s;s\in[0,t]\right\}$. From (\ref{equ:def of V}), it follows immediately that $\mathbb{F}^{V}\subseteq\mathbb{F}^{Y^v,\theta}$. If the reverse inclusion also holds, then $\mathbb{F}^{V}$ and $\mathbb{F}^{Y^v,\theta}$ would coincide, allowing us to apply the martingale representation theorem directly to obtain the desired stochastic integral representation of any $\mathbb{F}^{Y^v,\theta}$-martingale with respect to $V$. However, no general result appears to establish this equivalence between the two filtrations.

To overcome this difficulty, inspired by \cite[Theorem 3.1]{Fujisaki Kallianpur Kunita 1972} and \cite[Lemma 2.1]{Yao Zhang Zhou 2006}, we can show that any $\mathbb{F}^{Y^v,\theta}$-martingale has a stochastic integral representation with respect to $V$ and $M$ through an appropriate measure transformation approach, which yields the following result.

\begin{Theorem}
	\label{The: stoc inte repr}
	Suppose that $\kappa$ is an $\mathbb{R}^n$-valued square integrable  martingale with respect to $\mathbb{F}^{Y^v,\theta}$ under $\mathbb{P}$, then there exist  unique processes $\alpha\in \mathcal{L}^2_{\mathbb{F}^{Y^v,\theta}}([0,T];\mathbb{R}^{n\times r})$ and $\beta\in\mathcal{M}^2_{\mathbb{F}^{Y^v,\theta}}([0,T];\mathbb{R}^{n})$  such that, for any $t\in[0,T]$, 
	\begin{equation*}
		\label{ass: con of repre}
		\begin{aligned}
			\kappa_t=\kappa_0+\int_{0}^{t}\alpha_s\mathrm{d}V_s+\int_{0}^{t}\beta_s\bullet\mathrm{d}M_s, \hspace{0.2cm} \mathbb{P}-\text{a.s.}
		\end{aligned}
	\end{equation*}
\end{Theorem}

\begin{proof}
	From Lemma \ref{lem:filtr ident} and  relation (\ref{equ: rela betw N and V}),  the proof can be reduced to verifying the following martingale representation property: for any $\mathbb{R}^n$-valued square integrable  martingale $\kappa$ with respect to $\mathbb{F}^{\lambda,\theta}$ under $\mathbb{P}$, there exist  unique processes $\alpha\in \mathcal{L}^2_{\mathbb{F}^{\lambda,\theta}}([0,T];\mathbb{R}^{n\times r})$ and $\beta\in\mathcal{M}^2_{\mathbb{F}^{\lambda,\theta}}([0,T];\mathbb{R}^{n})$  such that
	\begin{equation*}
		\label{equ:equi obje}
		\begin{aligned}
			\kappa_t=\kappa_0+\int_{0}^{t}\alpha_s\mathrm{d}N_s+\int_{0}^{t}\beta_s\bullet\mathrm{d}M_s, \hspace{0.2cm} \mathbb{P}-\text{a.s.}
		\end{aligned}
	\end{equation*}
	
	We first introduce a stochastic exponential process $\mathscr{E}=\{\mathscr{E}_t\}_{t\in[0,T]}$ defined as $	\mathscr{E}_t=\mathrm{exp}\{-\int_{0}^{t}\hbar_s^{\top}\mathrm{d}N_s-\frac{1}{2}\int_{0}^{t}|\hbar_s|^2\mathrm{d}s\}$, where $\hbar_s= F\mathbb{E}^{\lambda,\theta}_{s}[\eta_s]+\widehat{F}\mathbb{E}^{\theta}_{s}[\eta_s]+f$ and the process $N$  is defined in (\ref{equ:def N}).
	
	Without loss of generality, we assume that $\kappa_0=0$ for simplicity. Define a sequence of stopping times $\{T_n\}_{n\in\mathbb{N}}$ via 
	\begin{equation*}
		\begin{aligned}
			T_n=\inf\left\{t\geq 0; \int_{0}^{t}|\hbar_s|^2\mathrm{d}s>n \quad \text{or}\quad  |\mathscr{E}_t|>n \right\}\land T.
		\end{aligned}
	\end{equation*}By the definition of stopping times, the stopped process $\mathscr{E}^n=\{\mathscr{E}_{t\land T_n}\}_{t\in[0,T]}$ is a uniformly integrable martingale under $\mathbb{P}$. Accordingly, we  construct a sequence of probability measures  $\{\mathbb{Q}^n\}_{n\in\mathbb{N}}$ equivalent to $\mathbb{P}$ by the Radon-Nikodym derivative $\mathrm{d}\mathbb{Q}^n=\mathscr{E}_{T_n}\mathrm{d}\mathbb{P}$. By Girsanov theorem for Markov-modulated stochastic processes, we derive that the process $\lambda^{n}=\{\lambda^n_t\}_{t\in[0,T]}$  defined by 
	\begin{equation*}
		\begin{aligned}
			\lambda^n_t=N_t+\int_{0}^{t\land T_n}\hbar_s\mathrm{d}s
		\end{aligned}
	\end{equation*}is an $\mathbb{R}^r$-valued standard Brownian motion under $\mathbb{Q}^n$. Moreover, the dynamic evolution of the underlying Markov chain $\theta$ remains unchanged under the measure change, and $\lambda^{n}$ is independent of $\theta$ under $\mathbb{Q}^n$;  see \cite[Lemma 2.1]{Yao Zhang Zhou 2006}.

	Denote the filtration generated by $\lambda^n$ and $\theta$ as  $\mathcal{F}^{\lambda^n,\theta}_t=\sigma\{\lambda^n_s,\theta_s;s\in[0,t]\}$. Let  $\widetilde{\kappa}$ be an arbitrary  square integrable $\mathbb{F}^{\lambda^n,\theta}$-martingale  under $\mathbb{Q}^n$. Applying Theorem B.4.6. in \cite{Donnelly 2008}, we conclude that there exist unique processes $\widetilde{\alpha}^{n}\in \mathcal{L}^2_{\mathbb{F}^{\lambda^n,\theta}}([0,T];\mathbb{R}^{n\times r})$ and $\widetilde{\beta}^{n}\in\mathcal{M}^2_{\mathbb{F}^{\lambda^n,\theta}}([0,T];\mathbb{R}^{n})$ such that 
	\begin{equation*}
		\label{ass:repre of tilde k}
		\begin{aligned}
			\widetilde{\kappa}_t=\int_{0}^{t}\widetilde{\alpha}^{n}_s\mathrm{d}\lambda^n_s+\int_{0}^{t}\widetilde{\beta}^{n}_s\bullet\mathrm{d}M_s,\hspace{0.5cm}\mathbb{Q}^n-\text{a.s.}
		\end{aligned}
	\end{equation*}We further introduce the stopped filtration $\mathbb{F}^{\lambda^n,\theta}_{T_n}=\{\mathcal{F}^{\lambda^n,\theta}_{t \land T_n}\}_{t\in[0,T]}$,  where
	\begin{equation}
		\label{equ: def of  F lambda  n stop}
		\begin{aligned}
			\mathcal{F}^{\lambda^n,\theta}_{t \land T_n}=\{\mathscr{A}\in\mathcal{F}^{\lambda^n,\theta}_t; \mathscr{A}\cap \{T_n\leq s\} \in\mathcal{F}^{\lambda^n,\theta}_s, \forall s\in[0,T] \}.
		\end{aligned}
	\end{equation}

	Combining the optional sampling theorem  \cite[Theorem 16, p.~10]{Protter 2005} and the fact that $\lambda^n_s=\lambda_s$ for all $s\in[0,T_n]$, we obtain
	\begin{equation*}
		\begin{aligned}
			\widetilde{\kappa}_{t\land T_n}=\int_{0}^{t\land T_n}\widetilde{\alpha}^{n}_s\mathrm{d}\lambda_s+\int_{0}^{t\land T_n}\widetilde{\beta}^{n}_s\bullet\mathrm{d}M_s.
		\end{aligned}
	\end{equation*}The process $\widetilde{\kappa}^n=\{\widetilde{\kappa}_{t\land T_n}\}_{t\in[0,T]}$ is a square integrable martingale with respect to $\mathbb{F}^{\lambda^n,\theta}_{T_n}$ under $\mathbb{Q}^n$. It is worth emphasizing that the original filtration $\mathbb{F}^{\lambda,\theta}$ differs from $\mathbb{F}^{\lambda^n,\theta}$, which hinders the direct derivation of the target representation for $\mathbb{F}^{\lambda,\theta}$-martingale. To bridge this gap, we define the stopped version of the original filtration $\mathbb{F}^{\lambda,\theta}_{T_n}=\{\mathcal{F}^{\lambda,\theta}_{t \land T_n}\}_{t\in[0,T]}$, where $\mathcal{F}^{\lambda,\theta}_{t \land T_n}$ is defined as \eqref{equ: def of  F lambda  n stop} with    $\mathcal{F}^{\lambda^n,\theta}_t$ replaced by $\mathcal{F}^{\lambda,\theta}_t$.

	Considering that $\lambda^n_s=\lambda_s$ for any $s\in[0,T_n]$ and utilizing Lemma A.24 in \cite{Bain Crisan 2009}, we verify that 
	\begin{equation}
		\label{equ: iden fo two fil}
		\begin{aligned}
			\mathcal{F}^{\lambda^n,\theta}_{t \land T_n}=& \sigma\{\lambda^n_{s \land T_n},\theta_{s \land T_n};s\in[0,t]\}
			= \sigma\{\lambda_{s \land T_n},\theta_{s \land T_n};s\in[0,t]\}
			= \mathcal{F}^{\lambda,\theta}_{t \land T_n}.
		\end{aligned}
	\end{equation}In view of identity \eqref{equ: iden fo two fil}, it follows that, if $\widetilde{\kappa}=\{\widetilde{\kappa}_t\}_{t\in[0,T]}$ is a square integrable martingale with respect to  $\mathbb{F}^{\lambda, \theta}_{T_n}$ under $\mathbb{Q}^n$, then it admits the representation 
	\begin{equation}
		\label{equ: stop mart}
		\begin{aligned}
			\widetilde{\kappa}_t=\widetilde{\kappa}_{t\land T_n}=\int_{0}^{t\land T_n}\widetilde{\alpha}^{n}_s\mathrm{d}\lambda_s+\int_{0}^{t\land T_n}\widetilde{\beta}^{n}_s\bullet\mathrm{d}M_s.
		\end{aligned}
	\end{equation}

	Now consider an arbitrary square integrable $\mathbb{F}^{\lambda,\theta}$-martingale $\kappa$    under $\mathbb{P}$. Let $\widetilde{\kappa}_t=\mathscr{E}_t^{-1}\kappa_t$. By Lemma 7.1.3 in \cite{Kallianpur 1980}, the process $\widetilde{\kappa}^n=\{\widetilde{\kappa}_{t\land T_n}\}_{t\in[0,T]}$ is a square integrable martingale with respect to $\mathbb{F}^{\lambda,\theta}_{T_n}$ under $\mathbb{Q}^n$. Substituting  identity \eqref{equ:def N} into \eqref{equ: stop mart},  we have 
	\begin{equation*}
		\begin{aligned}
			\widetilde{\kappa}_{t\land T_n}
			= \int_{0}^{t\land T_n}\widetilde{\alpha}^{n}_s\mathrm{d}N_s+\int_{0}^{t\land T_n}\widetilde{\alpha}^{n}_s\hbar_s\mathrm{d}s +\int_{0}^{t\land T_n}\widetilde{\beta}^{n}_s\bullet\mathrm{d}M_s.
		\end{aligned}
	\end{equation*}We then apply It$\widehat{\text{o}}$'s formula to $\mathscr{E}_t\widetilde{\kappa}_t$. After straightforward stochastic calculus computations, we obtain  
	\begin{equation}
		\label{equ:repre of k stop}
		\begin{aligned}
			\kappa_{t\land T_n}=&\  \int_{0}^{t\land T_n}\alpha^n_s\mathrm{d}N_s+\int_{0}^{t \land T_n}\beta^n_s\bullet\mathrm{d}M_s,
		\end{aligned}
	\end{equation}where $\alpha^n_s=\mathscr{E}_s\left(\widetilde{\alpha}_s^n-\widetilde{\kappa}_s\hbar_s^{\top}\right)$ and $\beta^n_s=\mathscr{E}_s\widetilde{\beta}^n_s$. These two processes satisfy 
	\begin{equation*}
		\begin{aligned}
			\mathbb{E}\left[\int_{0}^{t}|\mathbb{I}_{[0,T_n]}(s)\alpha_s^n|^2\mathrm{d}s\right]
			+  \sum_{\ell,j=1}^{D}\mathbb{E}\left[ \int_{0}^{t}|\mathbb{I}_{[0,T_n]}(s)(\beta^n_s)^{\ell j}|^2\mathrm{d}[M^{\ell j}]_s\right]
			= \mathbb{E}|\kappa_{t\land T_n}|^2
			< \infty.
		\end{aligned}
	\end{equation*}
	
	We finally discuss the uniqueness and consistency of integrands across different stopping indices. Based on the uniqueness of martingale representation in \eqref{equ:repre of k stop}, for any $n,m\in\mathbb{N}$ and  $n<m$, we have  $\alpha^n_s=\alpha^m_s$, $\mathrm{d}s\times \mathrm{d}\mathbb{P}$- a.e. on $[0,T_n]$, and $\beta^n_s=\beta^m_s$, $\mathrm{d}\left\langle M \right\rangle_s\times \mathrm{d}\mathbb{P}$-a.e. on $[0,T_n]$. This implies that there exist square  integrable   processes $\alpha_s$ and 
	$\beta_s$ adapted to $\mathcal{F}^{\lambda,\theta}_s$ satisfying $\alpha_s=\alpha^n_s$ and $\beta_s=\beta^n_s$ for $s\in[0,T_n]$ and arbitrary  $n\in\mathbb{N}$. Taking the limit  $n\rightarrow \infty$ in \eqref{equ:repre of k stop}, we obtain the desired result.
\end{proof}

Suppose that $u^{*}$ is an optimal control of \textbf{(FO-CMF)}, $\mathbb{X}^{*}$ is the optimal trajectory and $Y^{*}$ is the associated observation process. We introduce the adjoint equation as follows, which is a linear conditional mean-field BSDE with regime-switching,
\begin{equation}
	\label{equ:adj equ}
	\left\{\begin{aligned}
		\mathrm{d}p_t=&\   -\left[A^{\top}_{(\theta_t,t)}p_t+\widehat{A}^{\top}_{(\theta_t,t)}\mathbb{E}^{\theta}_t[p_t]+H_{(\theta_t,t)}\mathbb{X}^{*}_t +\widehat{H}_{(\theta_t,t)}\mathbb{E}^{\theta}_t[\mathbb{X}^{*}_t]+G^{\top}_{(\theta_t,t)}u^{*}_t +\widehat{G}^{\top}_{(\theta_t,t)}\mathbb{E}^{\theta}_t[u^{*}_t]\right]\mathrm{d}t\\
		&\ +q_t\mathrm{d}V_t+k_t\bullet\mathrm{d}M_t,\\
		p_T=\ &S_{\theta_T}\mathbb{X}^{*}_T+\widehat{S}_{\theta_T}\mathbb{E}^{\theta}_T[\mathbb{X}^{*}_T].
	\end{aligned}\right.
\end{equation}

Regarding the adjoint equation, we have the following result.
\begin{Lemma}
	\label{lem: well of adj equ}
	Under assumptions \ref{ass:coef in dyna} and \ref{ass: coe in cost one},  BSDE (\ref{equ:adj equ}) admits a unique solution $(p,q,k)\in\mathcal{S}^2_{\mathbb{F}^{Y^*,\theta}}([0,T];\mathbb{R}^{n})\times\mathcal{L}^2_{\mathbb{F}^{Y^*,\theta}}([0,T];\mathbb{R}^{n\times r})\times\mathcal{M}^2_{\mathbb{F}^{Y^*,\theta}}([0,T];\mathbb{R}^{n})$.
\end{Lemma}
\begin{proof}
	For simplicity, we use $g(t,p,q,\widehat{p},\widehat{q})= A^{\top}_{(\theta_t,t)}p+\widehat{A}^{\top}_{(\theta_t,t)}\widehat{p}+H_{(\theta_t,t)}\mathbb{X}^{*}_t+\widehat{H}_{(\theta_t,t)}\mathbb{E}^{\theta}_t[\mathbb{X}^{*}_t]+G^{\top}_{(\theta_t,t)}u^{*}_t+\widehat{G}^{\top}_{(\theta_t,t)}\mathbb{E}^{\theta}_t[u^{*}_t]$ and $\rho=S_{\theta_T}\mathbb{X}^{*}_T+\widehat{S}_{\theta_T}\mathbb{E}^{\theta}_T[\mathbb{X}^{*}_T]$ to denote the generator and terminal value of equation (\ref{equ:adj equ}), respectively. Under assumptions  \ref{ass:coef in dyna} and \ref{ass: coe in cost one}, it can be verified that $g(t,p,q,\widehat{p},\widehat{q})$ is adapted to $\mathcal{F}^{Y^*,\theta}_t$,  uniformly Lipschitz with respect to $(p,q,\widehat{p},\widehat{q})$, and satisfies $g(t,0,0,0,0)\in\mathcal{L}^2_{\mathbb{F}^{Y^*,\theta}}([0,T];\mathbb{R}^n)$. Moreover, $\rho$ is $\mathcal{F}^{Y^*,\theta}_T$-measurable and square integrable.

	For given processes $\widetilde{p}\in\mathcal{S}^2_{\mathbb{F}^{Y^*,\theta}}([0,T];\mathbb{R}^{n})$ and  $\widetilde{q}\in\mathcal{L}^2_{\mathbb{F}^{Y^*,\theta}}([0,T];\mathbb{R}^{n\times r})$. Introduce process $G=\{G_t\}_{t\in[0,T]}$ by $G_t=\mathbb{E}^{Y^*,\theta}_t[\rho+\int_{0}^{T}g(t,\widetilde{p}_t,\widetilde{q}_t,\mathbb{E}^{\theta}_t[\widetilde{p}_t],\mathbb{E}^{\theta}_t[\widetilde{q}_t])\mathrm{d}t]$. Obviously, $G$ is  a  square integrable $\mathbb{F}^{Y^*,\theta}$-martingale. It follows from Theorem \ref{The: stoc inte repr} that there exist processes $q\in\mathcal{L}^2_{\mathbb{F}^{Y^*,\theta}}([0,T];\mathbb{R}^{n\times r})$ and $k\in\mathcal{M}^2_{\mathbb{F}^{Y^*,\theta}}([0,T];\mathbb{R}^{n})$ such that $G_t=G_0+\int_{0}^{t}q_s\mathrm{d}V_s+\int_{0}^{t}k_s\bullet\mathrm{d}M_s$.

	Let $	p_t=G_t-\int_{0}^{t}g(s,\widetilde{p}_s,\widetilde{q}_s,\mathbb{E}^{\theta}_s[\widetilde{p}_s],\mathbb{E}^{\theta}_s[\widetilde{q}_s])\mathrm{d}s$. Then we have
	\begin{equation}
		\label{equ:exis of p}
		\begin{aligned}
			p_t=&\ G_0+\int_{0}^{t}q_s\mathrm{d}V_s+\int_{0}^{t}k_s\bullet\mathrm{d}M_s -\int_{0}^{t}g\left(s,\widetilde{p}_s,\widetilde{q}_s,\mathbb{E}^{\theta}_s[\widetilde{p}_s],\mathbb{E}^{\theta}_s[\widetilde{q}_s]\right)\mathrm{d}s\\
			=&\ \rho+\int_{t}^{T}g\left(s,\widetilde{p}_s,\widetilde{q}_s,\mathbb{E}^{\theta}_s[\widetilde{p}_s],\mathbb{E}^{\theta}_s[\widetilde{q}_s]\right)\mathrm{d}s -\int_{t}^{T}q_s\mathrm{d}V_s-\int_{t}^{T}k_s\bullet\mathrm{d}M_s.
		\end{aligned}
	\end{equation}By virtue of (\ref{equ:exis of p}), we can constitute a mapping $(p,q,k)=\varTheta((\widetilde{p},\widetilde{q},\widetilde{k}))$ from Banach space $\mathcal{S}^2_{\mathbb{F}^{Y^*,\theta}}([0,T];\mathbb{R}^{n})\times\mathcal{L}^2_{\mathbb{F}^{Y^*,\theta}}([0,T];\mathbb{R}^{n\times r})\times\mathcal{M}^2_{\mathbb{F}^{Y^*,\theta}}([0,T];\mathbb{R}^{n})$ to itself. Following the argument similar to those in \cite[Theorem 3.4]{Nguyen 2020}, we can show that $\varTheta$ is a contraction mapping. Therefore, by the Banach fixed-point theorem, there exists a unique solution of (\ref{equ:adj equ}).
\end{proof}

With the help of adjoint equation (\ref{equ:adj equ}), a  necessary  and sufficient condition of optimality is provided. 
\begin{Lemma}
	\label{the:openloop char}
	Under assumptions \ref{ass:coef in dyna}, \ref{ass: coe in cost one}, and \ref{ass: coe in cost two}, if $u^{*}$ is an optimal control of \textbf{(FO-CMF)}, $\mathbb{X}^{*}$ is the optimal trajectory, then the following maximum condition holds:
	\begin{equation}
		\label{equ:coupl condi}
		\begin{aligned}
			&\ G_{(\theta_t,t)}\mathbb{X}^{*}_t+\widehat{G}_{(\theta_t,t)}\mathbb{E}^{\theta}_t[\mathbb{X}^{*}_t]+B^{\top}_{(\theta_t,t)}p_t+\widehat{B}^{\top}_{(\theta_t,t)}\mathbb{E}^{\theta}_t[p_t]
			+ R_{(\theta_t,t)}u^{*}_t+\widehat{R}_{(\theta_t,t)}\mathbb{E}^{\theta}_t[u^{*}_t]=0,\hspace{1cm} t\in[0,T],
		\end{aligned}
	\end{equation}where $(p,q,k)$ is the solution of (\ref{equ:adj equ}).
	
	Conversely, if $(\mathbb{X}^*,u^{*},p,q,k)$ satisfies  (\ref{equ: rewr dyna}), (\ref{equ:adj equ}), and (\ref{equ:coupl condi}), then $u^{*}$ is an optimal control of \textbf{(FO-CMF)}.
\end{Lemma}

\begin{proof}
	From Theorem \ref{the: equi betw two pro} and Lemma \ref{lem:equi bet con on two sets}, we derive that $u^*$ also minimizes  $J^{F}\left(\xi,\ell_0;v\right)$ over $\mathcal{U}$. For any admissible control $\widetilde{v}\in\mathcal{U}$, introduce the following variation  equation:
	\begin{equation*}
		\left\{\begin{aligned}
			\mathrm{d}\widetilde{\mathbb{X}}_t=&\ \left[A\widetilde{\mathbb{X}}_t+\widehat{A}\mathbb{E}^{\theta}_t[\widetilde{\mathbb{X}}_t]+B\widetilde{v}_t+\widehat{B}\mathbb{E}^{\theta}_t[\widetilde{v}_t]\right]\mathrm{d}t,\\
			\widetilde{\mathbb{X}}_0=&\ 0.
		\end{aligned}\right.
	\end{equation*}Then we can derive that the G$\hat{\text{a}}$teaux derivative  of $J^{F}\left(\xi,\ell_0;\cdot\right)$ at $u^{*}$   along the direction  $\widetilde{v}$ is 
	\begin{equation}
		\label{equ:vari inequ}
		\begin{aligned}
			 \frac{\mathrm{d}J^{F}\left(\xi,\ell_0;u^{*}+\epsilon \widetilde{v}\right)}{\mathrm{d}\epsilon}\Big|_{\epsilon=0}
			=&\ \mathbb{E}\left\{\int_{0}^{T}\left[\left\langle\widetilde{\mathbb{X}}, H\mathbb{X}^{*}+G^{\top}u^{*}\right\rangle+\left\langle\widetilde{v}, G\mathbb{X}^{*}+Ru^{*}\right\rangle +\left\langle\mathbb{E}^{\theta}_t[\widetilde{\mathbb{X}}], \widehat{H}\mathbb{E}^{\theta}_t[\mathbb{X}^{*}]+\widehat{G}^{\top}\mathbb{E}^{\theta}_t[u^{*}]\right\rangle\right.\right.\\
			&\ \left.\left.+\left\langle\mathbb{E}^{\theta}_t[\widetilde{v}], \widehat{G}\mathbb{E}^{\theta}_t[\mathbb{X}^{*}]+\widehat{R}\mathbb{E}^{\theta}_t[u^{*}]\right\rangle\right]\mathrm{d}t +\left\langle\widetilde{\mathbb{X}}_T, S_{\theta_T}\mathbb{X}^{*}_T\right\rangle+\left\langle\mathbb{E}^{\theta}_T[\widetilde{\mathbb{X}}_T], \widehat{S}_{\theta_T}\mathbb{E}^{\theta}_T[\mathbb{X}^{*}_T]\right\rangle\right\}.
		\end{aligned}
	\end{equation}

	In view of  Lemma \ref{lem: well of adj equ}, applying It$\widehat{\text{o}}$'s formula for semi-martingale  to $\langle\widetilde{\mathbb{X}}_t,p_t\rangle$ and taking expectation, we obtain  
	\begin{equation}
		\label{equ:adjo ident}
		\begin{aligned}
			&\ \mathbb{E}\left\{\left\langle\widetilde{\mathbb{X}}_T, S_{\theta_T}\mathbb{X}^{*}_T\right\rangle+\left\langle\mathbb{E}^{\theta}_T[\widetilde{\mathbb{X}}_T], \widehat{S}_{\theta_T}\mathbb{E}^{\theta}_T[\mathbb{X}^{*}_T]\right\rangle\right\}\\
			=&\ \mathbb{E}\left\{\int_{0}^{T}\left[-\left\langle\widetilde{\mathbb{X}}, H\mathbb{X}^{*}+G^{\top}u^{*}\right\rangle+\left\langle\mathbb{E}^{\theta}_t[\widetilde{v}], \widehat{B}^{\top}p\right\rangle-\left\langle\mathbb{E}^{\theta}_t[\widetilde{\mathbb{X}}], \widehat{H}\mathbb{E}^{\theta}_t[\mathbb{X}^{*}]+\widehat{G}^{\top}\mathbb{E}^{\theta}_t[u^{*}]\right\rangle+\left\langle \widetilde{v},B^{\top}p\right\rangle\right]\mathrm{d}t\right\}.
		\end{aligned}
	\end{equation}Substituting (\ref{equ:adjo ident}) into (\ref{equ:vari inequ}) yields
	\begin{equation}
		\label{equ:zero vari}
		\begin{aligned}
			\frac{\mathrm{d}J^{F}\left(\xi,\ell_0;u^{*}+\epsilon \widetilde{v}\right)}{\mathrm{d}\epsilon}\Big|_{\epsilon=0}
			= \mathbb{E}\left\{\int_{0}^{T}\left\langle\widetilde{v}, G\mathbb{X}^{*}+\widehat{G}\mathbb{E}^{\theta}_t[\mathbb{X}^{*}]+B^{\top}p+\widehat{B}^{\top}\mathbb{E}^{\theta}_t[p]+Ru^{*}+\widehat{R}\mathbb{E}^{\theta}_t[u^{*}]\right\rangle \mathrm{d}t\right\}.
		\end{aligned}
	\end{equation}Since $u^{*}$ is an optimal  control, the G$\hat{\text{a}}$teaux derivative  of cost functional  is $0$. Therefore, combining (\ref{equ:zero vari}) and the arbitrariness of $\widetilde{v}$ gives identity  (\ref{equ:coupl condi}). 
	
	To prove  sufficiency, we observe  from (\ref{equ:zero vari}) that if $(\mathbb{X}^*,u^{*},p,q,k)$ satisfies  (\ref{equ: rewr dyna}), (\ref{equ:adj equ}), and (\ref{equ:coupl condi}), then the G$\hat{\text{a}}$teaux derivative  of the  cost functional  is $0$. Moreover, under assumptions \ref{ass: coe in cost one} and \ref{ass: coe in cost two}, the cost functional is uniformly convex. Therefore, $u^{*}$ is an optimal control.
\end{proof}

Lemma \ref{the:openloop char} provides a necessary and sufficient condition for optimality. Specifically, equations (\ref{equ: rewr dyna}), (\ref{equ:adj equ}), and (\ref{equ:coupl condi})  are referred to as a Hamiltonian system, which is  a  fully coupled conditional mean-field forward and backward stochastic differential equation (FBSDE),  with the coupling structure derived from (\ref{equ:coupl condi}). However, solving such an  FBSDE is often computationally demanding and may be impractical for implementation. In many applications, it is preferable to express the optimal control in a feedback form with respect to the state process. To this end, we obtain the following result.

\begin{Theorem}\label{lem:opti cont of CMF-SDE}
	Suppose that  assumptions \ref{ass:coef in dyna}, \ref{ass: coe in cost one}, and \ref{ass: coe in cost two} are fulfilled.  	Consider the closed-loop system
	\begin{equation}\label{equ: optimal state}
		\left\{\begin{aligned}
			\mathrm{d}\mathbb{X}^*_t=&\ \left[\left(A_{(\theta_t,t)}-B_{(\theta_t,t)}R^{-1}_{(\theta_t,t)}\widetilde{B}_{(\theta_t,t)}\right)\left(\mathbb{X}^*_t-\mathbb{E}^{\theta}_t[\mathbb{X}^*_t]\right) +\left(\check{A}_{(\theta_t,t)}-\check{B}_{(\theta_t,t)}\check{R}^{-1}_{(\theta_t,t)}\widetilde{\check{B}}_{(\theta_t,t)}\right)\mathbb{E}^{\theta}_t[\mathbb{X}^*_t]\right.\\
			&\ \left. +b_{(\theta_t,t)}-\check{B}_{(\theta_t,t)}\check{R}^{-1}_{(\theta_t,t)}\mathbb{B}_{(\theta_t,t)}\right]\mathrm{d}t +Q_t\mathrm{d}V_t,\\
			\mathbb{X}^*_0=&\ \mu,
		\end{aligned}\right.
	\end{equation}where 
	\begin{equation*}
		\begin{aligned}
			\widetilde{B}_{(\ell,t)}=&\ B^{\top}_{(\ell,t)}\Lambda_{(\ell,t)}+G_{(\ell,t)},\\ 
			\widetilde{\check{B}}_{(\ell,t)}=&\ \check{B}^{\top}_{(\ell,t)}\Gamma_{(\ell,t)}+\check{G}_{(\ell,t)},\\
			\mathbb{B}_{(\ell,t)}=&\ \check{B}^{\top}_{(\ell,t)}\phi_{(\ell,t)},
		\end{aligned}
	\end{equation*}and $\Lambda_{(\ell,t)}$, $\Gamma_{(\ell,t)}$, and $\phi_{(\ell,t)}$ are solutions of  
	\begin{equation}
		\label{equ: ricca 1}
		\left\{\begin{aligned}
			&\dot{\Lambda}_{(\ell,t)}+A^{\top}_{(\ell,t)}\Lambda_{(\ell,t)}+\Lambda_{(\ell,t)}A_{(\ell,t)}+H_{(\ell,t)}-\widetilde{B}_{(\ell,t)}^{\top}R_{(\ell,t)}^{-1}\widetilde{B}_{(\ell,t)}+\sum_{j=1}^{D}\pi^{\ell j}\Lambda_{j}=0,\\
			&\Lambda_{(\ell,T)}=S_{\ell},
		\end{aligned}\right.
	\end{equation}
	\begin{equation}
		\label{equ: ricca 2}
		\left\{\begin{aligned}
			&\dot{\Gamma}_{(\ell,t)}+\check{A}_{(\ell,t)}^{\top}\Gamma_{(\ell,t)}+\Gamma_{(\ell,t)}\check{A}_{(\ell,t)}+\check{H}_{(\ell,t)}-	\widetilde{\check{B}}_{(\ell,t)}^{\top}\check{R}_{(\ell,t)}^{-1}	\widetilde{\check{B}}_{(\ell,t)}+\sum_{j=1}^{D}\pi^{\ell j}\Gamma_{j}=0,\\
			&\Gamma_{(\ell,T)}=\check{S}_{\ell},
		\end{aligned}\right.
	\end{equation}and
	\begin{equation}
		\label{equ: first order ode}
		\left\{\begin{aligned}
			&\dot{\phi}_{(\ell,t)}+\Big(\check{A}_{(\ell,t)}^{\top}-	\widetilde{\check{B}}_{(\ell,t)}^{\top}\check{R}_{(\ell,t)}^{-1} \check{B}_{(\ell,t)}^{\top}\Big)\phi_{(\ell,t)}+\Gamma_{(\ell,t)}b_{(\ell,t)}+\sum_{j=1}^{D}\pi^{\ell j}\phi_{j}=0,\\
			&\phi_{(\ell,T)}=0,
		\end{aligned}\right.
	\end{equation}respectively.  Let
	\begin{equation}
		\label{equ: opti cont of JLQ}
		\begin{aligned}
			u^{*}_t=&\ -R_{(\theta_t,t)}^{-1}\widetilde{B}_{(\theta_t,t)}\left(\mathbb{X}^{*}_t-\mathbb{E}^{\theta}_t[\mathbb{X}^{*}_t]\right) -\check{R}_{(\theta_t,t)}^{-1}\left(	\widetilde{\check{B}}_{(\theta_t,t)}\mathbb{E}^{\theta}_t[\mathbb{X}^{*}_t]+\mathbb{B}_{(\theta_t,t)}\right).
		\end{aligned}	
	\end{equation}Then $(\mathbb{X}^*,u^{*})$ is the optimal pair of \textbf{(FO-CMF)}. Moreover, the optimal cost  is 
	\begin{equation}
		\label{equ:opti ccost of CMF-LQ}
		\begin{aligned}
			\mathscr{V}^F\left(\xi,\ell_0\right)=&\ \frac{1}{2}\mathbb{E}\left\{\int_{0}^{T}\left[-|\check{R}^{-\frac{1}{2}}\mathbb{B}|^2+\mathrm{tr}\left(Q^{\top}\Lambda Q\right) +2\left\langle\phi,b\right\rangle\right]\mathrm{d}t+\left\langle \mu,\Gamma_{(\ell_0,0)}\mu+2\phi_{(\ell_0,0)}\right\rangle\right\}.	
		\end{aligned}
	\end{equation}
\end{Theorem}

\begin{proof}Inspired by the terminal value  $p_T$, we make the ansatz
	\begin{equation}
		\label{equ: ansatz of p}
		\begin{aligned}
			p_t=\Lambda_{(\theta_t,t)}\left(\mathbb{X}^{*}_t-\mathbb{E}^{\theta}_t[\mathbb{X}^{*}_t]\right)+\Gamma_{(\theta_t,t)}\mathbb{E}^{\theta}_t[\mathbb{X}^{*}_t]+\phi_{(\theta_t,t)},
		\end{aligned}
	\end{equation}where $\Lambda_{(\ell,t)}$, $\Gamma_{(\ell,t)}$, and $\phi_{(\ell,t)}$ are deterministic functions satisfying boundary conditions $\Lambda_{(\ell,T)}=S_{\ell}$, $ \Gamma_{(\ell,T)}=\check{S}_{\ell}$, and  $\phi_{(\ell,T)}=0$, for any $\ell\in\mathbb{D}$.

	Taking $\mathbb{E}^{\theta}_t[\cdot]$ on both sides of (\ref{equ: rewr dyna}) and utilizing conclusion \ref{asser two in main filter}  in Theorem  \ref{the:filter of X}, we obtain  
	\begin{equation}
		\label{equ:sde for condi x}
		\left\{\begin{aligned}
			\mathrm{d}\mathbb{E}^{\theta}_t[\mathbb{X}^{*}_t]=&\ \left[\check{A}_{(\theta_t,t)}\mathbb{E}^{\theta}_t[\mathbb{X}^{*}_t]+\check{B}_{(\theta_t,t)}\mathbb{E}^{\theta}_t[u^{*}_t]+b_{(\theta_t,t)}\right]\mathrm{d}t,\\
			\mathbb{E}^{\theta}_{0}[\mathbb{X}^{*}_0]=&\ \mu.
		\end{aligned}\right.
	\end{equation}Utilizing (\ref{equ:sde for condi x}) and  applying It$\widehat{\text{o}}$'s formula for Markov-modulated process to the right-hand  side of (\ref{equ: ansatz of p}), it follows that 
	\begin{equation}
		\label{equ:ito for ana}
		\begin{aligned}
			\mathrm{d}p_t
			=&\ \Big[\Big(\dot{\Lambda}_{(\theta_t,t)}+\sum_{j=1}^{D}\pi^{\theta_t j}\Lambda_{(j,t)}\Big)\left(\mathbb{X}^{*}_t-\mathbb{E}^{\theta}_{t}[\mathbb{X}^{*}_t]\right) +\Lambda_{(\theta_t,t)} A_{(\theta_t,t)}\left(\mathbb{X}^{*}_t-\mathbb{E}^{\theta}_{t}[\mathbb{X}^{*}_t]\right)\\
			&\ +\Lambda_{(\theta_t,t)} B_{(\theta_t,t)}\left(u^{*}_t-\mathbb{E}^{\theta}_{t}[u^{*}_t]\right)+	\Big(\dot{\Gamma}_{(\theta_t,t)}+\sum_{j=1}^{D}\pi^{\theta_t j}\Gamma_{(j,t)}\Big)\mathbb{E}^{\theta}_{t}[\mathbb{X}^{*}_t] +\Gamma_{(\theta_t,t)}\check{A}_{(\theta_t,t)}\mathbb{E}^{\theta}_{t}[\mathbb{X}^{*}_t]\\
			&\ +\Gamma_{(\theta_t,t)}\check{B}_{(\theta_t,t)}\mathbb{E}^{\theta}_{t}[u^{*}_t]+\Gamma_{(\theta_t,t)} b_{(\theta_t,t)}+\dot{\phi}_{(\theta_t,t)}+\sum_{j=1}^{D}\pi^{\theta_t j}\phi_{(j,t)}\Big]\mathrm{d}t +\Lambda_{(\theta_t,t)} Q_t\mathrm{d}V_t\\
			&\ +\sum_{\ell,j=1}^{D}\left[\left(\Lambda_{(\ell,t)}-\Lambda_{(j,t)}\right)\left(\mathbb{X}^{*}_t-\mathbb{E}^{\theta}_{t}[\mathbb{X}^{*}_t]\right) +\left(\Gamma_{(\ell,t)}-\Gamma_{(j,t)}\right)\mathbb{E}^{\theta}_{t}[\mathbb{X}^{*}_t]+\phi_{(\ell,t)}-\phi_{(j,t)}\right]\mathrm{d}M^{\ell j}_t.
		\end{aligned}
	\end{equation}Comparing  (\ref{equ:ito for ana}) with  (\ref{equ:adj equ}), we obtain 
	\begin{equation}
		\label{equ:equa dt}
		\begin{aligned}
			&\  \Big(\dot{\Lambda}_{(\theta_t,t)}+\sum_{j=1}^{D}\pi^{\theta_t j}\Lambda_{(j,t)}\Big)\left(\mathbb{X}^{*}_t-\mathbb{E}^{\theta}_{t}[\mathbb{X}^{*}_t]\right)
			+  \Lambda_{(\theta_t,t)} A_{(\theta_t,t)}\left(\mathbb{X}^{*}_t-\mathbb{E}^{\theta}_{t}[\mathbb{X}^{*}_t]\right)+\Lambda_{(\theta_t,t)} B_{(\theta_t,t)}\left(u^{*}_t-\mathbb{E}^{\theta}_{t}[u^{*}_t]\right)\\
			+&\ 	\Big(\dot{\Gamma}_{(\theta_t,t)}+\sum_{j=1}^{D}\pi^{\theta_t j}\Gamma_{(j,t)}\Big)\mathbb{E}^{\theta}_{t}[\mathbb{X}^{*}_t]
			+	\Gamma_{(\theta_t,t)}\check{A}_{(\theta_t,t)}\mathbb{E}^{\theta}_{t}[\mathbb{X}^{*}_t]+\Gamma_{(\theta_t,t)}\check{B}_{(\theta_t,t)}\mathbb{E}^{\theta}_{t}[u^{*}_t]+\Gamma_{(\theta_t,t)} b_{(\theta_t,t)}\\
			+&\ \dot{\phi}_{(\theta_t,t)}+\sum_{j=1}^{D}\pi^{\theta_t j}\phi_{(j,t)}
			+ A_{(\theta_t,t)}\left(p_t-\mathbb{E}^{\theta}_{t}[p_t]\right)+\check{A}_{(\theta_t,t)}\mathbb{E}^{\theta}_{t}[p_t]+H_{(\theta_t,t)}\left(\mathbb{X}^{*}_t-\mathbb{E}^{\theta}_{t}[\mathbb{X}^{*}_t]\right)\\
			+&\  \check{H}_{(\theta_t,t)}\mathbb{E}^{\theta}_{t}[\mathbb{X}^{*}_t]+G^{\top}_{(\theta_t,t)}\left(u^{*}_t-\mathbb{E}^{\theta}_{t}[u^{*}_t]\right)+\check{G}^{\top}_{(\theta_t,t)}\mathbb{E}^{\theta}_{t}[u^{*}_t]=0,
		\end{aligned}
	\end{equation}and
		\begin{equation*}
		\begin{aligned}
			q_t=&\ \Lambda_{(\theta_t,t)} Q_t,\\
			k^{\ell j}_t=&\ \left(\Lambda_{(\ell,t)}-\Lambda_{(j,t)}\right)\left(\mathbb{X}^{*}_t-\mathbb{E}^{\theta}_{t}[\mathbb{X}^{*}_t]\right) +\left(\Gamma_{(\ell,t)}-\Gamma_{(j,t)}\right)\mathbb{E}^{\theta}_{t}[\mathbb{X}^{*}_t]+\phi_{(\ell,t)}-\phi_{(j,t)}.
		\end{aligned}
	\end{equation*}
	
	Taking $\mathbb{E}^{\theta}_t[\cdot]$ on  both  sides of (\ref{equ: ansatz of p}), we get 
	\begin{equation}
		\label{equ:cond of ana p}
		\left\{\begin{aligned}
			\mathbb{E}^{\theta}_{t}[p_t]=&\ \Gamma_{(\theta_t,t)}\mathbb{E}^{\theta}_{t}[\mathbb{X}^{*}_t]+\phi_{(\theta_t,t)},\\
			p_t-\mathbb{E}^{\theta}_{t}[p_t]=&\ \Lambda_{(\theta_t,t)}\left(\mathbb{X}^{*}_t-\mathbb{E}^{\theta}_{t}[\mathbb{X}^{*}_t]\right).
		\end{aligned}\right.
	\end{equation}Injecting  (\ref{equ:cond of ana p}) into (\ref{equ:coupl condi}), detailed  computation leads to  
	\begin{equation}
		\label{equ:condi u}
		\left\{\begin{aligned}
			\mathbb{E}^{\theta}_{t}[u^{*}_t]=&\ -\check{R}^{-1}_{(\theta_t,t)}\left(\widetilde{\check{B}}_{(\theta_t,t)}\mathbb{E}^{\theta}_{t}[\mathbb{X}^{*}_t] +\mathbb{B}_{(\theta_t,t)}\right),\\		
			u^{*}_t-\mathbb{E}^{\theta}_{t}[u^{*}_t]=&\ -R^{-1}_{(\theta_t,t)}\widetilde{B}_{(\theta_t,t)}\left(\mathbb{X}^{*}_t-\mathbb{E}^{\theta}_{t}[\mathbb{X}^{*}_t]\right).
		\end{aligned}\right.
	\end{equation}By injecting (\ref{equ:cond of ana p}) and (\ref{equ:condi u}) into (\ref{equ:equa dt}), we have
	\begin{equation*}\label{equ: inje ana}
		\begin{aligned}
			&\ \left[\dot{\Lambda}_{(\theta_t,t)}+A^{\top}_{(\theta_t,t)}\Lambda_{(\theta_t,t)}+\Lambda_{(\theta_t,t)}A_{(\theta_t,t)}+H_{(\theta_t,t)}-\widetilde{B}_{(\theta_t,t)}^{\top}R_{(\theta_t,t)}^{-1}\widetilde{B}_{(\theta_t,t)}+\sum_{j=1}^{D}\pi^{\theta_t j}\Lambda_{(j,t)}\right]\left(\mathbb{X}^{*}_t-\mathbb{E}^{\theta}_t[\mathbb{X}^{*}_t]\right)\\
			+&\ \left[ \dot{\Gamma}_{(\theta_t,t)}+\check{A}_{(\theta_t,t)}^{\top}\Gamma_{(\theta_t,t)}+\Gamma_{(\theta_t,t)}\check{A}_{(\theta_t,t)}+\check{H}_{(\theta_t,t)}-	\widetilde{\check{B}}_{(\theta_t,t)}^{\top}\check{R}_{(\theta_t,t)}^{-1}	\widetilde{\check{B}}_{(\theta_t,t)}+\sum_{j=1}^{D}\pi^{\theta_t j}\Gamma_{(j,t)}\right]\mathbb{E}^{\theta}_t[\mathbb{X}^{*}_t]\\
			+&\ \dot{\phi}_{(\theta_t,t)}+\Big(\check{A}_{(\theta_t,t)}^{\top}-	\widetilde{\check{B}}_{(\theta_t,t)}^{\top}\check{R}_{(\theta_t,t)}^{-1} \check{B}_{(\theta_t,t)}^{\top}\Big)\phi_{(\theta_t,t)}+\Gamma_{(\theta_t,t)}b_{(\theta_t,t)}+\sum_{j=1}^{D}\pi^{\theta_t j}\phi_{(j,t)}=0, \quad \text{a.s.}
		\end{aligned}
	\end{equation*}Motivated  by   the above identity, we define $\Lambda_{(\ell,t)}$, $\Gamma_{(\ell,t)}$, and $\phi_{(\ell,t)}$  as the solutions of the regime-wise equations \eqref{equ: ricca 1}-\eqref{equ: first order ode}. Equations (\ref{equ: ricca 1}) and (\ref{equ: ricca 2}) are two sets of coupled Riccati equations and (\ref{equ: first order ode}) is a set of  coupled first-order linear ODE. Under assumptions \ref{ass:coef in dyna}, \ref{ass: coe in cost one}, and \ref{ass: coe in cost two},  it follows from \cite[Lemma 1]{Zhang Yin 1999} that (\ref{equ: ricca 1}) and (\ref{equ: ricca 2}) admit unique positive semi-definite solutions $\Lambda_{(\ell,t)}$ and $\Gamma_{(\ell,t)}$ for all $\ell\in\mathbb{D}$. Furthermore, since  all the coefficients in (\ref{equ: first order ode}) are bounded, equation (\ref{equ: first order ode}) has a unique solution $\phi_{(\ell,t)}$ for all $\ell\in\mathbb{D}$ once $\Gamma_{(\ell,t)}$ is determined. Let 
	\begin{equation}\label{equ: u p q k}
		\left\{\begin{aligned}
			u^{*}_t=&\ -R^{-1}_{(\theta_t,t)}\widetilde{B}_{(\theta_t,t)}\left(\mathbb{X}^{*}_t-\mathbb{E}^{\theta}_t[\mathbb{X}^{*}_t]\right) -\check{R}^{-1}_{(\theta_t,t)}\left(	\widetilde{\check{B}}_{(\theta_t,t)}\mathbb{E}^{\theta}_t[\mathbb{X}^{*}_t]+\mathbb{B}_{(\theta_t,t)}\right),\\
			p_t=&\ \Lambda_{(\theta_t,t)}\left(\mathbb{X}^{*}_t-\mathbb{E}^{\theta}_t[\mathbb{X}^{*}_t]\right)+\Gamma_{(\theta_t,t)}\mathbb{E}^{\theta}_t[\mathbb{X}^{*}_t]+\phi_{(\theta_t,t)},\\
			q_t=&\ \Lambda_{(\theta_t,t)} Q_t,\\
			k^{\ell j}_t=&\ \left(\Lambda_{(\ell,t)}-\Lambda_{(j,t)}\right)\left(\mathbb{X}^{*}_t-\mathbb{E}^{\theta}_{t}[\mathbb{X}^{*}_t]\right) +\left(\Gamma_{(\ell,t)}-\Gamma_{(j,t)}\right)\mathbb{E}^{\theta}_{t}[\mathbb{X}^{*}_t]+\phi_{(\ell,t)}-\phi_{(j,t)}.
		\end{aligned}\right.
	\end{equation}Reversing the decoupling procedure above, we can derive that  $(\mathbb{X}^*,u^*,p,q,k)$ defined in   \eqref{equ: optimal state} and  \eqref{equ: u p q k} satisfy Hamiltonian systems (\ref{equ: rewr dyna}), (\ref{equ:adj equ}), and (\ref{equ:coupl condi}). Lemma \ref{the:openloop char} implies that $(\mathbb{X}^*,u^{*})$ is the optimal pair of \textbf{(FO-CMF)}.

	We now proceed to compute the corresponding optimal cost. Applying It$\widehat{\text{o}}$'s formula to $\left\langle\mathbb{E}^{\theta}_t[\mathbb{X}^{v}_t],\Gamma_{(\theta_t,t)}\mathbb{E}^{\theta}_t[\mathbb{X}^{v}_t]\right.$ $\left.+2\phi_{(\theta_t,t)}\right\rangle$ and  $\left\langle\left(\mathbb{X}^{v}_t-\mathbb{E}^{\theta}_t[\mathbb{X}^{v}_t]\right),\Lambda_{(\theta_t,t)}\left(\mathbb{X}^{v}_t-\mathbb{E}^{\theta}_t[\mathbb{X}^{v}_t]\right)\right\rangle$, integrating from $0$ to $T$,  taking expectation, and inserting the derived expressions to (\ref{equ: rewr cost}),  we have
	\begin{equation}	\label{equ: inequ optimal cost}
		\begin{aligned}
			J^{F}\left(\xi,\ell_0;v\right)
			=&\  \frac{1}{2}\mathbb{E}\left\{\int_{0}^{T}\left[\left|R_{(\theta_t,t)}^{\frac{1}{2}}\left[v_t-\mathbb{E}^{\theta}_t[v_t]+R^{-1}_{(\theta_t,t)}\widetilde{B}_{(\theta_t,t)}\left(\mathbb{X}^v_t-\mathbb{E}^{\theta}_t[\mathbb{X}^v_t]\right)\right]\right|^2\right.\right.\\
			&\  \left.\left.+\left|\check{R}^{\frac{1}{2}}_{(\theta_t,t)}\left[\mathbb{E}^{\theta}_t[v_t]+\check{R}^{-1}_{(\theta_t,t)}\left(	\widetilde{\check{B}}_{(\theta_t,t)}\mathbb{E}^{\theta}_t[\mathbb{X}^v_t]+\mathbb{B}_{(\theta_t,t)}\right)\right]\right|^2+\mathrm{tr}\left(Q^{\top}_t\Lambda_{(\theta_t,t)} Q_t\right)\right.\right.\\
			&\ \left.\left.+2\left\langle \phi_{(\theta_t,t)},b_{(\theta_t,t)}\right\rangle-|\check{R}^{-\frac{1}{2}}_{(\theta_t,t)}\mathbb{B}_{(\theta_t,t)}|^2\right]\mathrm{d}t+\left\langle \mu,\Gamma_{(\ell_0,0)}\mu+2\phi_{(\ell_0,0)}\right\rangle \right\}\\
			\geq &\  \frac{1}{2}\mathbb{E}\left\{\int_{0}^{T}\left[-|\check{R}^{-\frac{1}{2}}_{(\theta_t,t)}\mathbb{B}_{(\theta_t,t)}|^2+\mathrm{tr}\left(Q^{\top}_t\Lambda_{(\theta_t,t)} Q_t\right)+2\left\langle \phi_{(\theta_t,t)},b_{(\theta_t,t)}\right\rangle\right]\mathrm{d}t\right.\\
			&\ \left. +\left\langle \mu,\Gamma_{(\ell_0,0)}\mu+2\phi_{(\ell_0,0)}\right\rangle\right\}.
		\end{aligned}
	\end{equation}

	The equality  in (\ref{equ: inequ optimal cost}) holds if and only if  $v_t-\mathbb{E}^{\theta}_t[v_t]=-R^{-1}_{(\theta_t,t)}\widetilde{B}_{(\theta_t,t)}(\mathbb{X}^v_t-\mathbb{E}^{\theta}_t[\mathbb{X}^v_t])$ and $\mathbb{E}^{\theta}_t[v_t]=-\check{R}^{-1}_{(\theta_t,t)}(	\widetilde{\check{B}}_{(\theta_t,t)}\mathbb{E}^{\theta}_t[\mathbb{X}^v_t]+\mathbb{B}_{(\theta_t,t)})$, which is equivalent to $ v_t=-R^{-1}_{(\theta_t,t)}\widetilde{B}_{(\theta_t,t)}(\mathbb{X}^v_t-\mathbb{E}^{\theta}_t[\mathbb{X}^v_t])- \check{R}^{-1}_{(\theta_t,t)}(	\widetilde{\check{B}}_{(\theta_t,t)}\mathbb{E}^{\theta}_t[\mathbb{X}^v_t]+\mathbb{B}_{(\theta_t,t)})$. This coincides exactly with  optimal control (\ref{equ: opti cont of JLQ}). Combining this with (\ref{equ: inequ optimal cost}), we conclude that the optimal cost is given by (\ref{equ:opti ccost of CMF-LQ}).
\end{proof}

\subsection{Solution  of \textbf{(PO-CMF)}}

By virtue of the equivalence between \textbf{(PO-CMF)}  and  \textbf{(FO-CMF)}, we can derive the optimal control and associated optimal cost for   \textbf{(PO-CMF)} from Theorem \ref{lem:opti cont of CMF-SDE}. The result is summarized as follows.

\begin{Theorem}\label{the:opti cont of CMF-LQG}
	Under assumptions \ref{ass:coef in dyna}, \ref{ass: coe in cost one}, and \ref{ass: coe in cost two}, the optimal control of \textbf{(PO-CMF)} admits the following  feedback representation:  
	\begin{equation}
		\label{equ: opti cont of CMF-LQg}
		\begin{aligned}
			u^{*}_t=&\ -R_{(\theta_t,t)}^{-1}\widetilde{B}_{(\theta_t,t)}\left(\mathbb{E}^{Y^*,\theta}_t[X^{*}_t]-\mathbb{E}^{\theta}_t[X^{*}_t]\right) -\check{R}_{(\theta_t,t)}^{-1}\left(	\widetilde{\check{B}}_{(\theta_t,t)}\mathbb{E}^{\theta}_t[X^{*}_t]+\mathbb{B}_{(\theta_t,t)}\right),
		\end{aligned}	
	\end{equation}where    $\Lambda_{(\ell,t)}$, $\Gamma_{(\ell,t)}$, and $\phi_{(\ell,t)}$ with $\ell\in\mathbb{D}$ are solutions of  (\ref{equ: ricca 1})-(\ref{equ: first order ode}).

	The  optimal cost of \textbf{(PO-CMF)} is 
	\begin{equation}
		\label{equ: opti cost of CMF-LQG}
		\begin{aligned}
			\mathscr{V}^P\left(\xi,\ell_0\right)=&\ \frac{1}{2}\mathbb{E}\left\{\int_{0}^{T}\left[-|\check{R}^{-\frac{1}{2}}\mathbb{B}|^2+2\left\langle\phi,b\right\rangle +\mathrm{tr}\left(Q^{\top}\Lambda Q+F\Phi\varphi\Phi F^{\top}+\bar{C}^{\top}\varphi\bar{C}\right)\right]\mathrm{d}t\right.\\
			&\ \left. +\left\langle\mu,\Gamma_{(\ell_0,0)}\mu+2\phi_{(\ell_0,0)}\right\rangle+\mathrm{tr}\left(\varphi_0\sigma\right)\right\},
		\end{aligned}
	\end{equation}where  $(\varphi,\vartheta)$ is the unique solution to the  following BSDE driven by Markov chain:
	\begin{equation}
		\label{equ: ode for varphi}
		\left\{\begin{aligned}
			\mathrm{d}\varphi_t=&\ -\left[\mathcal{A}^{\top}_t\varphi_t+\varphi_t\mathcal{A}_t+H_{(\theta_t,t)}\right]\mathrm{d}t+\vartheta_t\bullet \mathrm{d}M_t,\\
			\varphi_T=&\ S_{\theta_T}. 
		\end{aligned}\right.
	\end{equation}
	%	\begin{equation}
		%		\label{equ: ode for varphi}
		%		\left\{\begin{aligned}
			%			&\dot{\varphi}_{\ell}+\left(A_{\ell}-Q_{\ell}L^{-1}_{\ell}F_{\ell}\right)^{\top}\varphi_{\ell}+\varphi_{\ell}\left(A_{\ell}-Q_{\ell}L^{-1}_{\ell}F_{\ell}\right)+H_{\ell}+\sum_{j=1}^{D}\pi^{\ell j}\varphi_{j}=0,\\
			%			&\varphi_{(\ell,T)}=S_{\ell}. 
			%		\end{aligned}\right.
		%	\end{equation}
\end{Theorem}
\begin{proof}
	Optimal control (\ref{equ: opti cont of CMF-LQg})  is obtained  by replacing $\mathbb{X}^{*}_t$ and $\mathbb{E}^{\theta}_t[\mathbb{X}^{*}_t]$ with $\mathbb{E}^{Y^*,\theta}_t[X^{*}_t]$ and $\mathbb{E}^{\theta}_t[X^{*}_t]$ in (\ref{equ: opti cont of JLQ}).

	In view of (\ref{equ: deco of cost}) and (\ref{equ:opti ccost of CMF-LQ}), it remains to compute  $\mathbb{J}\left(\xi,\ell_0\right)$ in order to derive the  optimal cost $\mathscr{V}^P\left(\xi,\ell_0\right)$ of \textbf{(PO-CMF)}. Under assumptions \ref{ass:coef in dyna} and \ref{ass: coe in cost one}, Theorem 6.1 in \cite{Cohen  Elliott 2008} ensures  that   (\ref{equ: ode for varphi}) has a unique solution $(\varphi,\vartheta)\in\mathcal{S}^2_{\mathbb{F}^{\theta}}([0,T];\mathbb{R}^{n\times n})\times\mathcal{M}^2_{\mathbb{F}^{\theta}}([0,T];\mathbb{R}^{n\times n})$. Applying It$\widehat{\text{o}}$'s formula to  $\left\langle\mathcal{E}_t,\varphi_t\mathcal{E}_t\right\rangle$, where $\mathcal{E}$ satisfies (\ref{equ:dyna for error}),  we have 
	\begin{equation}
		\label{equ:ito for E varphi}
		\begin{aligned}
			\mathbb{E}\left\{\left\langle \mathcal{E}_T, S_{\theta_T} \mathcal{E}_T\right\rangle -\left\langle\xi-\mu,\varphi_0\left(\xi-\mu\right)\right\rangle\right\}
			= \mathbb{E}\left\{\int_{0}^{T}\left[-\left \langle \mathcal{E}, H \mathcal{E}\right\rangle+\mathrm{tr}\left(F\Phi\varphi\Phi F^{\top}+\bar{C}^{\top}\varphi\bar{C}\right)\right]\mathrm{d}t\right\}.
		\end{aligned}
	\end{equation}It follows from  (\ref{equ: cost inde of cont}) and \eqref{equ:ito for E varphi} that 
	\begin{equation}
		\label{equ:cost inde of cont}
		\begin{aligned}
			 \mathbb{J}\left(\xi,\ell_0\right)
			= \frac{1}{2}\mathbb{E} \left\{\int_{0}^{T}\mathrm{tr}\left(F\Phi\varphi\Phi F^{\top}+\bar{C}^{\top}\varphi\bar{C}\right) \mathrm{d}t+\mathrm{tr}\left(\varphi_0\sigma\right)\right\}.
		\end{aligned}
	\end{equation}Using (\ref{equ: deco of cost}),  (\ref{equ:opti ccost of CMF-LQ}), and (\ref{equ:cost inde of cont}), we obtain
	\begin{equation*}
		\begin{aligned}
			\mathscr{V}^P\left(\xi,\ell_0\right)
			=&\ \mathscr{V}^F\left(\xi,\ell_0\right)+ \mathbb{J}\left(\xi,\ell_0\right)\\
			=&\ \frac{1}{2}\mathbb{E}\left\{\int_{0}^{T}\left[-|\check{R}^{-\frac{1}{2}}\mathbb{B}|^2+2\left\langle\phi,b\right\rangle +\mathrm{tr}\left(Q^{\top}\Lambda Q+F\Phi\varphi\Phi F^{\top}+\bar{C}^{\top}\varphi\bar{C}\right)\right]\mathrm{d}t\right.\\
			&\  \left.  +\left\langle\mu,\Gamma_{(\ell_0,0)}\mu+2\phi_{(\ell_0,0)}\right\rangle+\mathrm{tr}\left(\varphi_0\sigma\right)\right\},
		\end{aligned}
	\end{equation*}which gives (\ref{equ: opti cost of CMF-LQG}).
\end{proof}

Theorem \ref{the:opti cont of CMF-LQG} shows that the optimal control of  \textbf{(PO-CMF)} has a feedback form of state estimate $\mathbb{E}^{Y^*,\theta}_t[X^{*}_t]$ and its conditional expectation $\mathbb{E}^{\theta}_t[X^*_t]$. Moreover, the optimal control can be obtained by replacing the true state   with state estimate in the feedback control law of complete observation problem. It means that the separation principle  holds for  \textbf{(PO-CMF)}.

\begin{Remark}
	Theorem~\ref{the:opti cont of CMF-LQG} encompasses the following two classical results as special cases.
	
	\textit{Case 1}: If the Markov chain is trivial, that is, $\theta$ remains in a fixed regime and does not switch for any $t\in[0,T]$, then its transition rates matrix $\Pi$ becomes the zero matrix and the associated filtration reduces to $\mathcal{F}^{\theta}_t=\{\emptyset,\Omega\}$. In this situation, the conditional expectation $\mathbb{E}^{\theta}_t[X^v_t]$ degenerates to the unconditional expectation $\mathbb{E}[X^v_t]$, and the three sets of interconnected equations (\ref{equ: ricca 1})-(\ref{equ: first order ode}) become  three single equations. Consequently, Theorem~\ref{the:opti cont of CMF-LQG} reduces to the solution of the partially observed mean-field type LQ optimal control problem studied in~\cite{Moon Basar 2024}.
	
	\textit{Case 2}: If the conditional expectations of the state and control are not taken into account, i.e., the coefficients $\widehat{A}_{(\ell,t)}$, $\widehat{B}_{(\ell,t)}$, $\widehat{F}_{(\ell,t)}$, $\widehat{H}_{(\ell,t)}$, $\widehat{G}_{(\ell,t)}$, $\widehat{R}_{(\ell,t)}$, $\widehat{S}_{\ell}$ are all zero matrices for any $\ell\in\mathbb{D}$ and $t\in[0,T]$, then the problem setting reduces to the classical partially observed regime-switching LQ problem considered in~\cite{Dufour adaptive 1998,Ji 1992,Mariton1990}.
\end{Remark}

\section{Applications  and Numerical Simulations}
\label{sec: applications  and numerical simulations}
In this Section, we present two applications and corresponding numerical simulations to validate the effectiveness of our theoretical results.
\subsection{A Vector-Valued LQ Example}

In this part,  we consider a two-dimensional example in which the Markov chain $\theta$ takes values in the three-state space $\{1,2,3\}$ with transition rates matrix
\begin{equation*}
	\begin{aligned}
		\Pi=\left[\begin{array}{ccc}
			-0.08 &  0.05 & 0.03 \\
			0.04 & -0.1 &  0.06 \\
			0.02 &  0.07 & -0.09
		\end{array}\right].
	\end{aligned}
\end{equation*}The initial state is given by $\xi=[1\ \ 5]^{\top}$. The coefficients appearing in the state-observation dynamics \eqref{equ:state}-\eqref{equ:obser} and the cost functional \eqref{equ:ori cost} corresponding to different regimes  of $\theta_t$ are specified in \eqref{equ: coe 1}-\eqref{equ: coe 3}, respectively. In this example, some coefficients are designed as time-varying functions. It is readily verified that these coefficients satisfy assumptions \ref{ass:coef in dyna}, \ref{ass: coe in cost one}, and \ref{ass: coe in cost two}, hence, Theorem \ref{the:opti cont of CMF-LQG} applies.  We set the terminal time as $T = 50$. Numerical simulations are conducted using the Euler-Maruyama method, and the results are presented in Figs. \ref{fig:simulinks of markov chain and solutions of odes}-\ref{fig:empirical estimation error}.

\begin{figure*}[!t]
	{\scriptsize \begin{equation}\label{equ: coe 1}
			(\theta_t=1)\left\{\begin{aligned}
				A_{(1,t)}
				=&\  \left[
				\begin{array}{cc}
					-1+0.2e^{-0.05t} & 0.2\\
					-0.1 & -1.5+\dfrac{0.1}{1+t}
				\end{array}\right],
				\widehat{A}_{(1,t)}
				=
				\left[
				\begin{array}{cc}
					0.1\tanh(0.1t) & 0\\
					0 & 0.15e^{-0.1t}
				\end{array}
				\right],
				B_{(1,t)}
				=
				\left[
				\begin{array}{cc}
					1+\dfrac{0.2t}{1+t} & 0\\
					0.1 & 1
				\end{array}
				\right],\\
				\widehat{B}_{(1,t)}
				=&\ 
				\left[
				\begin{array}{cc}
					\dfrac{0.1}{1+t^2} & 0\\
					0 & 0.1\tanh(0.05t)
				\end{array}
				\right],
				C_{(1,t)}
				=
				\left[
				\begin{array}{cc}
					0.5+0.1e^{-0.1t} & 0\\
					0 & 0.4+\dfrac{0.1t}{1+t}
				\end{array}
				\right],
				\bar C_{(1,t)}
				=
				\left[
				\begin{array}{cc}
					0.2e^{-0.05t} & 0\\
					0 & 0.2
				\end{array}
				\right],\\
				F_{(1,t)}
				=&\ 
				\left[
				\begin{array}{cc}
					1+0.1\tanh(0.1t) & 0\\
					0.05 & 1
				\end{array}
				\right],
				\widehat{F}_{(1,t)}
				=
				\left[
				\begin{array}{cc}
					0.1e^{-0.05t} & 0\\
					0 & \dfrac{0.1}{1+t}
				\end{array}
				\right],
				b_{(1,t)}
				=
				\left[
				\begin{array}{c}
					e^{-0.05t}\\
					\dfrac{t}{1+t}
				\end{array}
				\right],
				f_{(1,t)}
				=
				\left[
				\begin{array}{c}
					\dfrac{0.1}{1+t}\\
					0.2e^{-0.1t}
				\end{array}
				\right],\\
				H_{(1,t)}
				=&\ 
				\left[
				\begin{array}{cc}
					1+0.2\sin t & 0\\
					0 & 1.5+0.1\cos t
				\end{array}
				\right],
				\widehat{H}_{(1,t)}
				=
				\left[
				\begin{array}{cc}
					1+0.1\cos t & 0\\
					0 & 1+0.2\sin t
				\end{array}
				\right],\\
				R_{(1,t)}
				=&\
				\left[
				\begin{array}{cc}
					2.5+0.1\sin t & 0\\
					0 & 2.8+0.1\cos t
				\end{array}
				\right],
				\widehat{R}_{(1,t)}
				= 
				\left[
				\begin{array}{cc}
					2.2 & 0\\
					0 & 2.1
				\end{array}
				\right],
				G_{(1,t)}
				=
				\left[
				\begin{array}{cc}
					0.2\sin t & 0\\
					0 & 0.1\cos t
				\end{array}
				\right],\\
				\widehat{G}_{(1,t)}
				=&\ 
				\left[
				\begin{array}{cc}
					0.1 & 0\\
					0 & 0.1\sin t
				\end{array}
				\right],S_{1}
				=
				\left[
				\begin{array}{cc}
					2 & 0\\
					0 & 1
				\end{array}
				\right],\widehat{S}_{1}
				=
				\left[
				\begin{array}{cc}
					1 & 0\\
					0 & 2
				\end{array}
				\right].
			\end{aligned}\right.
	\end{equation}}
	\hrulefill
	\vspace*{-4mm}
\end{figure*}

\begin{figure*}[!t]
	{\scriptsize \begin{equation}\label{equ: coe 2}
			(\theta_t=2)\left\{\begin{aligned}
				A_{(2,t)}
				=&\  \left[
				\begin{array}{cc}
					-0.8+\dfrac{0.15}{1+t} & 0.3\\
					-0.2 & -1.2+0.1e^{-0.03t}
				\end{array}\right],
				\widehat{A}_{(2,t)}
				=
				\left[
				\begin{array}{cc}
					0.2\tanh(0.05t) & 0.05\\
					0 & 0.1e^{-0.1t}
				\end{array}
				\right],
				B_{(2,t)}
				=
				\left[
				\begin{array}{cc}
					1 & \dfrac{0.1t}{1+t}\\
					0.3 & 1.1
				\end{array}
				\right],\\
				\widehat{B}_{(2,t)}
				=&\ 
				\left[
				\begin{array}{cc}
					0.1e^{-0.1t} & 0\\
					0.05 & \dfrac{0.2t^2}{1+t^2}
				\end{array}
				\right],
				C_{(2,t)}
				=
				\left[
				\begin{array}{cc}
					0.3 & \dfrac{0.1}{1+t}\\
					0.1 & 0.5
				\end{array}
				\right],
				\bar C_{(2,t)}
				=
				\left[
				\begin{array}{cc}
					0.15+\dfrac{0.05t}{1+t} & 0\\
					0 & 0.25
				\end{array}
				\right],\\
				F_{(2,t)}
				=&\ 
				\left[
				\begin{array}{cc}
					1+0.1e^{-0.05t} & 0\\
					0.1 & 1
				\end{array}
				\right],
				\widehat{F}_{(2,t)}
				=
				\left[
				\begin{array}{cc}
					0.1\tanh(0.1t) & 0\\
					0.05 & 0.1
				\end{array}
				\right],
				b_{(2,t)}
				=
				\left[
				\begin{array}{c}
					\dfrac{0.5t}{1+t}\\
					1
				\end{array}
				\right],
				f_{(2,t)}
				=
				\left[
				\begin{array}{c}
					0.1\\
					\dfrac{0.2t}{1+t}
				\end{array}
				\right],\\
				H_{(2,t)}
				=&\ 
				\left[
				\begin{array}{cc}
					1.2+0.1\sin(0.5t) & 0.05\\
					0.05 & 1.3+0.1\cos t
				\end{array}
				\right],\widehat{H}_{(2,t)}
				=
				\left[
				\begin{array}{cc}
					1+0.1\cos(0.3t) & 0\\
					0 & 1.1+0.1\sin t
				\end{array}
				\right],\\R_{(2,t)}
				=&\
				\left[
				\begin{array}{cc}
					2.6 & 0.05\\
					0.05 & 2.4
				\end{array}
				\right],
				\widehat{R}_{(2,t)}
				= 
				\left[
				\begin{array}{cc}
					2.3+0.1\sin t & 0\\
					0 & 2.2
				\end{array}
				\right],G_{(2,t)}
				=
				\left[
				\begin{array}{cc}
					0.15\cos t & 0.05\\
					0 & 0.1\sin t
				\end{array}
				\right],\\ 
				\widehat{G}_{(2,t)}
				=&\ 
				\left[
				\begin{array}{cc}
					0.1\sin t & 0\\
					0.05 & 0.1
				\end{array}
				\right],S_{2}
				=
				\left[
				\begin{array}{cc}
					1.5 & 0.2\\
					0.2 & 1
				\end{array}
				\right],\widehat{S}_{2}
				=
				\left[
				\begin{array}{cc}
					0.5 & 0\\
					0 & 1
				\end{array}
				\right].
			\end{aligned}\right.
	\end{equation}}
	\hrulefill
	\vspace*{-4mm}
\end{figure*}

\begin{figure*}[!t]
	{\scriptsize \begin{equation}\label{equ: coe 3}
			(\theta_t=3)\left\{\begin{aligned}
				A_{(3,t)}
				=&\  \left[
				\begin{array}{cc}
					-1.3+0.1e^{-0.05t} & 0.1\\
					-0.3 & -0.9+\dfrac{0.1}{1+t^2}
				\end{array}\right],
				\widehat{A}_{(3,t)}
				=
				\left[
				\begin{array}{cc}
					0.1 & 0\\
					0.05\tanh(0.1t) & 0.2e^{-0.1t}
				\end{array}
				\right],
				\\
				B_{(3,t)}
				=&\
				\left[
				\begin{array}{cc}
					0.8+\dfrac{0.1t}{1+t} & 0\\
					0.1 & 1.2
				\end{array}
				\right],\widehat{B}_{(3,t)}
				= 
				\left[
				\begin{array}{cc}
					0.15 & \dfrac{0.05}{1+t}\\
					0 & 0.1
				\end{array}
				\right],
				C_{(3,t)}
				=
				\left[
				\begin{array}{cc}
					0.6 & 0\\
					\dfrac{0.1t}{1+t} & 0.3
				\end{array}
				\right],
				\bar C_{(3,t)}
				=
				\left[
				\begin{array}{cc}
					0.1 & 0.05\\
					0 & 0.2+\dfrac{0.05t}{1+t}
				\end{array}
				\right],\\
				F_{(3,t)}
				=&\ 
				\left[
				\begin{array}{cc}
					0.8 & 0.1\\
					0 & 1.2+0.1e^{-0.05t}
				\end{array}
				\right],
				\widehat{F}_{(3,t)}
				=
				\left[
				\begin{array}{cc}
					0.1 & \dfrac{0.05}{1+t}\\
					0 & 0.15
				\end{array}
				\right],
				b_{(3,t)}
				=
				\left[
				\begin{array}{c}
					1\\
					\dfrac{0.5}{1+t}
				\end{array}
				\right],
				f_{(3,t)}
				=
				\left[
				\begin{array}{c}
					0.15e^{-0.1t}\\
					0.1
				\end{array}
				\right],\\
				H_{(3,t)}
				=&\ 
				\left[
				\begin{array}{cc}
					1.8+0.1\sin t & 0\\
					0 & 1.4+0.1\cos(2t)
				\end{array}
				\right],\widehat{H}_{(3,t)}
				=
				\left[
				\begin{array}{cc}
					1.2 & 0.05\\
					0.05 & 1.1
				\end{array}
				\right],R_{(3,t)}
				=
				\left[
				\begin{array}{cc}
					3+0.1\cos t & 0\\
					0 & 2.7
				\end{array}
				\right],\\
				\widehat{R}_{(3,t)}
				=&\ 
				\left[
				\begin{array}{cc}
					2.1 & 0\\
					0 & 2.5+0.1\sin t
				\end{array}
				\right],G_{(3,t)}
				=
				\left[
				\begin{array}{cc}
					0.2 & 0\\
					0 & 0.15\cos t
				\end{array}
				\right],\widehat{G}_{(3,t)}
				=
				\left[
				\begin{array}{cc}
					0.1\cos t & 0\\
					0 & 0.1\sin t
				\end{array}
				\right],S_{3}
				=
				\left[
				\begin{array}{cc}
					3 & 0\\
					0 & 2
				\end{array}
				\right],\\\widehat{S}_{3}
				=&\ 
				\left[
				\begin{array}{cc}
					1 & 0.1\\
					0.1 & 1
				\end{array}
				\right].
			\end{aligned}\right.
	\end{equation}}
	\hrulefill
	\vspace*{-4mm}
\end{figure*}

Fig.~\ref{fig:simulinks of markov chain and solutions of odes} presents the sample path of Markov chain $\theta_t$ and  the evolution of  each component of  $\Lambda_{(\theta_t,t)}$.

Fig.~\ref{fig:simulinks of control} illustrates the trajectories of the optimal control $u^{*}_t$ and its conditional expectation $\mathbb{E}^{\theta}_t[u^{*}_t]$. Fig.~\ref{fig:simulinks of state and filtering} presents the evolution of the optimal state $X^{*}_t$, its conditional expectation $\mathbb{E}^{\theta}_t[X^{*}_t]$, and the optimal filter $\mathbb{E}^{Y^*,\theta}_t[X^{*}_t]$ under the optimal control. It is observed that both the optimal control and the optimal state exhibit distinct dynamical behaviors across different regimes of the Markov chain $\theta_t$. These simulation results clearly demonstrate the significant influence of the regime-switching mechanism on the control strategy and the corresponding state evolution.

Fig.~\ref{fig:empirical estimation error} presents the 2-norm of the empirical estimation error for each state component, defined by
$(\frac{1}{N}\sum_{K=1}^{N}|X^{i,*,(K)}_t-\mathbb{E}^{Y^*,\theta}_t[X^{i,*,(K)}_t]|^2)^{\frac{1}{2}}$
where $N=1000$ denotes the number of Monte Carlo samples and $X^{i,*,(K)}_t$  represents the $i$-th ($i\in\{1,2\}$) component of the $K$-th sample path. As shown in Fig.~\ref{fig:empirical estimation error}, the 2-norm of  the empirical estimation error remains below $0.15$ for both state components throughout the simulation horizon. This demonstrates that the optimal filter $\mathbb{E}^{Y^*,\theta}_t[X^{*}_t]$ provides an accurate estimate of the unobservable state $X^{*}_t$, thereby confirming the effectiveness of the proposed filtering approach.

\begin{figure}[htbp!]
	\centering
	\subfloat[Sample path of Markov chain.]{%
		\label{fig:(a)trajectory of markov 	chain}\resizebox*{7cm}{!}{\includegraphics[width=2cm,height=1.5cm]{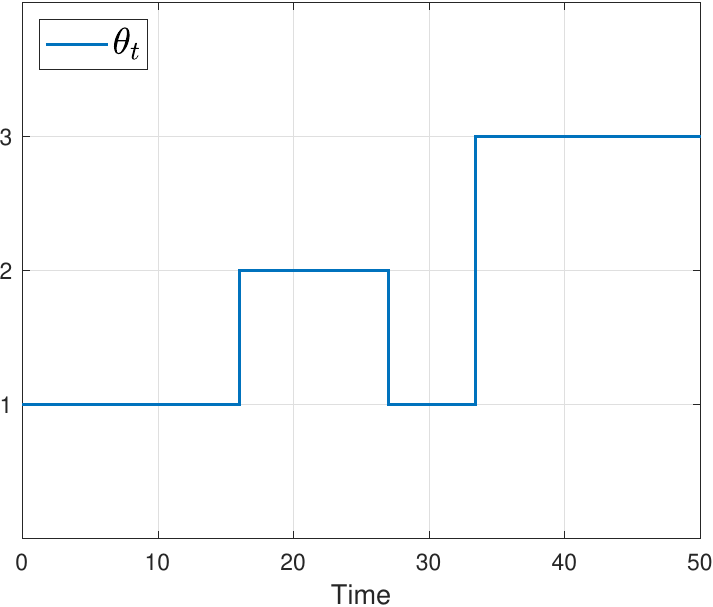}}}\hspace{0cm}
	\subfloat[Evolution of  each component of  $\Lambda_{(\theta_t,t)}$.]{%
		\label{fig:(b)solution of Lambda}\resizebox*{7cm}{!}{\includegraphics[width=2cm,height=1.5cm]{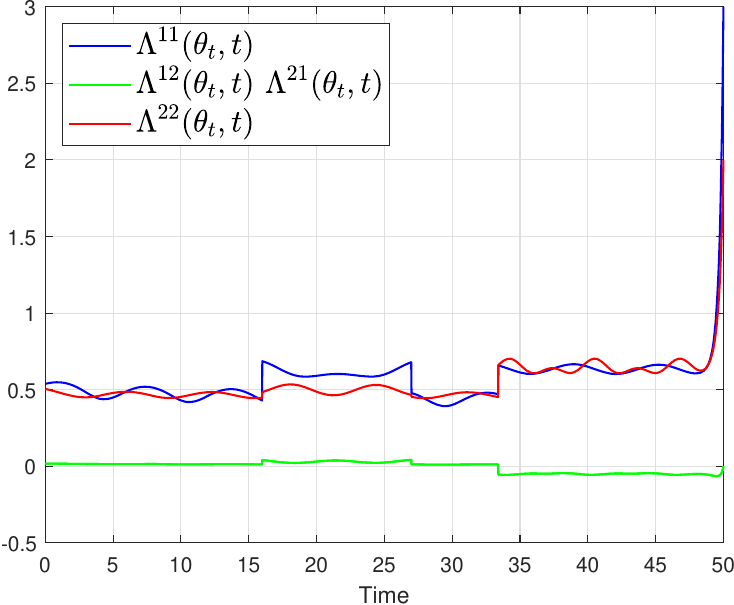}}}\hspace{0cm}
	\caption{Sample path of Markov chain $\theta_t$ and  evolution of  each component of  $\Lambda_{(\theta_t,t)}$.}
	\label{fig:simulinks of markov chain and solutions of odes}
\end{figure}

\begin{figure}[htbp!]
	\centering
	\subfloat{%
		\label{fig:(a)trajectory of u 1}{\includegraphics[width=5in]{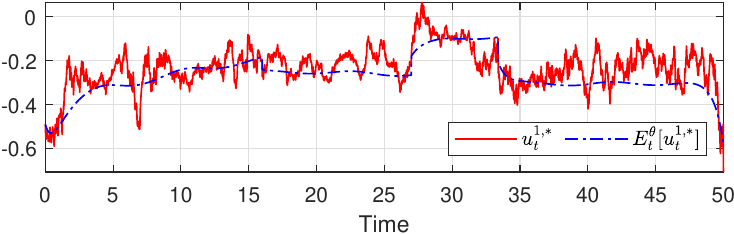}}}\hspace{0cm}
	\subfloat{%
		\label{fig:(b)trajectory of u 2}{\includegraphics[width=5in]{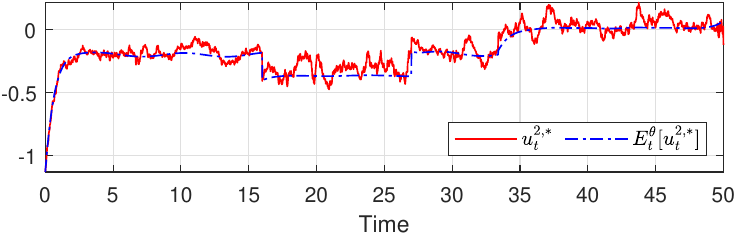}}}\hspace{0cm}
	\caption{Trajectories of $u^{*}_t$ and  $\mathbb{E}^{\theta}_t[u^{*}_t]$.}
	\label{fig:simulinks of control}
\end{figure}

\begin{figure}[htbp!]
	\centering
	\subfloat{%
		\label{fig:(a)trajectory of X 1}{\includegraphics[width=5in]{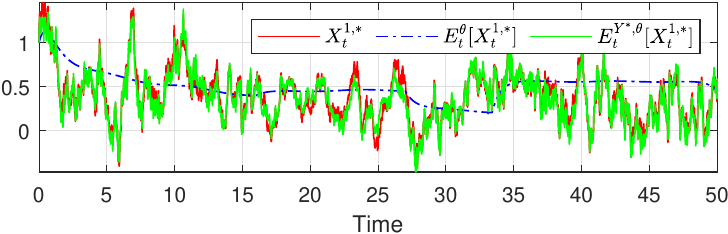}}}\hspace{0cm}
	\subfloat{%
		\label{fig:(b)trajectory of X 2}{\includegraphics[width=5in]{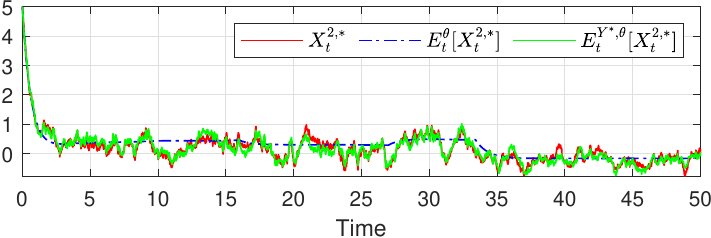}}}\hspace{0cm}
	\caption{Trajectories of $X^{*}_t$, $\mathbb{E}^{\theta}_t[X^{*}_t]$, and $\mathbb{E}^{Y^{*},\theta}_t[X^{*}_t]$.}
	\label{fig:simulinks of state and filtering}
\end{figure}

\begin{figure}[h]
	\centering
	\includegraphics[width=0.6\textwidth]{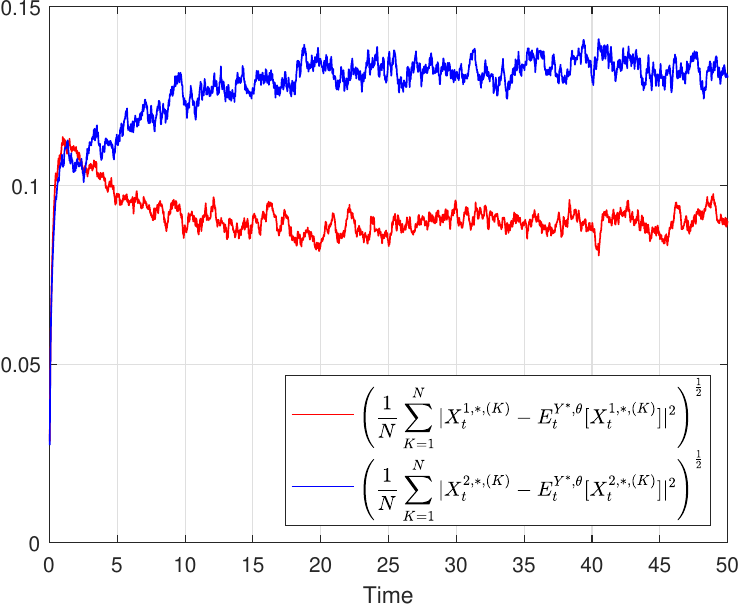}
	\caption{The 2-norm of  the empirical estimation error.}
	\label{fig:empirical estimation error}
\end{figure}

\subsection{Coupled Electrical Machines Control Problem}\label{sec:coupled ele machine}
In this part, we study an optimal control problem for a network of coupled electrical machines. Inspired by \cite{Loparo Blankenship 2003}, we consider the following $M$ second-order systems that describe the dynamics of the machines operating within the network: for $m\in\{1,\dots,M\}$,
\begin{equation}
	\label{equ:model of ele mech}
	\begin{aligned}
		 \ddot{\varphi}_{m}(t)+2\alpha_{m}\dot{\varphi}_{m}(t)+\omega^2_{m}\varphi_{m}(t)
		= \sum_{n=1}^{M}\epsilon\mu_{mn}(t)\varphi_{n}(t)+b_{m}v_{m}(t)+\bar{C}_{m}\bar{\mathbb{W}}_m(t).
	\end{aligned}
\end{equation}In \eqref{equ:model of ele mech}, $M$ denotes the number of machines, and $\varphi_{m}(t)$ represents the deviation of the rotor angle of the $m$-th machine at time $t$ with respect to its nominal value.  Parameters $\alpha_{m}$ and $\omega_{m}$ correspond to the damping coefficient and natural frequency of the $m$-th machine, respectively. $v_m(t)$ denotes the mechanical power control, which can adjust the input mechanical power of the generator to compensate for the speed deviations.  $b_m$ and $\bar{C}_m$ are the effectiveness factor of control input and noise for $m$-th machine. Following \cite{Loparo Blankenship 2003}, we assume $0=\alpha_{1}<\alpha_{2}\leq \cdots \leq \alpha_{M}$, which implies that the first machine has zero damping, whereas all  remaining machines have positive damping. Coefficient $\epsilon$ characterizes the strength of the coupling among the $M$ machines.  Random telegraph process $\mu_{mn}(t)$ models the   coupling fluctuations between machines $m$ and $n$. Process $\bar{\mathbb{W}}_m(t)$ denotes the engineering white noise, which satisfies certain properties. For instance, for any $t' \neq t''\in[0,T]$, $\bar{\mathbb{W}}_m(t')$ and $\bar{\mathbb{W}}_m(t'')$ are independent, the process is stationary with $\mathbb{E}[\bar{\mathbb{W}}_m(t)] = 0$ for all $t$.

Following the treatment in \cite[Section 10.3.1]{Costa Fragoso Todorov 2012}, we restrict our attention to the case of two machines for simplicity, i.e., $M = 2$. Furthermore, to obtain a more explicit formulation, we assume that the coupling coefficients satisfy $\mu_{11}(t) = \mu_{22}(t) = \theta_{1}(t)$ and $\mu_{12}(t) = \mu_{21}(t) = \theta_{2}(t)$, where $\theta_{1}(t)$ and $\theta_{2}(t)$ are independent, time homogeneous Markov chains taking values in $\{1, -1\}$. The transition rates matrices $\Pi_{m}$ of $\theta_{m}(t)$ are given by $\renewcommand{\arraystretch}{0.8}\Pi_{m}=\left[\begin{array}{@{}cc@{}}  % ?? @{} ?????????
	-\pi_m & \pi_m \\
	\pi_m & -\pi_m
\end{array}\right]$ for $m\in\{1,2\}$.

To derive a state space representation of second-order systems (\ref{equ:model of ele mech}), we define 
\begin{equation}
	\label{equ:sta of ele mac}
	\begin{aligned}
		X^v_t=\left[\varphi_1(t),\varphi_2(t),\dot{\varphi}_1(t),\dot{\varphi}_2(t)\right]^{\top},
	\end{aligned}
\end{equation}and assume $X^v_0=[1,1,1,1]^{\top}$. Let $\theta_t=h(\theta_1(t),\theta_2(t))$ with $h(-1,-1)=1$, $h(-1,1)=2$, $h(1,-1)=3$, and $h(1,1)=4$. Then process $\theta$ is a  time homogeneous Markov chain with state space $\{1,2,3,4\}$ and its translation rates matrix  is 
\begin{equation*}
	\renewcommand{\arraystretch}{0.8}
	\begin{aligned}
		\Pi=\left[\begin{array}{@{}c@{}c@{}c@{}c@{}}
			-(\pi_1+\pi_2) & \pi_2 & \pi_1 & 0\\ 
			\pi_2 & -(\pi_1+\pi_2) & 0 & \pi_1\\
			\pi_1 & 0 & -(\pi_1+\pi_2) & \pi_2\\
			0 & \pi_1 & \pi_2 & -(\pi_1+\pi_2)\\
		\end{array}\right].
	\end{aligned}
\end{equation*}Let $\bar{W} = [\bar{W}_{1}, \bar{W}_{2}]^{\top}$ be an $\mathbb{R}^2$-valued standard Brownian motion, and $v = [v_{1}, v_{2}]^{\top}$ denote the control process. Following \cite[Section 3.1]{Oksendal}, we replace  $\bar{\mathbb{W}}_{1}(t)\mathrm{d}t$ and $\bar{\mathbb{W}}_{2}(t)\mathrm{d}t$ by the stochastic differentials $\mathrm{d}\bar{W}_{1}(t)$ and $\mathrm{d}\bar{W}_{2}(t)$, respectively. Suppose that the initial value   $\xi=[1,1,1,1]^{\top}$. Under this formulation, $X^{v}_t$ defined in (\ref{equ:sta of ele mac}) satisfies 
\begin{equation*}
	\label{equ: stat equ of ele}
	\left\{\begin{aligned}
		\mathrm{d}X^v_t=&\ \left[A_{(\theta_t,t)}X^v_t+B_{(\theta_t,t)}v_t\right]\mathrm{d}t+\bar{C}_{(\theta_t,t)}\mathrm{d}\bar{W}_t,\\
		X^v_0=&\ \xi,
	\end{aligned}\right.
\end{equation*}where
\begin{equation}\label{equ: coe of the coup}
	\renewcommand{\arraystretch}{0.8} 
	\begin{aligned}
		A_{(1,t)}=&\  \left[\begin{array}{cccc}
			0 & 0 & 1 & 0\\	
			0 & 0 & 0 & 1\\
			-\omega_1^2-\epsilon & -\epsilon & 0 & 0\\
			-\epsilon & -\omega_2^2-\epsilon & 0 & -2\alpha_2
		\end{array}\right],
		A_{(2,t)}= \left[\begin{array}{cccc}
			0 & 0 & 1 & 0\\	
			0 & 0 & 0 & 1\\
			-\omega_1^2-\epsilon & \epsilon & 0 & 0\\
			\epsilon & -\omega_2^2-\epsilon & 0 & -2\alpha_2
		\end{array}\right],\\
		A_{(3,t)}=&\ \left[\begin{array}{cccc}
			0 & 0 & 1 & 0\\	
			0 & 0 & 0 & 1\\
			-\omega_1^2+\epsilon & -\epsilon & 0 & 0\\
			-\epsilon & -\omega_2^2+\epsilon & 0 & -2\alpha_2
		\end{array}\right],
		A_{(4,t)}= \left[\begin{array}{cccc}
			0 & 0 & 1 & 0\\	
			0 & 0 & 0 & 1\\
			-\omega_1^2+\epsilon & \epsilon & 0 & 0\\
			\epsilon & -\omega_2^2+\epsilon & 0 & -2\alpha_2
		\end{array}\right],\\
	\end{aligned}
\end{equation}and, for $\ell\in\{1,2,3,4\}$,
\begin{equation*}
	\renewcommand{\arraystretch}{0.8} 
	\begin{aligned}
		B_{(\ell,t)}=&\ \left[\begin{array}{cccc}
			0 & 0 & b_1 & 0\\
			0 & 0 & 0   &b_2\\
		\end{array}\right]^{\top},
		\bar{C}_{(\ell,t)}=\ \left[\begin{array}{cccc}
			0 & 0 & \bar{C}_{1} & 0\\
			0 & 0 & 0 & \bar{C}_{2}\\
		\end{array}\right]^{\top}.\\
	\end{aligned}
\end{equation*}

In practice, the exact  parameters of the  electrical machines in operation  are not  directly observed and can only be measured by some sensors. Assume that the measurement process $Y^v_t$ is given by 
\begin{equation*}
	\label{equ: obs equ of ele}
	\left\{\begin{aligned}
		\mathrm{d}Y^v_t=&\ F_{(\theta_t,t)}X^v_t\mathrm{d}t+\mathrm{d}W_t,\\
		Y^v_0=&\ [0, 0,  0,  0]^{\top},
	\end{aligned}\right.
\end{equation*}where $F_{(\ell,t)}$ is bounded matrix function for each $\ell\in\{1,2,3,4\}$ and  measurement noise $W$ is an $\mathbb{R}^4$-valued standard Brownian motion that is independent of $\bar{W}$. 

As found in \cite{Loparo Blankenship 2003}, the coupling structure between the machines can induce instability in the system, as energy would be transferred from the machine with positive damping to the one with zero damping. To suppress such undesired behavior, we aim to minimize the following performance criterion:
\begin{equation}
	\label{equ?cost of con var}
	\begin{aligned}
		J^P\left(\xi,\ell_{0};v\right)
		= \frac{1}{2}\mathbb{E}\Bigg\{\int_{0}^{T}\Big[\left\langle v_t,R_{(\theta_t,t)}v_t\right\rangle+\lambda\sum_{m=1}^{2}\left(\mathbb{VAR}\left[\varphi_m(t)|\mathcal{F}^{\theta}_t\right] +\mathbb{VAR}\left[\dot{\varphi}_m(t)|\mathcal{F}^{\theta}_t\right]\right)\Big]\mathrm{d}t\Bigg\},
	\end{aligned}
\end{equation}where $R_{(\ell,t)}$ is a symmetric and uniformly positive definite matrix  for each $\ell\in\{1,2,3,4\}$ and $\lambda>0$. The conditional variances of $\varphi_m(t)$ and $\dot{\varphi}_m(t)$ ($m\in\{1,2\}$) are included in the cost functional to penalize random oscillations in the system, thereby promoting stability as the cost is minimized. From the definition of conditional variance,  cost functional \eqref{equ?cost of con var} can be equivalently rewritten as
\begin{equation}
	\label{equ?cost of qua form}
	\begin{aligned}
		J^P\left(\xi,\ell_{0};v\right)
		= \frac{1}{2}\mathbb{E}\Bigg\{\int_{0}^{T}\left[\left\langle v_t,R_{(\theta_t,t)}v_t\right\rangle+\left\langle X^v_t,H_{(\theta_t,t)}X^v_t\right\rangle +\left\langle \mathbb{E}^{\theta}_t[X^v_t],\widehat{H}_{(\theta_t,t)}\mathbb{E}^{\theta}_t[X^v_t]\right\rangle\right]\mathrm{d}t\Bigg\},
	\end{aligned}
\end{equation}where $H_{(\ell,t)}=\lambda\mathcal{I}_{4\times 4}$ and $\widehat{H}_{(\ell,t)}=-H_{(\ell,t)}$ for each $\ell\in\{1,2,3,4\}$.

By applying Theorem~\ref{the:opti cont of CMF-LQG}, we can obtain the optimal  mechanical power control  $u^*$ for the coupled electrical machines system. To demonstrate the effectiveness of the proposed result in an intuitive manner, we perform a numerical simulation.  For the parameters that also appear in \cite{Loparo Blankenship 2003}, we adopt the same system data as in their numerical example: $\epsilon=0.1,\omega_1=1,\omega_2=10,  \alpha_2=0.05, \pi_1=2.5,  \pi_2=0.5.$ The remaining parameters in our model  are specified as $b_1=b_2=0.5,  \bar{C}_1=\bar{C}_2=0.5,\lambda=0.2, F_{(\ell,t)}= 2\mathcal{I}_{4\times 4},
R_{(\ell,t)}=2\mathcal{I}_{2\times 2}, \ell\in\{1,2,3,4\}$. Let terminal time $T=20$. We obtain the following numerical simulations  using the Euler-Maruyama method. In Fig.~\ref{fig:coupled electrical optimal state}, we illustrate the trajectories of the optimal angle deviations $\varphi^*_1(t)$ and  $\varphi^*_2(t)$, the optimal  angular velocities $\dot{\varphi}^*_1(t)$ and $\dot{\varphi}^*_2(t)$, the optimal mechanical power  control $u^*_1(t)$  and $u^*_2(t)$ together with the mode evolution of $\theta_t$.

% \begin{figure}[htbp!]
	%    \centering
	%      \includegraphics[width=3.3in]{coupled electrical optimal state.pdf}
	%    \caption{Trajectories of optimal  angle deviation $\varphi^*_1(t)$ and $\varphi^*_2(t)$, optimal  angular velocity $\dot{\varphi}^*_1(t)$ and  $\dot{\varphi}^*_2(t)$, optimal mechanical power control $u^*_1(t)$ and  $u^*_2(t)$, and  mode of $\theta_t$}
	%    \label{fig:coupled electrical optimal state}
	%  \end{figure}

\begin{figure}[htbp!]
	\centering
	\subfloat[Trajectories of optimal  angle deviation $\varphi^*_1(t)$ and $\varphi^*_2(t)$.]{%
		\includegraphics[width=0.7\textwidth]{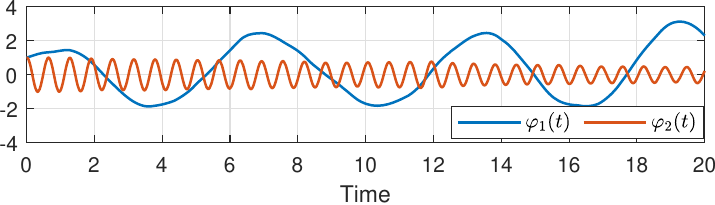}%
		\label{fig:image1}
	}
	\hspace{0.5cm} % ???????????
	\subfloat[Trajectories of optimal  angular velocity $\dot{\varphi}^*_1(t)$ and  $\dot{\varphi}^*_2(t)$.]{%
		\includegraphics[width=0.7\textwidth]{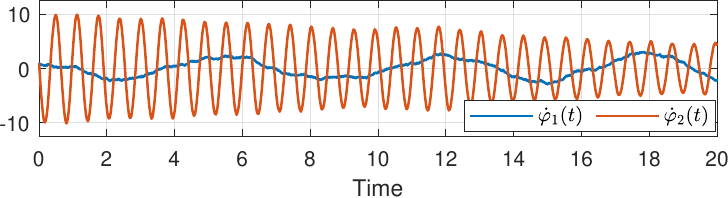}%
		\label{fig:image2}
	}
	
	\subfloat[Trajectories of optimal mechanical power control $u^*_1(t)$ and  $u^*_2(t)$.]{%
		\includegraphics[width=0.7\textwidth]{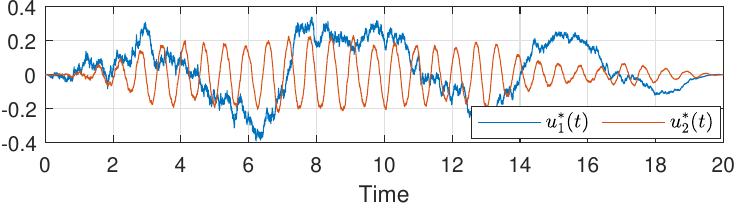}%
		\label{fig:image3}
	}
	\hspace{0.5cm}
	\subfloat[Sample path of Markov chain.]{%
		\includegraphics[width=0.7\textwidth]{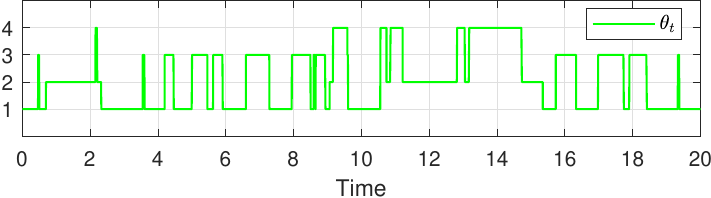}%
		\label{fig:image4}
	}
	\caption{Simulations for coupled electrical machines.}
	\label{fig:coupled electrical optimal state}
\end{figure}

To further demonstrate the effectiveness of the proposed method for the coupled electrical machine control problem, we conduct two comparative studies.

First, we investigate the role of the conditional expectation term in the cost functional. Specifically, we consider a benchmark case in which the conditional expectation weighting matrix is removed, namely, $\widehat{H}_{(\ell,t)}=0$ in \eqref{equ?cost of qua form} for all $\ell\in\{1,2,3,4\}$ and $t\in[0,T]$. We generate $50$ sample trajectories of the angular velocity $\dot{\varphi}_1^*(t)$ under the optimal controls corresponding to the original cost functional \eqref{equ?cost of qua form} and the benchmark case. The resulting trajectories are displayed in Fig.~\ref{fig:comparison between angular velocity}\subref{fig:(a)state with optimal control} and Fig.~\ref{fig:comparison between angular velocity}\subref{fig:(b)state without optimal control}, respectively. It can be observed that the trajectories in Fig.~\ref{fig:comparison between angular velocity}\subref{fig:(a)state with optimal control} remain highly synchronized throughout the entire time horizon and are concentrated within a relatively narrow fluctuation range. In contrast, the trajectories in Fig.~\ref{fig:comparison between angular velocity}\subref{fig:(b)state without optimal control} gradually diverge in the later stage, exhibiting noticeable phase mismatch and trajectory dispersion. These observations indicate that the inclusion of the conditional expectation term promotes coordination among the electrical machines, leading to improved synchronization and tighter state concentration. The results clearly demonstrate the necessity and effectiveness of incorporating the conditional  expectation term into the cost functional.

\begin{figure}[htbp!]
	\centering
	\subfloat[50 optimal trajectories of angular velocity $\dot{\varphi}^*_1(t)$ under the cost \eqref{equ?cost of qua form}.]{%
		\label{fig:(a)state with optimal control}\resizebox*{7cm}{!}{\includegraphics[width=2cm,height=1.5cm]{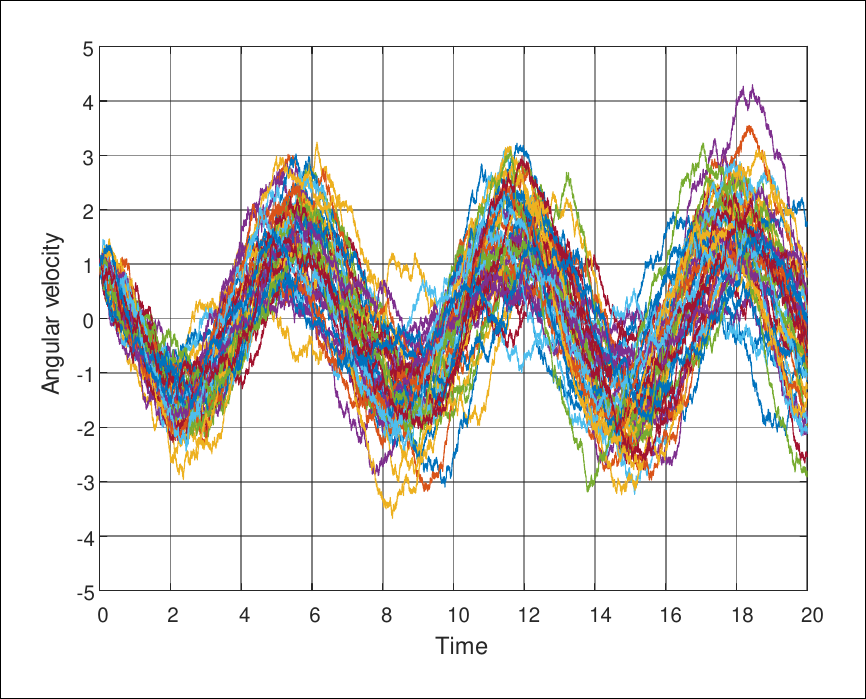}}}\hspace{0cm}
	\subfloat[50 optimal trajectories of angular velocity $\dot{\varphi}_1(t)$ under the cost \eqref{equ?cost of qua form} with $\widehat{H}_{(\ell,t)}=0$.]{%
		\label{fig:(b)state without optimal control}\resizebox*{7cm}{!}{\includegraphics[width=2cm,height=1.5cm]{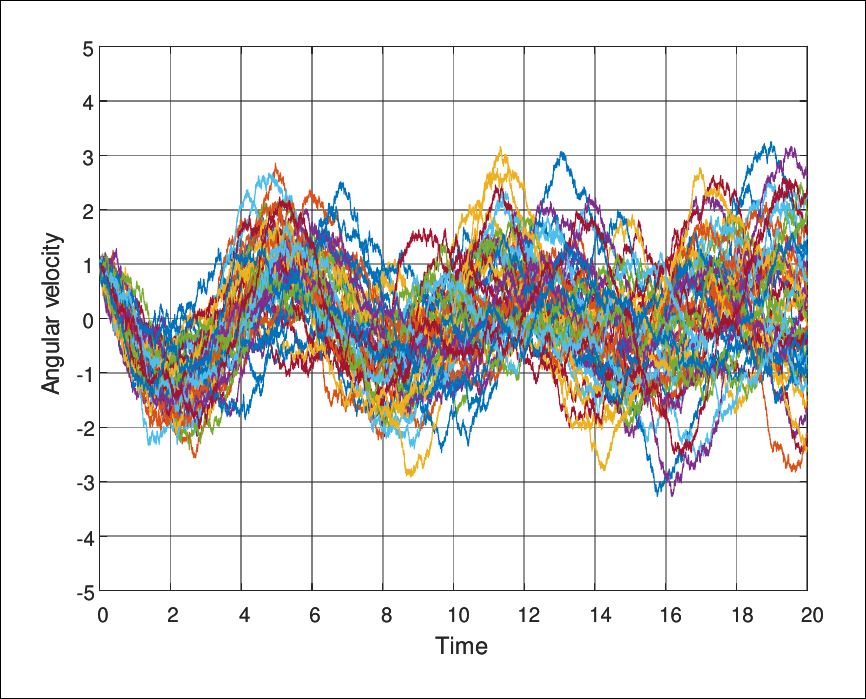}}}\hspace{0cm}
	\caption{Comparison between the optimal angular velocity.}
	\label{fig:comparison between angular velocity}
\end{figure}

Second, we compare the proposed filtering-based control framework with the filtering approach for Markov-modulated systems presented in Section 6.2.2 of \cite{Costa Fragoso Todorov 2012}. According to the separation principle established in Theorem \ref{the:opti cont of CMF-LQG}, the optimal control law is obtained by first estimating the unobserved state and then substituting the resulting estimate into the corresponding full-information control law. Consequently, the quality of the state estimate plays a crucial role in determining the overall control performance. The filtering method developed in Theorem  \ref{the:filter of X} exploits not only the current regime but also the historical path information of the Markov chain. In contrast, the filtering scheme in \cite{Costa Fragoso Todorov 2012} computes the conditional covariance primarily based on the current regime and the statistical characteristics of the Markov chain, without fully utilizing the available sample-path information. Moreover, the approach in \cite{Costa Fragoso Todorov 2012} relies on the assumption that the state and observation noises are independent (see line 15 on page 98 of \cite{Costa Fragoso Todorov 2012}), whereas our proposed filter remains applicable even when these noise processes are correlated. To evaluate the impact of the filtering scheme on the resulting control performance, we compare the sample-average costs achieved by the two filtering-based controllers. Numerical experiments are conducted for coupling strengths $\epsilon$ ranging from $0$ to $6$ with a step size of $0.2$. As can be seen from \eqref{equ: coe of the coup}, the parameter $\epsilon$ determines the degree of variation in the system dynamics across different regimes of the Markov chain. For each value of $\epsilon$, 100 independent sample paths are generated, and the corresponding optimal costs are recorded. The results are presented in Fig.~\ref{fig:comparison_with_Costa}. When the coupling strength is small, the two approaches exhibit comparable performance. However, as the coupling strength increases, the control strategy constructed from our proposed filter consistently achieves lower costs than that based on the filtering method in \cite{Costa Fragoso Todorov 2012}. The performance gap becomes increasingly pronounced as the regime-dependent heterogeneity of the system grows. This improvement can be attributed to the effective utilization of the historical evolution of the Markov chain in the filtering procedure. The results demonstrate that incorporating sample-path information can significantly enhance estimation quality and, consequently, improve control performance, particularly in strongly coupled power systems.
\begin{figure}[h]
	\centering
	\includegraphics[width=0.6\textwidth]{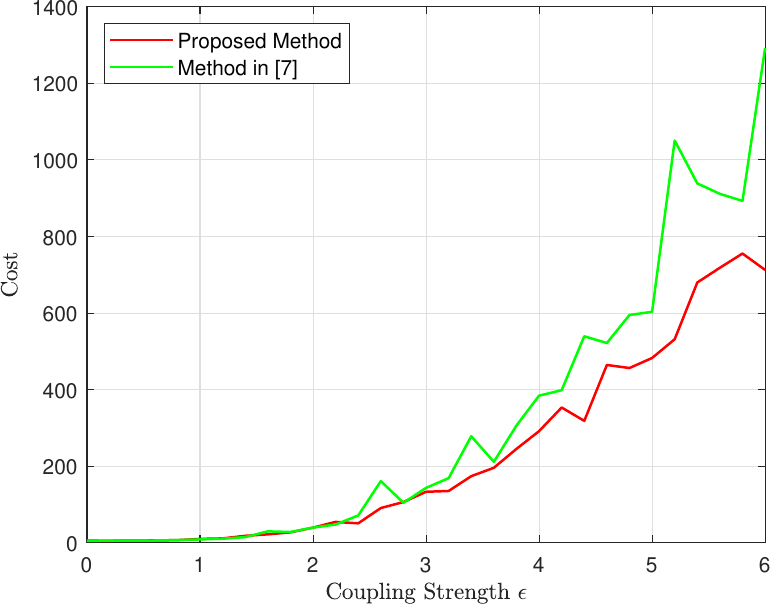}
	\caption{Comparison of the sample-average realized costs obtained from our proposed filtering-based control framework and the benchmark filtering approach in \cite{Costa Fragoso Todorov 2012} for varying coupling strengths $\epsilon$. }
	\label{fig:comparison_with_Costa}
\end{figure}

\section{Conclusion}
\label{sec: conclusion}
In this paper, we investigate a conditional mean-field type LQ  optimal control problem driven by a regime-switching diffusion  under partial observation. In the proposed model, the conditional expectations of the state and control explicitly appear in both the system dynamics and the cost functional.  The true state cannot be fully observed.  We establish a separation principle for the problem through the following steps. First, we introduce a modified definition of the admissible controls set, which eliminates the coupled dependence between the control and observation processes. Second, we derive a closed system of  filtering equations by exploiting the conditionally Gaussian property of the state. Then, we reformulate the original problem into an equivalent fully observed control problem driven by the filtering process. Using the stochastic maximum principle and a decoupling technique, we obtain the optimal control for the equivalent problem. Finally,   the solution of the original problem follows. These results extend the classical separation principle to the class of conditional mean-field systems with regime switching.

There are two interesting directions for future research. The first one is to explore the separation principle for discrete-time conditional mean-field systems, which remains challenging due to the invalidity of stochastic analysis tools such as It$\widehat{\text{o}}$'s formula. The second one is to investigate conditional mean-field problems in which both the exact values of  Markov chain and  system state are unavailable.

\appendix 
\section{Proof of Lemma \ref{lem:equi bet con on two sets}} \label{sec: pro appe}
\begin{proof}Since  $\mathcal{U}_{ad}\subseteq\mathcal{U}$, it is immediate that
	\begin{equation*}
		\begin{aligned}
			\inf_{v\in\mathcal{U}}J^{P}\left(\xi,\ell_0; v\right)\leq\inf_{v\in\mathcal{U}_{ad}}J^{P}\left(\xi,\ell_0; v\right).
		\end{aligned}
	\end{equation*}Therefore, it remains to establish the reverse inequality. The proof is divided into the following four steps.

	\textbf{Step 1:} Sensitivity estimate.

	By Gronwall's inequality and   Jensen's inequality, there exists a constant $K>0$ such that,  for any   $v,v^{\prime} \in\mathcal{U}$, 
	\begin{equation}\label{equ: sen est}
		\begin{aligned}
			\mathbb{E}\left[\sup_{t\in[0,T]}\left|X^{v}_t-X^{v^{\prime}}_t\right|^2\right]\leq K\mathbb{E}\left[\int_{0}^{T}\left|v_t-v^{\prime}_t\right|^2\mathrm{d}t\right].
		\end{aligned}
	\end{equation}  
	
	\textbf{Step 2:} Density  of $\mathcal{U}_{ad}$ in $\mathcal{U}$.

	Let $v\in\mathcal{U}$. For each $\ell\ge2$, define the process $v^\ell$ by
	\begin{equation}\label{equ: def of seque}
		\begin{aligned}
			v^{\ell}_t=\left\{\begin{array}{cc}
				v_0, & \hspace{-1.25cm} t\in[0,\Delta^{\ell}],\\
				\frac{1}{\Delta^{\ell}}\int_{(i-1)\Delta^{\ell}}^{i\Delta^{\ell}}v_s\mathrm{d}s, &\hspace{0.2cm} t\in(i\Delta^{\ell},(i+1)\Delta^{\ell}], 
			\end{array}\right.
		\end{aligned}
	\end{equation}where $\Delta^\ell=T/\ell$ and $i\in\{1,\ldots,\ell-1\}$. We show that  $v^\ell\in\mathcal U_{ad}$ by establishing the equality $\mathcal{F}_t^{Y^{v^\ell},\theta} = \mathcal{F}_t^{\lambda,\theta}$ for all $t\in[0,T]$.  
	
	For $t \in [0,\Delta^\ell]$, since $\lambda_0 = 0$ is deterministic, $v_0$ is $\mathcal{F}_0^\theta$-measurable, and thus $v_t^\ell$ is $\mathcal{F}_0^\theta$-measurable. By \eqref{equ:x1} and \eqref{equ:Y1}, with all coefficients being Markov-modulated, $Y_t^{1,v^\ell}$ is $\mathcal{F}_t^\theta$-adapted. Combining this with the second identity in \eqref{equ:sum rela} and the inclusions $\mathcal{F}_t^\theta \subseteq \mathcal{F}_t^{Y^{v^\ell},\theta}$ and $\mathcal{F}_t^\theta \subseteq \mathcal{F}_t^{\lambda,\theta}$, we obtain
	\begin{equation*}
		\begin{aligned}
			\mathcal{F}_t^{Y^{v^\ell},\theta} = \mathcal{F}_t^{\lambda,\theta} \quad \text{for} \quad t\in [0,\Delta^\ell].
		\end{aligned}
	\end{equation*}

	Assume inductively that $\mathcal{F}_t^{Y^{v^\ell},\theta} = \mathcal{F}_t^{\lambda,\theta}$ holds on $[0,k\Delta^\ell]$ for some $k < \ell$. For $t \in (k\Delta^\ell,(k+1)\Delta^\ell]$, $v_t^\ell$ is measurable with respect to $\mathcal{F}_{k\Delta^\ell}^{\lambda,\theta} = \mathcal{F}_{k\Delta^\ell}^{Y^{v^\ell},\theta}$. It follows that $Y_t^{1,v^\ell}$ is adapted to both $\mathcal{F}_{k\Delta^\ell}^{\lambda} \vee \mathcal{F}_t^\theta$ and $\mathcal{F}_{k\Delta^\ell}^{Y^{v^\ell}} \vee \mathcal{F}_t^\theta$. Invoking the second identity in \eqref{equ:sum rela}, we deduce that $\mathcal{F}_t^{Y^{v^\ell},\theta} = \mathcal{F}_t^{\lambda,\theta}$ on $(k\Delta^\ell,(k+1)\Delta^\ell]$. By induction over $k = 0,1,\dots,\ell-1$, this equality holds for all $t\in[0,T]$, so that $v^\ell \in \mathcal{U}_{ad}$.
	
	Moreover, $v^\ell$  satisfies $\mathbb{E}[\int_{0}^{T}|v^{\ell}_t-v_t|^2\mathrm{d}t]\rightarrow0$ as $\ell\to\infty$. Therefore, $\mathcal{U}_{ad}$ is dense in $\mathcal{U}$  with respect to the $L^2$-norm.

	\textbf{Step 3:} Convergence between $J^{P}\left(\xi,\ell_0; v^{\ell}\right)$ and $J^{P}\left(\xi,\ell_0; v\right)$. 
	
	Let $v \in \mathcal{U}$ and let $v^\ell\in \mathcal{U}_{ad}$ for every $\ell \ge 2$ be a sequence approximating $v$ as constructed previously. By Jensen's inequality, H$\ddot{\text{o}}$lder's inequality, and Assumption \ref{ass: coe in cost one}, there exists a constant $K>0$ such that
	\begin{equation}\label{equ: erro bet cost}
		\begin{aligned}
			&\ 2\left|J^{P}\left(\xi,\ell_0; v^{\ell}\right)-J^{P}\left(\xi,\ell_0; v\right)\right|\\
			\leq &\ K\left\{\left|X^{v}-X^{v^{\ell}}\right|_{L^2} \left(\left|v-v^{\ell}\right|_{L_2}+\left|X^{v}+X^{v^{\ell}}\right|_{L^2} +\left|v^{\ell}\right|_{L^2}\right) +\left|v-v^{\ell}\right|_{L^2}\left(\left|v+v^{\ell}\right|_{L^2}+\left|X^{v^{\ell}}\right|_{L^2}\right)\right.\\
			&\ \left. +\left(\mathbb{E}\left|X_T^{v}-X^{v^{\ell}}_T\right|^2\right)^{\frac{1}{2}}\left(\mathbb{E}\left|X_T^{v}+X^{v^{\ell}}_T\right|^2\right)^{\frac{1}{2}}\right\},
		\end{aligned}
	\end{equation}where $\left|\phi\right|_{L^2}=\left(\mathbb{E}[\int_{0}^{T}\left|\phi_t\right|^2\mathrm{d}t]\right)^{\frac{1}{2}}$ denotes the $L_2$ norm  of the process $\phi$.

	Under Assumption \ref{ass:coef in dyna}, it follows from Lemma 2.1 in \cite{Nguyen 2020} that, for any $v^\ell \in \mathcal{U}_{ad}$, equation \eqref{equ:state} admits a unique solution $X^{v^\ell}$ satisfying $\mathbb{E}[\sup_{t\in[0,T]}|X_t^{v^\ell}|^2] < \infty$. Combining \eqref{equ: erro bet cost} with the sensitivity estimate \eqref{equ: sen est}, we obtain
	\begin{equation}\label{equ: conv of cost}
		J^P(\xi,\ell_0;v^\ell) \to J^P(\xi,\ell_0;v), \quad \text{as } \ell\to\infty.
	\end{equation}
	
	\textbf{Step 4:} $\inf_{v\in\mathcal{U}_{ad}} J^P(\xi,\ell_0;v) \leq \inf_{v\in\mathcal{U}} J^P(\xi,\ell_0;v)$.

	Since for any $v \in \mathcal{U}$ there exists a sequence $v^\ell \in \mathcal{U}_{ad}$ satisfying \eqref{equ: conv of cost}, and for any $v^\ell \in \mathcal{U}_{ad}$ we have
	\begin{equation*}
		\begin{aligned}
			\inf_{\widetilde{v}\in\mathcal{U}_{ad}} J^P(\xi,\ell_0;\widetilde{v}) \leq J^P(\xi,\ell_0;v^\ell),
		\end{aligned}
	\end{equation*}letting $\ell \to \infty$ yields
	\begin{equation*}
		\begin{aligned}
			\inf_{\widetilde{v}\in\mathcal{U}_{ad}} J^P(\xi,\ell_0;\widetilde{v}) \leq J^P(\xi,\ell_0;v),
		\end{aligned}
	\end{equation*}for any $v \in \mathcal{U}$. Taking the infimum over $v \in \mathcal{U}$ gives
	\begin{equation*}
		\begin{aligned}
			\inf_{\widetilde{v}\in\mathcal{U}_{ad}} J^P(\xi,\ell_0;\widetilde{v}) \leq \inf_{v\in\mathcal{U}} J^P(\xi,\ell_0;v),
		\end{aligned}
	\end{equation*}which establishes the desired result.
\end{proof}


\begin{thebibliography}{99}
	\bibitem{Anderson 2011}
	Andersson, D., \& Djehiche, B. (2011). A maximum principle for SDEs of mean-field type. \textit{Applied Mathematics and Optimization}, 63(3), 341--356.
	
	\bibitem{Bain Crisan 2009}
	Bain, A., \& Crisan, D. (2009). \textit{Fundamentals of Stochastic Filtering}. Springer, New York, NY, USA.
	
	\bibitem{Barreiro-Gomez linear 2019}
	Barreiro-Gomez, J., Duncan, T. E., \& Tembine, H. (2019). Linear-quadratic mean-field-type games: Jump-diffusion process with regime switching. \textit{IEEE Transactions on Automatic Control}, 64(10), 4329--4336.
	
	
	
	\bibitem{Bensoussan 1992}
	Bensoussan, A. (1992). \textit{Stochastic Control of Partially Observable Systems}. Cambridge University Press, Cambridge, U.K.
	
	\bibitem{Bjork  Tomas Agatha 2017}
	Bj\"ork, T., Khapko, M., \& Murgoci, A. (2017). On time-inconsistent stochastic control in continuous time. \textit{Finance and Stochastics}, 21, 331--360.
	
	\bibitem{Costa discrete 1995}
	Costa, O. L. V., \& Fragoso, M. D. (1995). Discrete-time LQ-optimal control problems for infinite Markov jump parameter systems. \textit{IEEE Transactions on Automatic Control}, 40(12), 2076--2088.
	
	\bibitem{Costa Fragoso Todorov 2012}
	Costa, O. L. V., Fragoso, M. D., \& Todorov, M. G. (2012). \textit{Continuous-Time Markov Jump Linear Systems}. Springer, Berlin, Germany.
	
	\bibitem{Cohen  Elliott 2008}
	Cohen, S. N., \& Elliott, R. J. (2008). Solutions of backward stochastic differential equations on Markov chains. \textit{Communications in Stochastic Analysis}, 2(2), Art. 5.
	
	\bibitem{Dawson 1983}
	Dawson, D. A. (1983). Critical dynamics and fluctuations for a mean-field model of cooperative behavior. \textit{Journal of Statistical Physics}, 31(1), 29--85.
	
	\bibitem{Donnelly 2008}
	Donnelly, C. (2008). Convex duality in constrained mean-variance portfolio optimization under a regime-switching model. Doctoral dissertation, University of Waterloo.
	
	\bibitem{Dufour adaptive 1998}
	Dufour, F., \& Elliott, R. J. (1998). Adaptive control of linear systems with Markov perturbations. \textit{IEEE Transactions on Automatic Control}, 43(3), 351--372.
	
	\bibitem{Fujisaki Kallianpur Kunita 1972}
	Fujisaki, M., Kallianpur, G., \& Kunita, H. (1972). Stochastic differential equations for the nonlinear filtering problem. \textit{Osaka Journal of Mathematics}, 9, 19--40.
	
	\bibitem{Ji 1992}
	Ji, Y., \& Chizeck, H. J. (1992). Jump linear quadratic Gaussian control in continuous time. \textit{IEEE Transactions on Automatic Control}, 37(12), 1884--1892.
	
	\bibitem{Kallianpur 1980}Kallianpur, G. (1980). \textit{Stochastic filtering theory}. Springer, New York, NY, USA.
	
	\bibitem{Li 2012}
	Li, J. (2012). Stochastic maximum principle in the mean-field controls. \textit{Automatica}, 48(2), 366--373.
	
	\bibitem{Liptser 1978 1}
	Liptser, R. S., \& Shiryayev, A. N. (1978). \textit{Statistics of Random Processes: I: General Theory}. Springer, Berlin, Germany.
	
	%\bibitem{Lipster 1978 2}
	%Liptser, R. S., \& Shiryayev, A. N. (1978). \textit{Statistics of Random Processes II: Applications}. Springer, Berlin, Germany.
	
	\bibitem{Lipster 1989}
	Liptser, R. S., \& Shiryayev, A. N. (1989). \textit{Theory of Martingales}. Kluwer Academic, Dordrecht, The Netherlands.
	
	\bibitem{Loparo Blankenship 2003}
	Loparo, K., \& Blankenship, G. (2003). A probabilistic mechanism for small disturbance instabilities in electric power systems. \textit{IEEE Transactions on Circuits and Systems}, 32(2), 177--184.
	
	\bibitem{Lv Xiong Zhang 2023}
	Lv, S., Xiong, J., \& Zhang, X. (2023). Linear quadratic leader-follower stochastic differential games for mean-field switching diffusions. \textit{Automatica}, 154, 111072.
	
	\bibitem{Lv linear 2023}
	Lv, S., Wu, Z., \& Xiong, J. (2024). Linear quadratic nonzero-sum mean-field stochastic differential games with regime switching. \textit{Applied Mathematics and Optimization}, 90(2), 44.
	
	\bibitem{Mariton1990}
	Mariton, M. (1990). \textit{Jump Linear Systems in Automatic Control}. Marcel Dekker, New York, NY, USA.
	
	\bibitem{Moon Basar 2024}
	Moon, J., \& Ba\c{s}ar, T. (2024). Separation principle for partially observed linear-quadratic optimal control for mean-field type stochastic systems. \textit{IEEE Transactions on Automatic Control}, 69(12), 8370--8385.
	
	\bibitem{Nguyen 2020}
	Nguyen, S. L., Nguyen, D. T., \& Yin, G. (2020). A stochastic maximum principle for switching diffusions using conditional mean-fields with applications to control problems. \textit{ESAIM: Control, Optimisation and Calculus of Variations}, 26, 69.
	
	\bibitem{Nguyen George Yin Hoang 2020}
	Nguyen, S. L., Yin, G., \& Hoang, T. A. (2020). On laws of large numbers for systems with mean-field interactions and Markovian switching. \textit{Stochastic Processes and Their Applications}, 130(1), 262--296.
	
	\bibitem{Nguyen general 2021}
	Nguyen, S. L., Yin, G., \& Nguyen, D. T. (2021). A general stochastic maximum principle for mean-field controls with regime switching. \textit{Applied Mathematics and Optimization}, 84(3), 3255--3294.
	
	\bibitem{Oksendal}
	\O ksendal, B. (2013). \textit{Stochastic Differential Equations: An Introduction with Applications}. Springer, Berlin, Germany.
	
	\bibitem{Protter 2005}
	Protter, P. E. (2005). \textit{Stochastic Integration and Differential Equations}. Springer, Berlin, Germany.
	
	\bibitem{Rolón Gutiérrez markovian 2024}
	Rol\'on~Guti\'errez, E. J., Nguyen, S. L., \& Yin, G. (2024). Markovian-switching systems: Backward and forward-backward stochastic differential equations, mean-field interactions, and nonzero-sum differential games. \textit{Applied Mathematics and Optimization}, 89(2), 33.
	
	\bibitem{Rogers 2000}
	Rogers, L. C. G., \& Williams, D. (2000). \textit{Diffusions, Markov Processes and Martingales: Vol. 2, It\^{o} Calculus}. Cambridge University Press, Cambridge, U.K.
	
	\bibitem{Sznitman 1989}
	Sznitman, A. S. (1991). \textit{Topics in propagation of chaos}, in Ecole d'{\'E}t\'e de Probabilit{\'e}s de Saint-Flour XIX-1989. Springer, Berlin, Germany.
	
	\bibitem{Sun 2017}
	Sun, J. (2017). Mean-field stochastic linear-quadratic optimal control problems: Open-loop solvabilities. \textit{ESAIM: Control, Optimisation and Calculus of Variations}, 23(3), 1099--1127.
	
	\bibitem{Sun Xiong 2023}
	Sun, J., \& Xiong, J. (2023). Stochastic linear-quadratic optimal control with partial observation. \textit{SIAM Journal on Control and Optimization}, 61(3), 1231--1247.
	
	\bibitem{Sworder 1983}
	Sworder, D., \& Rogers, R. (1983). An LQ-solution to a control problem associated with a solar thermal central receiver. \textit{IEEE Transactions on Automatic Control}, 28(10), 971--978.
	
	\bibitem{Wang 2022}
	Wang, G., \& Wu, Z. (2022). A maximum principle for mean-field stochastic control system with noisy observation. \textit{Automatica}, 137, 110135.
	
	\bibitem{Williams 1991}
	Williams, D. (1991). \textit{Probability with Martingales}. Cambridge University Press, Cambridge, U.K.
	
	\bibitem{Wonham 1968}
	Wonham, W. M. (1968). On the separation theorem of stochastic control. \textit{SIAM Journal on Control}, 6(2), 312--326.
	
	
	
	\bibitem{Xiong mean 2007}
	Xiong, J., \& Zhou, X. Y. (2007). Mean-variance portfolio selection under partial information. \textit{SIAM Journal on Control and Optimization}, 46(1), 156--175.
	
	\bibitem{Yao Zhang Zhou 2006}
	Yao, D. D., Zhang, Q., \& Zhou, X. Y. (2006). A regime-switching model for European options, in \textit{Stochastic Processes, Optimization, and Control Theory: Applications in Financial Engineering, Queueing Networks, and Manufacturing Systems}, 281--300.
	
	\bibitem{Yin hybrid 2009}
	Yin, G., \& Zhu, C. (2009). \textit{Hybrid Switching Diffusions: Properties and Applications}. Springer, Berlin, Germany.
	
	\bibitem{Yong 2013}
	Yong, J. (2013). Linear-quadratic optimal control problems for mean-field stochastic differential equations. \textit{SIAM Journal on Control and Optimization}, 51(4), 2809--2838.
	
	\bibitem{Yong Zhou 1999}
	Yong, J., \& Zhou, X. Y. (1999). \textit{Stochastic Controls: Hamiltonian Systems and HJB Equations}. Springer, Berlin, Germany.
	
	\bibitem{Zhang Yin 1999}
	Zhang, Q., \& Yin, G.  (1999). On nearly optimal controls of hybrid LQG problems. \textit{IEEE Transactions on Automatic Control}, 44(12), 2271--2282.
	
	\bibitem{Zhang general 2018}
	Zhang, X., Sun, Z., \& Xiong, J. (2018). A general stochastic maximum principle for a Markov regime-switching jump-diffusion model of mean-field type. \textit{SIAM Journal on Control and Optimization}, 56(4), 2563--2592.
	
	\bibitem{Zhou 2003}
	Zhou, X. Y., \& Yin, G. (2003). Markowitz's mean-variance portfolio selection with regime switching: A continuous-time model. \textit{SIAM Journal on Control and Optimization}, 42(4), 1466--1482.
	
\end{thebibliography}
\end{document}